\documentclass[a4paper, titlepage, oneside, 11pt]{amsart}
\usepackage{amsmath,amssymb,amsfonts,amstext,amscd,cite,geometry,graphicx,latexsym,mathptmx,xcolor,pstricks}
\usepackage[utf8]{inputenc}
\usepackage{newunicodechar}
\usepackage[all]{xy}

\newtheorem{lem}{Lemma}[section]
\newtheorem{prop}[lem]{Proposition}
\newtheorem{thm}[lem]{Theorem}
\newtheorem{cor}[lem]{Corollary}

\newtheorem{rem}[lem]{Remark}

\newcommand{\vertiii}[1]{{\left\vert\kern-0.25ex\left\vert\kern-0.25ex\left\vert #1
    \right\vert\kern-0.25ex\right\vert\kern-0.25ex\right\vert}}

\numberwithin{equation}{section}

\allowdisplaybreaks

\begin{document}

\title[Dynamics of the box-ball system with random initial conditions]{Dynamics of the box-ball system\\with random initial conditions\\via Pitman's transformation}

\author[D.~A.~Croydon]{David A. Croydon}
\address{Department of Advanced Mathematical Sciences, Graduate School of Informatics, Kyoto University, Sakyo-ku, Kyoto 606--8501, Japan}
\email{croydon@acs.i.kyoto-u.ac.jp}

\author[T.~Kato]{Tsuyoshi Kato}
\address{Department of Mathematics, Graduate School of Science, Kyoto University, Sakyo-ku, Kyoto 606--8502, Japan}
\email{tkato@math.kyoto-u.ac.jp}

\author[M.~Sasada]{Makiko Sasada}
\address{Graduate School of Mathematical Sciences, University of Tokyo, 3-8-1, Komaba, Meguro-ku, Tokyo, 153--8914, Japan}
\email{sasada@ms.u-tokyo.ac.jp}

\author[S.~Tsujimoto]{Satoshi Tsujimoto}
\address{Department of Applied Mathematics and Physics, Graduate School of Informatics, Kyoto University, Sakyo-ku, Kyoto 606--8501, Japan}
\email{tujimoto@i.kyoto-u.ac.jp}

\begin{abstract}
The box-ball system (BBS), introduced by Takahashi and Satsuma in 1990, is a cellular automaton that exhibits solitonic behaviour. In this article, we study the BBS when started from a random two-sided infinite particle configuration. For such a model, Ferrari et al.\ recently showed the invariance in distribution of Bernoulli product measures with density strictly less than $\frac{1}{2}$, and gave a soliton decomposition for invariant measures more generally. We study the BBS dynamics using the transformation of a nearest neighbour path encoding of the particle configuration given by `reflection in the past maximum', which was famously shown by Pitman to connect Brownian motion and a three-dimensional Bessel process. We use this to characterise the set of configurations for which the dynamics are well-defined and reversible for all times. We give simple sufficient conditions for random initial conditions to be invariant in distribution under the BBS dynamics, which we check in several natural examples, and also investigate the ergodicity of the relevant transformation. Furthermore, we analyse various probabilistic properties of the BBS that are commonly studied for interacting particle systems, such as the asymptotic behavior of the integrated current of particles and of a tagged particle. Finally, for Bernoulli product measures with parameter $p\uparrow\frac12$ (which may be considered the `high density' regime), the path encoding we consider has a natural scaling limit, which motivates the introduction of a new continuous version of the BBS that we believe will be of independent interest as a dynamical system.
\end{abstract}



\date{\today}

\maketitle

\setcounter{tocdepth}{3}
\tableofcontents
\newpage

\section{Introduction}

In 1990, Takahashi and Satsuma introduced a simple cellular automaton in which states could be decomposed into solitary waves that interact in the same manner as solitons \cite{takahashi1990}. This model has since been named the box-ball system (BBS), and it has been widely studied from an integrable systems viewpoint (see \cite{T, TT} for introductory surveys of the mathematics of BBSs, and \cite{IKT} for a review of some of the connections with integrable structures). In particular, strong links have been established between the BBS and the well-known Korteweg--de Vries (KdV) equation \cite{KdV}:
\begin{eqnarray*}
 \frac{\partial u}{\partial {t}}+ 6u\frac{\partial u}{\partial {x}}+ \frac{\partial^3 u}{\partial {x}^3}= 0,
\end{eqnarray*}
where $u=(u(x,t))_{x,t\in\mathbb{R}}$, which has been used to model shallow water waves, and is a fundamental example of an integrable system with an infinite number of degrees of freedom. As with many integrable systems, computational interests have motivated the introduction of a discretised version, in this case the discrete KdV equation:
\begin{equation}
\dfrac{1}{u_{n+1}^{(t+1)}} - \dfrac{1}{u_n^{(t)}}
= \delta\left(u_{n+1}^{(t)} - u_{n}^{(t+1)}\right),\label{discretekdv}
\end{equation}
where $n,t\in\mathbb{Z}$, $\delta\in\mathbb{R}$. The latter equation has been explored extensively as a rational dynamical system \cite{hirota}, and can be seen as the stepping stone between the KdV equation and the BBS. Indeed, it has recently been shown that by applying a dynamical scale transform called ultra-discretisation \cite{TTMS}, or Maslov de-quantisation \cite{Kbook, LM}, it is possible to transform such systems into automata whose governing equations incorporate tropical geometry. From the discrete KdV equation, for instance, such a procedure yields the so-called ultra-discrete KdV equation (see \eqref{origdef} below) that defines the BBS \cite{HT:ukdv, TTMS}. Thus, despite its simple definition as an automaton, the BBS holds considerable interest as a core dynamical system in mathematical physics.

The aim of this article is to explore the dynamics of the BBS when the initial conditions are random. As a central part of our study, we explain how the evolution of the BBS precisely corresponds to the operation of ‘reflection in the past maximum’ of a certain path encoding of the particle configuration, where we note the latter transformation was famously shown by Pitman to link Brownian motion and a three-dimensional Bessel process \cite{Pitman}. This allows us to extend the dynamics in a systematic way to two-sided infinite configurations, and connect the microscopic particle system with a macroscopic picture via a scaling limit. The applicability of Pitman's transformation to queuing systems and links with random polymers and certain integrable systems are now well-established in the probability literature (for example, see \cite{HW,OC,OCslides,OCY} and the references therein). One of the contributions of this article is to show that Pitman's transformation also provides a useful tool for analysing various properties of the BBS as a dynamical system, including reversibility, invariant measures and ergodicity. Moreover, it allows us to analyse properties of the BBS that are commonly studied for interacting particle systems, such as the asymptotic behavior of the integrated current of particles and of a tagged particle, as well as scaling limits.  In short, the article is at the intersection of three areas, bringing probabilistic techniques to shed new light on an important dynamical/integrable system.

Let us start by presenting Takahashi and Satsuma’s original definition of the BBS from \cite{takahashi1990}. First, we will denote by $(\eta_n)_{n\in\mathbb{Z}}\in \{0,1\}^{\mathbb{Z}}$ a particle configuration. Specifically, we write $\eta_n = 1$ if there is a particle at $n$, and $\eta_n = 0$ otherwise. For the moment, as in \cite{takahashi1990}, we suppose there is a finite number of particles, that is, $\sum_{n\in\mathbb{Z}}\eta_n<\infty$. Without loss of generality, we can further assume that each of these particles are sited on the positive axis, i.e.\ $\sum_{n\leq 0}\eta_n=0$. In this case, the evolution of the system is described by an operator $T: \{0,1\}^{\mathbb{Z}}\rightarrow\{0,1\}^{\mathbb{Z}}$ characterised by the so-called ultradiscrete KdV equation:
\begin{equation}\label{origdef}
(T\eta)_{n}=\min\left\{1-\eta_{n},\sum_{m=-\infty}^{n-1}\left(\eta_m - (T\eta)_m\right)\right\},
\end{equation}
where we suppose $(T\eta)_n=0$ for $n\leq 0$, so the sums in the above definition are well-defined. In words, we can view this action in terms of a particle ‘carrier’, which moves along $\mathbb{Z}$ from left to right (that is, from negative to positive), picking up a particle when it crosses one, and dropping off a particle when it is holding at least one particle and sees a space. The latter description motivates the introduction of a `carrier process' $W=(W_n)_{n\in\mathbb{Z}}$, where $W_n$ records the number of particles held by the carrier as it passes spatial location $n$. In particular, in this finite particle setting, we set $W_n=0$ for $n\leq 0$, and, for $n\geq1$,
\begin{equation}\label{wupdate}
W_n=\left\{\begin{array}{ll}
               W_{n-1}+1, & \mbox{if }\eta_n=1,\\
               W_{n-1}, & \mbox{if }\eta_n=0\mbox{ and }W_{n-1}=0,\\
               W_{n-1}-1, & \mbox{if }\eta_n=0\mbox{ and }W_{n-1}>0.
             \end{array}\right.
\end{equation}
With this, the definition of the BBS at \eqref{origdef} can be rewritten
\begin{equation}\label{updaterule}
(T\eta)_{n}=\min\left\{1-\eta_{n},W_{n-1}\right\}.
\end{equation}
We note that the dynamics of the BBS for a finite number of particles are well-defined for all time, meaning we can define $T^k\eta$ for any $k\geq 0$. Moreover, in the original paper \cite{takahashi1990}, it was observed that the dynamics are reversible, in that we can obtain $\eta$ from $T\eta$ by simply running the carrier backwards, i.e.\ from right to left\footnote{In this instance, we are using the term reversible in a dynamical systems sense. Later in the article, we will also use the term reversible in a stochastic processes sense when describing various Markov chains. Although the two meanings of reversible are distinct, how the term is meant to be interpreted should be clear from the context.}, so that, in fact, $T^k\eta$ is well-defined for any $k\in \mathbb{Z}$ (for this comment to be true, we drop the restriction that all the particles are to the right of the origin). Furthermore, Takahashi and Satsuma described how any configuration could be decomposed into a collection of `basic strings' of the form $(1,0)$, $(1,1,0,0)$, $(1,1,1,0,0,0)$, etc., which acted like solitons in that they were preserved by the action of the carrier, and travelled at a constant speed (depending on their length) when in isolation, but experienced interactions when they met. See Figure \ref{fig:bbs-2soliton} for a simple example of a two-soliton interaction in the BBS.

\newcommand{\BA}{\circle{10}}
\newcommand{\BB}{\circle*{10}}
\newcommand{\BC}{\circle{11}}
\newcommand{\BD}{\circle*{11}}
\makeatother
\begin{figure}[t]
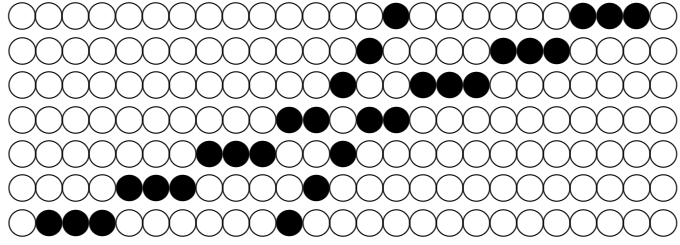

\begin{center}
\noindent
{
\hspace*{-1mm}\BA\BA\BA\BA\BA\BA\BA\BA \BA\BA\BA\BA\BA \BA\BB\BA\BA\BA \BA\BA\BA\BB\BB\BB\BA \\
\BA\BA\BA\BA\BA\BA\BA\BA \BA\BA\BA\BA\BA \BB\BA\BA\BA\BA \BB\BB\BB\BA\BA\BA \BA\\
\BA\BA\BA\BA\BA\BA\BA\BA \BA\BA\BA\BA\BB \BA\BA\BB\BB\BB \BA\BA\BA\BA\BA\BA \BA\\
\BA\BA\BA\BA\BA\BA\BA\BA \BA\BA\BB\BB\BA \BB\BB\BA\BA\BA \BA\BA\BA\BA\BA\BA\BA \\
\BA\BA\BA\BA\BA\BA\BA\BB \BB\BB\BA\BA\BB \BA\BA\BA\BA\BA \BA\BA\BA\BA\BA\BA \BA\\
\BA\BA\BA\BA\BB\BB\BB\BA \BA\BA\BA\BB\BA \BA\BA\BA\BA\BA \BA\BA\BA\BA\BA\BA \BA\\
\BA\BB\BB\BB\BA\BA\BA\BA \BA\BA\BB\BA\BA \BA\BA\BA\BA\BA \BA\BA\BA\BA\BA\BA \BA}
\end{center}
\caption{A two-soliton interaction of the box-ball system. (Time runs from the bottom row to the top row.)} \label{fig:bbs-2soliton}
\end{figure}

As noted above, the goal of this article is to study the BBS as an interacting particle system started from a random initial condition. In this setting it is natural to ask what configurations are invariant in distribution, i.e.\ when is $T\eta\buildrel{d}\over{=}\eta$? Of course, given the transience of the system (all particles move at speed at least one to the right), this question immediately necessitates the consideration of two-sided infinite particle configurations. However, whilst it is easy to extend the definitions of the previous paragraph to the case when $\eta_n=1$ infinitely often as $n\rightarrow +\infty$, the same is not true when $\eta_n=1$ infinitely often as $n\rightarrow -\infty$. Indeed, for such configurations the equation characterising the system at (\ref{origdef}) is no longer well-defined, and one needs to make sense of starting the carrier from $-\infty$. As we will discuss in more detail below, even though we can extend the definition for a certain class of configurations quite straightforwardly, the continued evolution and reversibility of the system can no longer be taken for granted.

It transpires that a convenient way to approach the issue of extending the dynamics to an infinite system is to introduce a certain path encoding of the particle configuration, and consider the dynamics of this. In particular, we define a two-sided nearest-neighbour path $S=(S_n)_{n\in\mathbb{Z}}$  by setting $S_0=0$, and
\begin{equation}\label{SRWrep}
S_{n}=S_{n-1}+1-2\eta_{n},\qquad\forall n\in\mathbb{Z},
\end{equation}
i.e.\ the increment $S_{n}-S_{n-1}$ is equal to $-1$ if there is a particle at $n$, and equal to $+1$ otherwise. Now, if $\eta_n=0$ eventually as $n\rightarrow-\infty$, then it is an elementary exercise (cf.\ Lemma \ref{tlem}) to check that the action of the carrier on $S$ is given by
\begin{equation}\label{pitmans}
(TS)_{n}=2M_n-S_n-2M_0,
\end{equation}
where we slightly abuse notation by writing $TS=((TS)_n)_{n\in\mathbb{Z}}$ for the path encoding of $T\eta$, and $M=(M_n)_{n\in\mathbb{Z}}$ is the past maximum of $S$, i.e.\
\begin{equation}\label{mdef}
M_n=\sup_{m\leq n}S_m.
\end{equation}
The mapping given by (\ref{pitmans}) is Pitman's transformation, which has been studied extensively in the stochastic processes literature, and we will discuss in Section \ref{pitmansec} how our results relate to known results in this area. We note that it is easy to understand this transformation pictorially, see Figure \ref{bbsfig} for an example realisation of $S$, $M$ and $TS$. Moreover, it is straightforward to connect the path encoding $S$ to the carrier process $W$ through the identity
\begin{equation}\label{wms}
W=M-S.
\end{equation}
(See Lemma \ref{wlem}.) Figure \ref{bbsw} shows the sample path of $W$ and $TW$ corresponding to the particle configuration of Figure \ref{bbsfig}, where we write $TW$ for the carrier process corresponding to particle configuration $T\eta$.

\begin{figure}[!htb]
\vspace{-20pt}
\centering
\scalebox{0.58}{\includegraphics{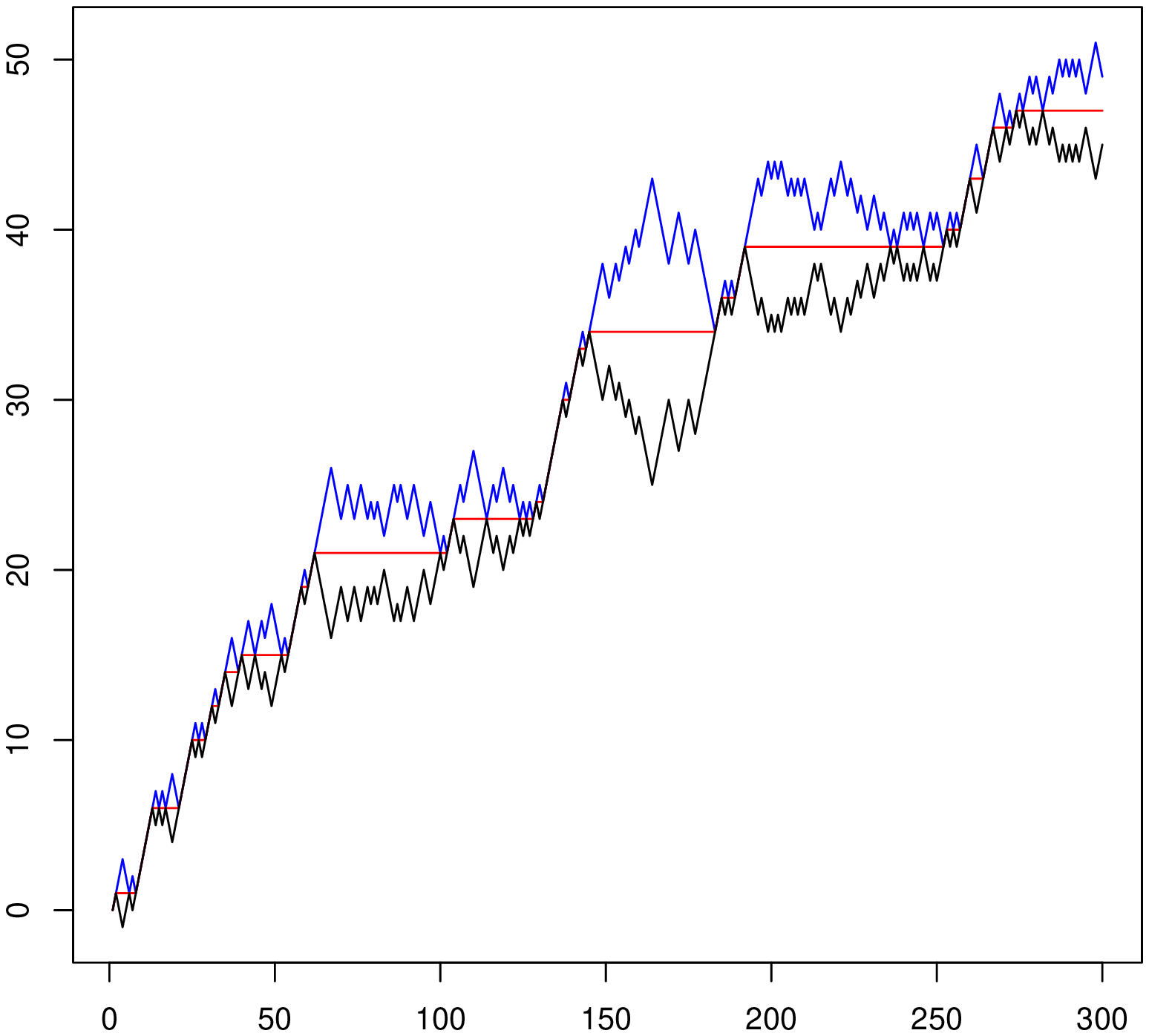}}
\vspace{-30pt}
\caption{Example sample path of $S$ (black), $M$ (red) and $TS$ (blue).}\label{bbsfig}

\scalebox{0.58}{\includegraphics{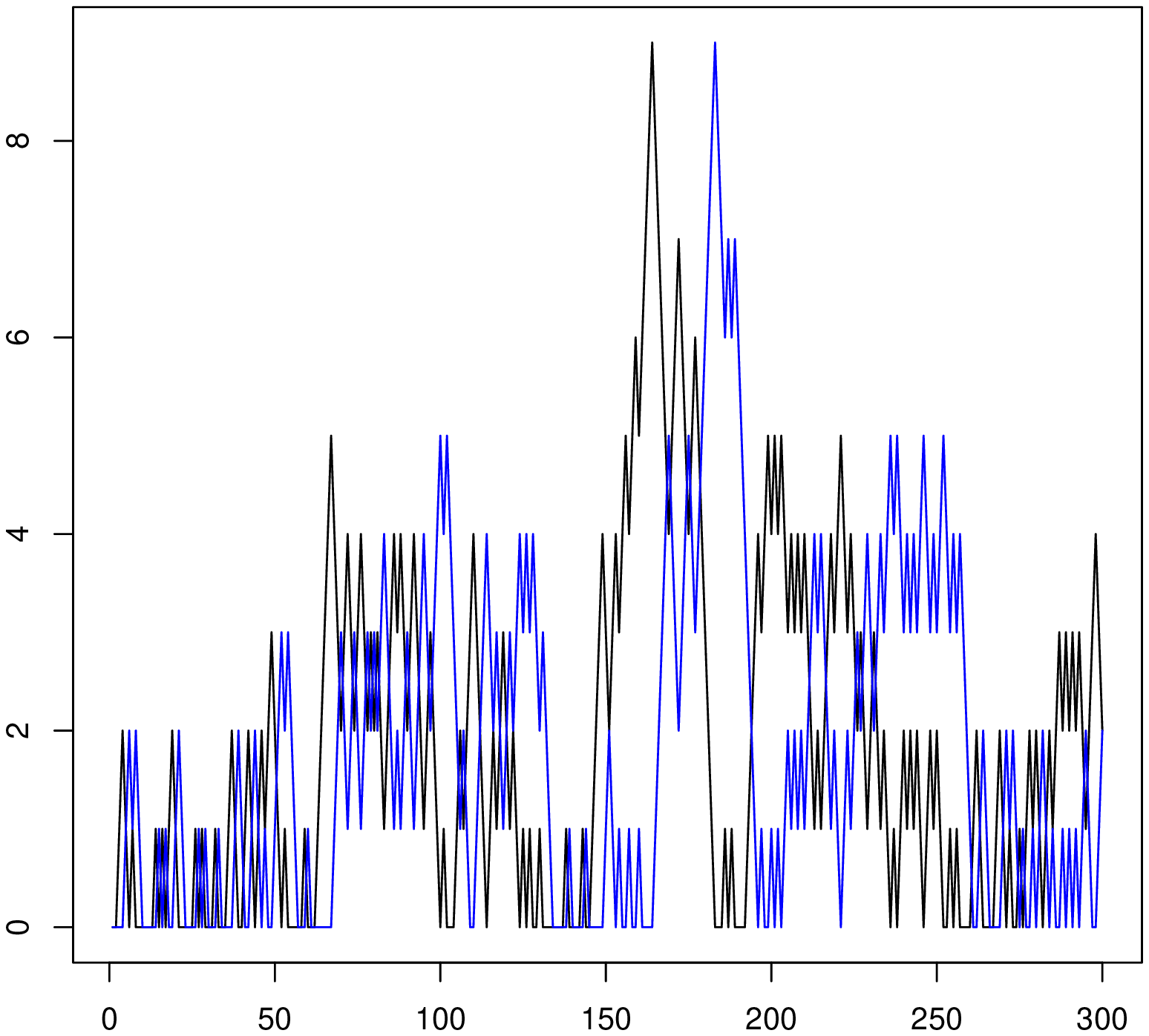}}
\vspace{-30pt}
  \caption{Example sample path of $W$ (black) and $TW$ (blue), corresponding to the particle configuration of Figure \ref{bbsfig}.}\label{bbsw}
\end{figure}

At least formally, the discussion of the previous paragraph suggests that we will be able to extend the dynamics of the system to infinite particle configurations whenever $M_0<\infty$. In Section \ref{extensionsec}, we show that this indeed is the case, and that the resulting system satisfies the update rule at (\ref{updaterule}). However, this picture is not completely satisfactory, as it does not guarantee that the second step of the dynamics will be well-defined, or that the dynamics are reversible. Regarding the latter issue in particular, as we noted for the finite particle case, \cite{takahashi1990} established that the inverse action of the system is given by running the carrier from right to left. In terms of path encodings, this is the map given by
\[T^{-1}S=2I-S-2I_0,\]
where $I=(I_n)_{n\in\mathbb{Z}}$ is the future minimum of $S$, i.e.\
\[I_n=\inf_{m\geq n}S_m.\]
(See Section \ref{inversesec}.) Similarly to the observation made for $T$ above, for $T^{-1}$ to be defined for infinite particle configurations, we will thus require $I_0>-\infty$. However, whilst we can introduce such a definition, for infinite particle systems the BBS is not conservative in general, by which we mean that for some particle configurations $T$ will send some particles `to infinity'. Since we can not expect the carrier run from right to left to recover these particles, this suggests that it is not the case that $T^{-1}T\eta=\eta$ for general particle configurations. (See Remark \ref{needcond} below for an example of a `bad' configuration, where $T^{-1}T\eta\neq\eta$.)

In view of resolving the issues raised in the previous paragraph, the initial aim of this work was to characterise the set
\begin{equation}\label{srevdef}
\mathcal{S}^{rev}:=\left\{S\in \mathcal{S}^0\::\:\mbox{$TS$, $T^{-1}S$, $T^{-1}TS$, $TT^{-1}S$ well-defined, $T^{-1}TS=S$, $TT^{-1}S=S$}\right\},
\end{equation}
upon which the one-step (forwards or backwards) dynamics are well-defined and reversible, where we have written
\begin{equation}\label{s0def}
\mathcal{S}^0:=\left\{S:\mathbb{Z}\rightarrow \mathbb{Z}:\: S_0=0,\:|S_n-S_{n-1}|=1,\:\forall n\in\mathbb{Z}\right\}
\end{equation}
for the set of two-sided nearest-neighbour paths started from 0. Moreover, since the BBS dynamics can take us out of this set, it is also natural to consider the invariant set
\begin{equation}\label{sinv}
\mathcal{S}^{inv}:=\left\{S\in \mathcal{S}^{0}\::\:\mbox{$T^kS\in \mathcal{S}^{rev}$ for all $k\in\mathbb{Z}$}\right\},
\end{equation}
upon which the dynamics are well-defined and reversible for all time.

In the following result, we give a complete description of both $\mathcal{S}^{rev}$ and $\mathcal{S}^{inv}$. For the statement of the result involving $\mathcal{S}^{inv}$, it is convenient to decompose the latter set according to the behaviour of functions at infinity, which we will describe in terms of four sets, $\mathcal{S}_{sub-critical}^{\pm}$ and
$\mathcal{S}_{critical}^{\pm}$. In particular, for these sets we have
\[\mathcal{S}_{sub-critical}^{\pm}\cap\mathcal{S}^{inv}=\left\{S\in \mathcal{S}^{inv}\::\:\lim_{n\rightarrow\pm\infty}S_n=\pm\infty\right\}.\]
The sub-criticality refers to the density of particles. Indeed, for particle configurations encoded by paths in $\mathcal{S}_{sub-critical}^{+}$ we have
\begin{equation}\label{denscond}
\sum_{m=0}^{n-1}\eta_m -\frac{n}{2}\rightarrow -\infty,
\end{equation}
as $n\rightarrow\infty$ (and a similar result holds on $\mathcal{S}_{sub-critical}^{-}$ as $n\rightarrow-\infty$), which can be interpreted as meaning we have a limiting particle density strictly below $1/2$, or, in other words, that we have infinitely many more spaces than particles asymptotically. We also have
\[\mathcal{S}_{critical}^{\pm}\cap\mathcal{S}^{inv}=\left\{S\in \mathcal{S}^{inv}\::\mbox{$S$ bounded as }n\rightarrow\pm\infty\right\}.\]
Similarly to (\ref{denscond}), we can consider particle configurations with path encodings in $\mathcal{S}_{critical}^{\pm}$ as having limiting particle density in the relevant direction of precisely $1/2$. In the following result, we establish that the sets $\mathcal{S}_{sub-critical}^{\pm}$ and $\mathcal{S}_{critical}^{\pm}$ cover all the possible boundary behaviour for functions in $\mathcal{S}^{inv}$, and give a full description of their composition. In Theorem \ref{characterizationinv} below, we give a slightly more detailed decomposition of $\mathcal{S}^{inv}$, which establishes that the set is naturally partitioned in a finer way, and also show that the BBS dynamics respect this partition.

\begin{thm}\label{mr1} (a) It holds that
\[\mathcal{S}^{rev}=\left\{S\in \mathcal{S}^0\::\:M_0<\infty,\:I_0>-\infty,\: \limsup_{n\rightarrow\infty}S_n=M_\infty,\;\liminf_{n\rightarrow-\infty}S_n=I_{-\infty}\right\},\]
where the limits $M_\infty =\lim_{n\rightarrow\infty}M_n=\sup_{n\in\mathbb{Z}}S_n$ and $I_{-\infty}=\lim_{n\rightarrow-\infty}I_n=\inf_{n\in\mathbb{Z}}S_n$ are well-defined by monotonicity.\\
(b) It holds that
\[\mathcal{S}^{inv}=\bigcup_{*_1,*_2\in\{sub-critical,critical\}}\left(\mathcal{S}_{*_1}^-\cap\mathcal{S}_{*_2}^+\right),\]
and also
\begin{eqnarray}
\lefteqn{\mathcal{S}_{sub-critical}^{\pm}}\nonumber\\
&=&\left\{S\in \mathcal{S}^{0}\::\:\lim_{n\rightarrow\pm\infty}\frac{S_n}{F(n)}=1\mbox{ for some strictly increasing function }F:\mathbb{Z}\rightarrow\mathbb{R}\right\},\nonumber\\
&&\nonumber\\
\lefteqn{\mathcal{S}_{critical}^{\pm}}\nonumber\\
&=&
\left\{S\in \mathcal{S}^{0}\::\: \sup_{n\in\mathbb{Z}}\left(M_n-I_n\right)=\limsup_{n\rightarrow\pm\infty}S_n-\liminf_{n\rightarrow\pm\infty}S_n=K\:\mbox{ for some }K\in \mathbb{N}\right\}\label{skdef}\\
&=&
\left\{S\in \mathcal{S}^{0}\::\: \sup_{n\in\mathbb{Z}}\left(M_n-S_n\right)<\infty,\:
\limsup_{n \to \pm \infty}S_n= \liminf_{n \to \pm \infty}S_n + \sup_{n}(M_n-S_n) \in \mathbb{R}\right\}.\nonumber
\end{eqnarray}
NB. Any $F$ relevant to the definition of $\mathcal{S}_{sub-critical}^{\pm}$ is necessarily divergent as $n\rightarrow\pm\infty$.
\end{thm}

With the preceding preparations in place, we turn our attention to random initial conditions, and return to the issue of invariance in distribution under $T$. Our first main result in this direction, Theorem \ref{mra}, gives some basic properties of invariant measures. In particular, it states that any path encoding within the support of an invariant measure must have matching boundary conditions at $\pm{\infty}$. More precisely, any such path encoding almost-surely takes a value in either
\begin{equation}\label{subcritdef}
\mathcal{S}_{sub-critical}:=\mathcal{S}_{sub-critical}^-\cap\mathcal{S}_{sub-critical}^+
\end{equation}
or
\[\mathcal{S}_{critical}:=\mathcal{S}_{critical}^-\cap\mathcal{S}_{critical}^+.\]
(In fact, Proposition \ref{critsubcrit} shows invariant measures are supported on slightly smaller sets than these). Moreover, Theorem \ref{mra} demonstrates that any invariant configuration has a constant density.
(NB.\ $\mathbf{P}$-a.s.\ denotes that the event in question holds with probability $1$.)

{\thm\label{mra} Suppose $\eta$ is a random particle configuration whose path encoding has distribution supported in $\mathcal{S}^{rev}$, and which satisfies $T\eta\buildrel{d}\over{=}\eta$. It is then the case that \[S\in\mathcal{S}_{sub-critical}\cup\mathcal{S}_{critical},\qquad\mathbf{P}\mbox{-a.s.}\] Moreover, there exists a constant $\rho\in[0,\frac12]$ such that
\[\mathbf{P}\left(\eta_n=1\right)=\rho,\qquad \forall n\in\mathbb{Z},\]
where $\rho=\frac12$ if and only if $S\in \mathcal{S}_{critical}$, $\mathbf{P}$-a.s.}

\begin{rem}\label{nonsrem} The above result naturally brings one to ask whether particle configurations that are invariant under $T$ are necessarily stationary with respect to spatial shifts. However, the answer to this is negative. Indeed, it is easy to construct examples of random particle configurations that are invariant under $T$, but whose law is not stationary.
For example, in Theorem \ref{mre} below, we show that the particle configuration $\eta:=(\eta_n)_{n\in\mathbb{Z}}$ given by a sequence of independent identically distributed (i.i.d.) Bernoulli random variables with parameter $p\in(0,\frac{1}{2})$ is invariant under $T$, and it readily follows that so is the non-stationary configuration $\eta'$ given by $\eta_{n}':=\eta_{\lfloor n/k\rfloor}$ for some $k\in\mathbb{N}$.
\end{rem}

By virtue of the previous result, it is only necessary to consider the behaviour of invariant measures on the sets $\mathcal{S}_{sub-critical}$ and $\mathcal{S}_{critical}$ separately. We start with the latter of these, for which it turns out that the invariant dynamics are trivial, whereby spaces and particles are simply reversed, see Theorem \ref{mrb}. Moreover, in the following result, we also consider the issue of ergodicity for invariant measures supported on $\mathcal{S}_{critical}$, showing that any ergodic measure is supported on the simplest possible set for the aforementioned dynamics.
(NB.\ $\buildrel{d}\over{=}$ denotes equality in distribution.)

{\thm\label{mrb} Suppose $\eta$ is a random particle configuration whose path encoding has distribution supported in $\mathcal{S}_{critical}$.\\
(a) It holds that $T\eta\buildrel{d}\over{=}\eta$ if and only if $\eta\buildrel{d}\over{=}1-\eta$. Moreover, under either of these conditions, we have that $T\eta=1-\eta$, $\mathbf{P}$-a.s.\\
(b) If $T\eta\buildrel{d}\over{=}\eta$, then it is the case that $\eta$ is ergodic under $T$ if and only if the distribution of $\eta$ is supported on a two point set of the form $\{\eta^*,1-\eta^*\}\subseteq\mathcal{S}_{critical}$.}

\begin{rem}\label{K-Wrem} An alternative characterisation of invariance in the critical case is given in terms of the symmetry of $W$. Indeed, we show in Proposition \ref{critprop} below that the support of any invariant measure on $\mathcal{S}_{critical}$ naturally decomposes into the sets $\mathcal{S}_K$, $K\in\mathbb{N}$, where the boundary conditions at $\pm\infty$ of $\mathcal{S}_K$ are given by fixing a single $K$ in the sets defined at \eqref{skdef}. It is then possible to check that the particle configurations whose path encodings are supported on $\mathcal{S}_K$ and which are invariant under $T$ are completely characterised by carrier processes that are stochastic processes on the state space $\{0,1,\dots,K\}$ satisfying
\[\liminf_{n\rightarrow\pm\infty}W_n=0,\]
$\mathbf{P}$-a.s., and also
\begin{equation}\label{strongsym}
W\buildrel{d}\over{=} K-W.
\end{equation}
In particular, this includes the case when $W$ is any two-sided stationary Markov process that is irreducible on $\{0,1,\dots,K\}$ and satisfies (\ref{strongsym}).
\end{rem}

Concerning measures supported on $\mathcal{S}_{sub-critical}$, we relate invariance and ergodicity under $T$ to the current of particles crossing the origin. More specifically, observe that $W_0$ represents the number of particles moved by the carrier from $\{\dots,-1,0\}$ to $\{1,2,\dots\}$ on the first evolution of the BBS, and $(T^{k-1}W)_0$ the corresponding figure for the $k$th evolution. From a particle system perspective, it is natural to ask how much information about the initial configuration is contained in the current sequence $((T^{k}W)_0)_{k\in\mathbb{Z}}$. In Section \ref{currentsec}, we provide conditions under which the entire particle configuration can be reconstructed from it. As a consequence, we are able to describe how, in the sub-critical case, the invariance and ergodicity of the map $\eta\mapsto T\eta$ precisely aligns with the corresponding properties holding for the sequence $((T^{k}W)_0)_{k\in\mathbb{Z}}$ under the canonical shift $\theta$ (i.e.\ $\theta(\dots,x_{-1},x_0,x_1,\dots)=(\dots,x_0,x_1,x_2,\dots)$).

{\thm\label{mrc} Suppose $\eta$ is a random particle configuration whose path encoding has distribution supported in $\mathcal{S}_{sub-critical}$.\\
(a) It holds that $T\eta\buildrel{d}\over{=}\eta$ if and only if $((T^kW)_0)_{k\in\mathbb{Z}}$ is stationary under $\theta$.\\
(b) The configuration $\eta$ is invariant and ergodic under $T$ if and only if $((T^kW)_0)_{k\in\mathbb{Z}}$ is stationary and ergodic under $\theta$.}
\medskip

As we will describe in Theorem \ref{mrf} and Corollary \ref{ergcormrf} below, the previous result can be applied to check the ergodicity of several examples of random configurations. However, the current might not be the most straightforward object to study, and so in the following result, we provide an alternative means for checking the invariance in distribution of random configurations. In particular, we give some simple sufficient conditions for invariance in terms of the symmetry of the particle configuration $\eta$ and carrier process $W$. For the statement of the result, we introduce the reversed configuration $\overleftarrow{\eta}$, as defined by setting
\begin{equation}\label{revcon}
\overleftarrow{\eta}_n=\eta_{-(n-1)},
\end{equation}
and the reversed carrier process $\bar{W}$, given by
\[\bar{W}_n=W_{-n}.\]

\begin{thm}\label{mrd} Suppose $\eta$ is a random particle configuration, and that the distribution of the corresponding path encoding $S$ is supported on $\mathcal{S}^{rev}$. It is then the case that any two of the three following conditions imply the third:
\begin{equation}\label{threeconds}
\overleftarrow{\eta}\buildrel{d}\over{=}\eta,\qquad \bar{W}\buildrel{d}\over{=}W,\qquad T\eta\buildrel{d}\over{=}\eta.
\end{equation}
Moreover, in the case that two of the above conditions are satisfied, then the distribution of $S$ is actually supported on $\mathcal{S}^{inv}$.
\end{thm}

Regarding the application of this general result, we first note that when $\eta$ is a stationary, ergodic sequence of Bernoulli ($\rho$) random variables, then,  by the ergodic theorem, $S\in \mathcal{S}_{sub-critical}$ almost-surely whenever the density $\rho$ is strictly less than $\frac{1}{2}$. (See Section \ref{probsec} for further details.) Within this class, we are able to present several natural examples of random particle configurations for which we can check invariance in distribution under $T$. Figure \ref{solitonsfig} shows the typical evolution of $T^kW$ for the first of these examples, illustrating the solitonic behaviour of the system. See Remark \ref{solrem} for a brief discussion of this aspect of the BBS. (We observe Remark \ref{K-Wrem} exhibits a class of examples of invariant measures with $\rho=\frac12$ whose path encodings have distribution supported in $\mathcal{S}_{critical}$, and note that when $\rho>\frac12$, $\eta$ is almost-surely not in $\mathcal{S}^{rev}$.)

\begin{thm}\label{mre} The following particle configurations all give rise to path encodings such that $S\in\mathcal{S}_{sub-critical}$, $\mathbf{P}$-a.s., and which are invariant in distribution under $T$.\\
(a) The particle configuration $(\eta_n)_{n\in\mathbb{Z}}$ given by a sequence of independent identically distributed (i.i.d.) Bernoulli random variables with parameter $p\in[0,\frac{1}{2})$.\\
(b) The particle configuration $(\eta_n)_{n\in\mathbb{Z}}$ given by a two-sided stationary Markov chain on $\{0,1\}$ with transition matrix
\[\left(
  \begin{array}{cc}
    1-p_0 & p_0 \\
    1-p_1 & p_1 \\
  \end{array}
\right)\]
where $p_0\in (0,1)$, $p_1\in[0,1)$ satisfy $p_0+p_1<1$.\\
(c) The particle configuration $(\eta_n)_{n\in\mathbb{Z}}$ given by conditioning a sequence of i.i.d.\ Bernoulli random variables with parameter $p\in(0,1)$ on the event $\sup_{n\in\mathbb{Z}}W_n\leq K$, for any $K\in\mathbb{Z}_+$. (NB. Since the event in question has probability 0, this conditioning is non-trivial, and should be understood in terms of a limiting operation which is described in Section \ref{boundedsec} and yields a Markov carrier process.)
\end{thm}

\begin{figure}[t]
\vspace{-0pt}
\centering
\includegraphics[width = 12cm]{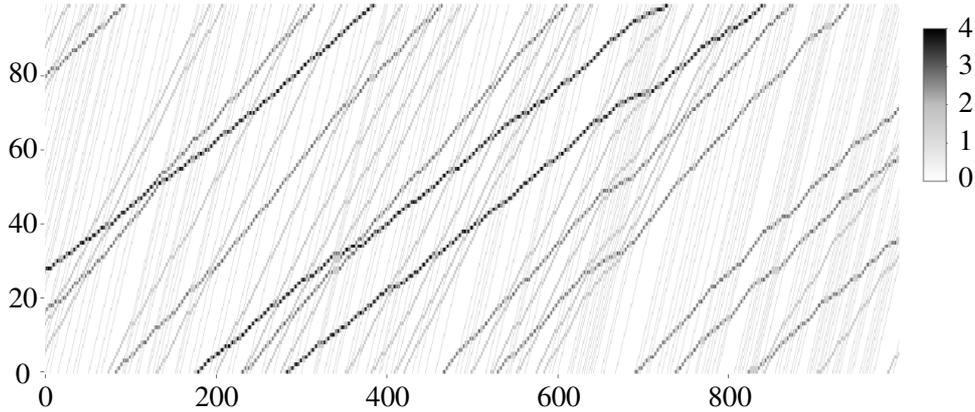}
\rput(-12,-0.2){$0$}
\rput(-9.75,-0.2){$200$}
\rput(-7.5,-0.2){$400$}
\rput(-5.25,-0.2){$600$}
\rput(-3,-0.2){$800$}
\rput(-12.3,0.1){$0$}
\rput(-12.3,1.1){$20$}
\rput(-12.3,2.1){$40$}
\rput(-12.3,3.1){$60$}
\rput(-12.3,4.1){$80$}
\rput(0.1,2.7){$0$}
\rput(0.1,3.2){$1$}
\rput(0.1,3.7){$2$}
\rput(0.1,4.2){$3$}
\rput(0.1,4.7){$4$}
\vspace{10pt}
\caption{Evolution of the box-ball system from a random initial condition. Specifically, the figure shows $((T^kW)_n)_{n=0,\dots,1,000,\: k=0,\dots,100}$ for initial configuration $\eta=(\eta_n)_{n\in\mathbb{Z}}$ a realisation of a sequence of i.i.d.\ Bernoulli($0.2$) random variables.}\label{solitonsfig}
\end{figure}

\begin{rem}
The invariance in distribution of the Markov initial configuration of Theorem \ref{mre}(b) was essentially established in \cite[Corollary 3]{HMOC}.
\end{rem}

\begin{rem} \label{mrem} We will further establish in Section \ref{examplessec} that the examples described in Remark \ref{K-Wrem} and Theorem \ref{mre} represent the only particle configurations whose path encodings are supported on ${S}^{rev}$, which are invariant under $T$, and for which $\eta$ or $W$ is a two-sided stationary Markov chain.
\end{rem}

\begin{rem}\label{solrem} Although in this paper we will not consider the soliton decomposition of particle configurations, we note that conditioning the i.i.d.\ configuration on the event $\sup_{n\in\mathbb{Z}}W_n\leq K$ as we do in the example of Theorem \ref{mre}(c) could alternatively be seen as conditioning on the configuration in question forming no solitons of size greater than $K$. Indeed, local maxima of the carrier process $W$ are in one-to-one correspondence with solitons -- that is, the basic strings of \cite{takahashi1990}, which are preserved by the BBS, and the maximum value obtained by an excursion of $W$ represents the size of the largest soliton contained within the part of the particle configuration encoded by that part of the carrier path. For more details about the soliton decomposition of random two-sided initial configurations, we refer the reader to the forthcoming article \cite{Ferrari}. In that work, the dynamics of the two-sided infinite BBS are studied for random particle configurations taking values in a subset of $\mathcal{S}_{sub-critical}$, and it is observed that the i.i.d.\ particle configuration of Theorem \ref{mre}(a) is invariant under $T$. The main focus of \cite{Ferrari}, though, is a soliton decomposition for invariant configurations with particle density strictly lower than $1/4$, with it being established that it is possible to decompose any such configuration into solitons of different sizes, and that the distributions of constituent parts of the decomposition must be independent. Moreover, \cite{Ferrari} studies the effective speed of the solitons of different sizes. In this direction, we also acknowledge the recent work of \cite{Lev}, which studies soliton sizes in the BBS with a one-sided infinite i.i.d.\ initial configuration.
\end{rem}

\begin{rem}\label{gibbs}
The invariant measures given in Theorem \ref{mre} are formally given as Gibbs measures
\[\frac{1}{Z}\exp\left(-\sum_{k=0}^{\infty}\beta_k f_k(\eta)\right)\mathbf{P}(d\eta),\]
where $\mathbf{P}$ is the reference measure under which $\eta$ is the i.i.d. sequence with density $\frac{1}{2}$, and $Z$ is a normalising constant. Moreover, in the above expression, $f_0(\eta)=\sum_{n \in \mathbb{Z}}\eta_n$ is the number of particles, and, for $k\geq1$, $f_k(\eta)$ is the number of solitons of size greater or equal to $K$. (In particular, $f_1(\eta)=\sum_{n \in \mathbb{Z}}\mathbf{1}_{\{\eta_n=1,\eta_{n+1}=0\}}$ is the number of solitons.) We note each of these is a formally conserved quantity of the BBS. Of course, in the infinite system, it is possible that some, or indeed all, of the quantities is infinite, and so to make the understanding rigourous, one would have to consider a finite box approximation, as is common when constructing Gibbs measures on infinite systems. Specifically, we observe that example (a) corresponds to taking parameters
\[\beta_0=\log\left(\frac{1-p}{p}\right),\qquad \beta_k=0,\:\forall k \ge 1,\]
where we note that the restriction $p<1/2$ is equivalent to taking $\beta_0>0$. Example (b) corresponds to parameters
\[\beta_0=\log \left(\frac{1-p_0}{p_1}\right),\qquad \beta_1=\log \left(\frac{p_1(1-p_0)}{p_0(1-p_1)}\right),\qquad\beta_k=0,\:\forall k \ge 2,\]
where $p_0+p_1<1$ is equivalent to $\beta_0 >0$. We note that this Gibbs measure takes the same form as that of the one-dimensional Ising model (or, more precisely, the related lattice gas model). Example (c) corresponds to parameters
\[\beta_0=\log\left(\frac{1-p}{p}\right),\qquad \beta_k=0,\:\forall k\in\{1,\dots K\},\qquad \beta_k=\infty,\:\forall k>K.\]
\end{rem}

\begin{rem} The BBSs that have been studied in the deterministic literature generally consist of a finite number of particles. As we commented above, however, no random configuration with a finite number of balls can be invariant under $T$. However, if we consider the periodic BBS introduced in \cite{YT} -- that is, the BBS that evolves on the torus $\mathbb{Z}/N\mathbb{Z}$, and there being strictly fewer than $N/2$ balls, then this can be embedded into our setting. Indeed, if we repeat the configuration in a cyclic fashion, then we obtain a configuration whose path encoding is in $\mathcal{S}_{sub-critical}$. Moreover, by placing equal probability on each of the distinct configurations that we see as the BBS evolves, then we obtain an invariant measure for the system. In this case, the resulting random configuration does not necessarily satisfy the symmetry requirements of Theorem \ref{mrd}. We note that the correlation functions under this measure have been studied in \cite{MaT}.
\end{rem}

In the remainder of the study, we consider more detailed properties of the evolution of the BBS that are often the focus of work in the area of interacting particle systems. One such topic we pursue is the current of particles crossing the origin $((T^{k}W)_0)_{k\in\mathbb{Z}}$, as introduced above. When $(\eta_n)_{n\in\mathbb{Z}}$ is given by a sequence of i.i.d.\ Bernoulli random variables with parameter $p<1/2$ (as in Theorem \ref{mre}(a)), we somewhat remarkably have that $((T^kW)_0)_{k\in\mathbb{Z}}$ is an i.i.d.\ sequence (see Theorem \ref{currentcltthm}). With more care, it is further possible to check that for the bounded soliton example of Theorem \ref{mre}(c) the current is a Markov process (see Proposition \ref{lll1}), and in the Markov configuration example of Theorem \ref{mre}(b), the two-dimensional process $((T^{k}\eta)_0,(T^{k}W)_0)_{k\in\mathbb{Z}}$ is a Markov chain. These observations give a route via which to analyse the current (which is, of course, particularly straightforward in the i.i.d.\ case). For instance, if we define
\begin{equation}\label{icdef}
C_k=\sum_{l=0}^{k-1}(T^{l}W)_0,
\end{equation}
which is the number of particles to have crossed the origin up to time $k$, and often referred to as the integrated current, then we can immediately deduce asymptotic distributional results for $(C_k)_{k\geq 0}$ from standard theory regarding Markov chains. In particular, we have the following result. (More detailed statements of the large deviations principle are given as parts of Theorems \ref{currentcltthm}, \ref{lll2} and \ref{llll2} below, including an explicit rate function in the i.i.d.\ case.)

\begin{thm}\label{mrf} (a) If $(\eta_n)_{n\in\mathbb{Z}}$ is given by one of the three examples of Theorem \ref{mre}, then the current sequence $((T^kW)_0)_{k\in\mathbb{Z}}$ is stationary and ergodic under $\theta$. In particular, it $\mathbf{P}$-a.s.\ holds that
\[\frac{C_k}{k}\rightarrow \mathbf{E}W_0.\]
(b) For each of the three examples, it holds that
\[\frac{C_k-k\mathbf{E}W_0}{\sqrt{\sigma^2 k}}\rightarrow N(0,1)\]
in distribution, where $N(0,1)$ is a standard normal random variable, and $\sigma^2\in(0,\infty)$ is given by
\begin{equation}\label{var1}
\sigma^2:=\mathrm{Var}\left(W_0\right)+2\sum_{k=1}^{\infty}\mathrm{Cov}\left(W_0,
\left(T^kW\right)_0\right).
\end{equation}
(c) Moreover, for each of the three examples, $(k^{-1}C_k)_{k\geq 1}$ satisfies a large deviations principle.
\end{thm}

\begin{rem} In the i.i.d.\ case of Theorem \ref{mre}(a), with parameter $p<1/2$, then $\mathbf{E}W_0$ and $\sigma^2$ can be computed explicitly to be
\begin{equation}\label{meanvar}
\mu_p=\frac{p}{1-2p},\qquad \sigma^2_p=\frac{p(1-p)}{(1-2p)^2}.
\end{equation}
More generally, in the Markov configuration of Theorem \ref{mre}(b), with parameters $p_0\in(0,1)$, $p_1\in[0,1)$ satisfying $p_0+p_1<1$, then we have that $\mathbf{E}W_0$ is equal to
\begin{equation}\label{mup0p1}
\mu_{p_0,p_1}=\frac{p_0(1-p_0+p_1)}{(1+p_0-p_1)(1-p_0-p_1)}.
\end{equation}
The limiting variance for this example can also be computed explicitly, being equal to
\begin{equation}\label{var2}
\sigma_{p_0,p_1}^2=\frac{q_0\left((1-q_0)(1+q_1)^2+2q_1(1+q_0)^2\right)}{(1+q_0)^3(1-q_1)^2},
\end{equation}
where $q_i:=p_i/(1-p_{1-i})$ for $i=0,1$.
The mean and limiting variance for the bounded soliton example of Theorem \ref{mre}(c) do not seem straightforward to compute explicitly.
\end{rem}

\begin{rem} Similarly to Remark \ref{mrem}, we establish in Corollary \ref{iidmarkovcharacterization} below that the examples of Theorem \ref{mre}(a) and (c) are the only spatially stationary random configurations with path encodings supported in $\mathcal{S}^{inv}$ that are invariant under $T$ and for which the current $((T^kW)_0)_{k\in\mathbb{Z}}$ is a two-sided stationary Markov chain.
\end{rem}

As an immediate consequence of Theorems \ref{mrc}(b) and \ref{mrf}(a), we obtain the following corollary concerning the ergodicity of the particle configurations introduced in Theorem \ref{mre}. We believe the ergodicity of $T$ will be even more widely true for stationary, ergodic (under spatial shifts) particle configurations in the sub-critical case, at least under suitable mixing conditions.

\begin{cor}\label{ergcormrf} If $(\eta_n)_{n\in\mathbb{Z}}$ is given by one of the three examples of Theorem \ref{mre}, then the transformation $\eta\mapsto T\eta$ is ergodic.
\end{cor}

We next study the progress of a single `tagged' particle in the BBS. For this part of the study, we will focus on the i.i.d.\ case of Theorem \ref{mre}(a), although, as we discuss in Remarks \ref{fiforem} and \ref{liforem}, it is also possible to extend most of the conclusions to the Markov initial configuration case of Theorem \ref{mre}(b) with only a small amount of extra work, and one of the results to the bounded soliton example of Theorem \ref{mre}(c). (See also Remark \ref{boundedrem} for comments on which of the following results can also be extended to the critical bounded soliton example of Remark \ref{K-Wrem}.) To this end, we start by making some heuristic observations (which actually hold for any stationary, ergodic $\eta$ whose path encoding satisfies $\mathbf{E}W_0=\mathbf{E}M_0<\infty$). In particular, since $W_n$ is the number of particles carried from $\{\dots,n-1,n,\}$ to $\{n+1,n+2,\dots\}$ on the first evolution of the BBS, the total particle distance travelled over the interval from $-n$ to $n$ is given by $\sum_{m=-n}^{n-1}W_m$ (roughly speaking, the area under $W$). Moreover, the number of particles in this region is $\sum_{m=-n}^{n-1}\eta_m$. Thus the average distance travelled per particle on one evolution of the BBS is $\mathbf{P}$-a.s.\ given by
\begin{equation}\label{anticipateddistance}
\lim_{n\rightarrow \infty}\frac{\sum_{m=-n}^{n-1}W_m}{\sum_{m=-n}^{n-1}\eta_m}=\frac{\mathbf{E}W_0}{\mathbf{E}\eta_0},
\end{equation}
where this final expression is equal to
\begin{equation}\label{vpdef}
v_p:=\frac{1}{1-2p}
\end{equation}
in the i.i.d.\ case. For systems that are suitably homogenous in space and time, one might anticipate that this spatial average matches the averaging seen over time for a tagged particle. That is, if we observe a tagged particle, then this should move at speed given by the formula at (\ref{anticipateddistance}), reaching a position $k{\mathbf{E}W_0}/{\mathbf{E}\eta_0}$ after $k$ evolutions of the system. In the i.i.d.\ setting, we will establish that this is indeed the case, and explore the fluctuations around this, showing that for two natural versions of the model these are of order $\sqrt{k}$. Thus we confirm that although individual solitons might move at a faster or slower rate, individual particles progress at a steady speed.

Before we get to the result, however, we need to define the tagged particle. This will be the particle that starts at position $\min\{n\geq 1:\:\eta_n=1\}$, and we will track this under repeated evolutions of the BBS. To do this, however, we need to provide more information about the action of the carrier on individual particles. Two  natural schemes one might consider for this are as follows:
\begin{description}
  \item[First-in-first-out (FIFO)] Namely, the carrier drops particles in the order in which they are collected. Note that this scheme preserves the particle ordering, and in the finite particle case is consistent with the particle picture described in
      \cite{TS1991} whereby one step of the BBS dynamics is given as follows:
      \begin{enumerate}
        \item First, move leftmost ball to its nearest empty box on the right;
        \item Move the leftmost ball of those not moved so far to its nearest empty box on the right;
        \item Repeat the previous step until all balls are moved exactly once.
      \end{enumerate}
  \item[Last-in-first-out (LIFO)] That is, the carrier drops the most recently collected particle first. This means that for an isolated string of adjacent particles, the order of particles is reversed by the action of the BBS. This scheme is consistent with the time evolution rule given in \cite{YT, YYT}, for which a single step is described as follows:
      \begin{enumerate}
        \item Move all balls with an empty box immediately on their right to that box;
        \item Neglecting the boxes to which and from the balls were moved in the previous step(s), Move all balls with an empty box immediately on their right to that box;
        \item Repeat the previous step until all balls are moved exactly once.
      \end{enumerate}
\end{description}
We will write $X^F=(X_k^F)_{k\geq 0}$ for the position of the tagged particle after $k$ evolutions of the BBS under the FIFO scheme, and $X^L=(X_k^L)_{k\geq 0}$ for the corresponding position under the LIFO scheme. It is straightforward to establish that for random particle configurations whose path encodings have a distribution supported in $\mathcal{S}^{inv}$ that $X^F$ and $X^L$ are well-defined $\mathbf{P}$-a.s. The main result we prove is as follows (see Theorem \ref{distancethm} for a more detailed statement).

\begin{thm}\label{mrg} If $(\eta_n)_{n\in\mathbb{Z}}$ is given by a sequence of i.i.d.\ Bernoulli random variables with parameter $p<1/2$, then $\mathbf{P}$-a.s.,
\[\frac{X^F_k}{k}\rightarrow v_p,\qquad \frac{X^L_k}{k}\rightarrow v_p,\]
where $v_p$ is defined as at (\ref{vpdef}). Moreover, $X^F$ admits fluctuations of order $\sqrt{k}$ around $kv_p$, and $X^L$ satisfies a central limit theorem and a large deviations principle.
\end{thm}

In the final part of the article, we study the evolution of the system in the high density regime, that is when the number of particles approaches the number of holes available. More precisely, we continue to restrict attention to the i.i.d.\ case, and consider the behaviour of the system as $p\uparrow \frac12$. As we see from the expressions for $\mu_p$ and $v_p$, defined at (\ref{meanvar}) and (\ref{vpdef}) respectively, in this regime the size of solitons and speed of particles explodes, and so scaling is necessary. The path encoding picture gives a straightforward way to do this, and in particular allows us to obtain a scaling limit using classical results of probability theory, with the limiting path encoding being a Brownian motion with drift. Moreover, the interpretation of the dynamics of the BBS in terms of the mapping at (\ref{pitmans}) naturally transfer to the limiting model, with the rescaled solitons persisting in the limit. Whilst the following result is quite elementary to prove given the results in the discrete setting, this scaling picture motivates a general continuous definition of BBS, which we call BBS on $\mathbb{R}$. In Section \ref{contsec} we present an initial exploration of this model, including showing that the natural analogues of Theorems \ref{mr1} and \ref{mrd} hold in this setting.

\begin{thm}\label{mrh} For $c>0$ and $N$ suitably large ($>c$), set
\[p_N=\frac{1}{2}-\frac{c}{2N}.\]
Let $\eta^N=(\eta^N_n)_{n\in\mathbb{Z}}$ be given by a sequence of i.i.d.\ Bernoulli random variables with parameter $p_N$, and let $S^N$ be its path encoding. It is then the case that, as $N\rightarrow\infty$,
\[\left(\frac{1}{N}S^N_{N^2t}\right)_{t\in\mathbb{R}}\rightarrow \left(B^c_t\right)_{t\in\mathbb{R}}\]
in distribution in $C(\mathbb{R},\mathbb{R})$, where $B^c_t:=B_t+ct$ for $B=(B_t)_{t\in\mathbb{R}}$ a standard two-sided Brownian motion, started from $B_0=0$ (i.e.\ $(B_t)_{t\geq 0}$ and $(B_{-t})_{t\geq 0}$ are independent standard Brownian motions started from 0). Moreover,
\[\left(\frac{1}{N}(TS^N)_{N^2t}\right)_{t\in\mathbb{R}}\rightarrow TB^c\]
in distribution in $C(\mathbb{R},\mathbb{R})$, where $TB^c$ is defined from $B^c$ analogously to the definition of $TS$ from $S$. In addition, the law of the process $B^c$ is invariant under the transformation $T$.
\end{thm}

\begin{rem} The final claim of the preceding result has already been observed in the heavy traffic regime of the queuing literature, see \cite[Theorem 3]{OCY}, and \cite{HW} for an even earlier proof.
\end{rem}

The remainder of the article is organised as follows. In Section \ref{pathsec}, we work in a deterministic framework, introducing the main objects of discussion, and establishing Theorem \ref{mr1}. Section \ref{probsec} concerns random initial configurations, and contains the proofs of Theorems \ref{mra}, \ref{mrb}, \ref{mrc}, \ref{mrd}, \ref{mre} and \ref{mrf}. An overview of the links with the literature concerning Pitman's transformation and the totally asymmetric exclusion process is then provided in Section \ref{pitmantasepsec}, following which Section \ref{contsec} details our results for the BBS on $\mathbb{R}$, including Theorem \ref{mrh}. Finally, in Section \ref{oq} we summarise some of the open questions that this article gives rise to. Regarding notational conventions, we distinguish $\mathbb{N}=\{1,2,\dots,\}$ and $\mathbb{Z}_+=\{0,1,\dots\}$.

\section{Path encodings of the BBS}\label{pathsec}

This section provides a detailed study of the path encodings of particle configurations, and their dynamics under the BBS. We start by presenting the path description of the initial particle configuration and carrier process, initially for the one-sided case, when particles are sited on a half-infinite line (see Sections \ref{configsec} and \ref{carriersec}), but later extend this construction to cover the two-sided case, when particles may be spread along the entirety of the integers (see Section \ref{extensionsec}), and also highlight the subtlety of defining the carrier in this more general setting (see Section \ref{uniquesection}). Moreover, in Section \ref{actionsec}, we introduce Pitman's transformation in the one-sided case, before extending this to the two-sided case (again in Section \ref{extensionsec}), discussing its inverse for two-sided infinite configurations (see Section \ref{inversesec}), and studying the invariant set $\mathcal{S}^{inv}$ introduced at (\ref{sinv}) (see Section \ref{invariantsetsec}). In particular, we establish Theorem \ref{mr1}. We will see in later sections that the viewpoint set out here is extremely useful for probabilistic analysis when the initial configuration is random. One further issue that will be particularly relevant in the study of BBSs with random two-sided infinite configurations is the question of whether the current contains enough information to recover the particle configuration; this is explored in Section \ref{currentsec}.

\subsection{Initial configuration of the one-sided BBS}\label{configsec}

In the next three subsections, we consider the one-sided case, that is, the box-ball system on $\mathbb{Z}_+=\{0,1,2,\dots\}$. We denote by $\eta=(\eta_n)_{n\in\mathbb{N}}\in \{0,1\}^{\mathbb{N}}$ a particle configuration. Specifically, as in the introduction, we write $\eta_n=1$ if there is a particle at $n$, and $\eta_n=0$ otherwise. Also as in the introduction, we can summarise this in a nearest-neighbour walk path $S=(S_n)_{n\in\mathbb{Z}_+}$, where $S_0=0$, and the increments of $S$ are defined as at (\ref{SRWrep}) for $n\geq 1$.

\subsection{Carrier process for the one-sided BBS}\label{carriersec}

As described in the introduction, the carrier moves along $\mathbb{Z}_+$, picking up a particle when it crosses one, and dropping off a particle when it is holding at least one particle and sees a space. In particular, this is the process $W=(W_n)_{n\in \mathbb{Z}_+}$, obtained by setting $W_0=0$ and satisfying (\ref{wupdate}). As claimed at (\ref{wms}), it turns out we can write $W$ as a difference of $S$ from its maximum process, which we denote by $M=(M_n)_{n\in \mathbb{Z}_+}$, and is defined as at (\ref{mdef}). Specifically, we have the following lemma.

{\lem\label{wlem}
It holds that
\[W_n=M_n-S_n,\qquad \forall n\in\mathbb{Z}_+.\]}
\begin{proof} We will proceed by induction. Clearly the result is true for $n=0$. Suppose that we have checked $W_{n-1}=M_{n-1}-S_{n-1}$ for some $n\geq 1$. It then holds that
\begin{equation}\label{wincrement}
W_n-W_{n-1}=\left\{\begin{array}{ll}
               +1, & \mbox{if }\eta_n=1,\\
               0, & \mbox{if }\eta_n=0\mbox{ and }M_{n-1}=S_{n-1},\\
               -1, & \mbox{if }\eta_n=0\mbox{ and }M_{n-1}>S_{n-1}.
             \end{array}\right.
\end{equation}
Now, if $\eta_n=1$, then $S_n=S_{n-1}-1$ and $M_n=M_{n-1}$, and so
\[M_n-S_n-(M_{n-1}-S_{n-1})=1.\]
Moreover, if $\eta_n=0$ and $M_{n-1}=S_{n-1}$, then it must also be the case that
$M_n=S_n$, and so
\[M_n-S_n-(M_{n-1}-S_{n-1})=0.\]
Similarly, if $\eta_n=0$ and $M_{n-1}>S_{n-1}=0$, then $M_n=M_{n-1}$ and $S_n=S_{n-1}+1$, and so
\[M_n-S_n-(M_{n-1}-S_{n-1})=-1.\]
In particular, we have checked that
\[W_n-W_{n-1}=M_n-S_n-(M_{n-1}-S_{n-1}),\]
which by the inductive hypothesis implies $W_n=M_n-S_n$, as desired.
\end{proof}

The above lemma explains how to obtain $W$ from $S$. It is also possible to describe the inverse mapping explicitly, as we do in the next result. To this end, we introduce a version of the local time of $W$ at 0, $\ell=(\ell_n)_{n\in\mathbb{Z}_+}$, by setting $\ell_0=0$ and, for $n\geq 1$,
\[\ell_n=\sum_{m=1}^{n}\mathbf{1}_{\{W_{m-1}=W_{m}=0\}}.\]

{\lem
\label{llem}
It holds that
\[S_n=\ell_n-W_n,\qquad \forall n\in\mathbb{Z}_+.\]}
\begin{proof} The result is obvious for $n=0$. For $n\geq1$, from (\ref{wincrement}) we have that
\begin{eqnarray*}
W_n&=&\sum_{m=1}^n\left(W_m-W_{m-1}\right)\\
&=&\sum_{m=1}^n \left[\left(S_{m-1}-S_m\right)\mathbf{1}_{\{S_{m-1}<M_{m-1}\}}+\mathbf{1}_{\{S_{m}-S_{m-1}=-1\}}\mathbf{1}_{\{S_{m-1}=M_{m-1}\}}\right]\\
&=&\sum_{m=1}^n \left[\left(S_{m-1}-S_m\right)+\left(S_m-S_{m-1}+\mathbf{1}_{\{S_{m}-S_{m-1}=-1\}}\right)\mathbf{1}_{\{S_{m-1}=M_{m-1}\}}\right]\\
&=&-S_n+\sum_{m=1}^n \mathbf{1}_{\{S_{m}-S_{m-1}=1\}}\mathbf{1}_{\{S_{m-1}=M_{m-1}\}}\\
&=&-S_n+\sum_{m=1}^n \mathbf{1}_{\{W_{m}=W_{m-1}=0\}}\\
&=&-S_n+\ell_n,
\end{eqnarray*}
which completes the proof.
\end{proof}

\subsection{Action of the carrier for the one-sided BBS}\label{actionsec}

In the one-sided setting, the action $T$ of the carrier on $\eta$ was defined at (\ref{origdef}). However, this formula is not especially convenient for analysis, especially when we seek to extend the dynamics to the two-sided infinite case. For what follows, we find that it is clearer when we consider the action of $T$ on $S$. In this direction, it is helpful to observe that the positions of the particles prior to the carrier passing them precisely corresponds to location of up jumps of $W$, and that the positions of particles after the carrier has visited corresponds to locations of down jumps of $W$.  Formally, we can write this as \[(T\eta)_n=\mathbf{1}_{\{W_n=W_{n-1}-1\}},\qquad \forall n\in\mathbb{N}.\]
The following lemma explains how this equation yields the identity at (\ref{pitmans}). (We note that $M_0=0$ in the present setting.) To be precise, as in the introduction, we write $TS=((TS)_n)_{n\geq 0}$ for the path encoding of $T\eta$.

{\lem\label{tlem} It holds that
\[(TS)_n=2M_n-S_n,\qquad \forall n\in\mathbb{Z}_+.\]}
\begin{proof} First observe that $(TS)_n-(TS)_{n-1}=-1$ if and only if there is a particle at $n$ after the carrier has passed. As noted above the lemma, the latter  is equivalent to $W_n-W_{n-1}=-1$. Thus
\begin{eqnarray*}
(TS)_n-(TS)_{n-1}&=&1-2\mathbf{1}_{\{W_n-W_{n-1}=-1\}}\\
&=&1-2\mathbf{1}_{\{S_{n-1}<M_{n-1},S_n-S_{n-1}=1\}}\\
&=&S_{n-1}-S_n+2\mathbf{1}_{\{S_{n-1}=M_{n-1},S_n-S_{n-1}=1\}}.
\end{eqnarray*}
Summing over the increments thus yields
\begin{eqnarray*}
(TS)_n-(TS)_0&=&\sum_{m=1}^n(TS)_m-(TS)_{m-1}\\
&=&S_0-S_n+2\sum_{m=1}^n\mathbf{1}_{\{S_{m-1}=M_{m-1},S_m-S_{m-1}=1\}}\\
&=&S_0-S_n+2(M_n-M_0).
\end{eqnarray*}
Since $(TS)_0=S_0=M_0=0$, we are done.
\end{proof}

\subsection{Extension to the two-sided BBS}\label{extensionsec}

In this section, we discuss extending to the case when we have a doubly-infinite particle configuration $(\eta_n)_{n\in\mathbb{Z}}\in\{0,1\}^\mathbb{Z}$. We can again encode this in a nearest neighbour path $S=(S_n)_{n\in\mathbb{Z}}$ by continuing to assume $S_0=0$, and defining increments of $S$ as at \eqref{SRWrep}. Such a path is always an element of $\mathcal{S}^0$, as defined at (\ref{s0def}). Whilst in the one-sided case, the construction of the carrier $W$ and transformed path $TS$ was possible for any particle configuration, in the two-sided case we do need some restriction to be able to define the relevant objects finitely. In this subsection, we make the following assumption on $S$:
\begin{equation}\label{sfinlim}
\limsup_{n\rightarrow-\infty}S_n<\infty,
\end{equation}
which means that the process $M=(M_n)_{n\in\mathbb{Z}}$ defined by (\ref{mdef}) is finite. In fact,  the assumption at (\ref{sfinlim}) is equivalent to the condition at $M_0<\infty$. Formally, we can then define the two-sided carrier process $W=(W_n)_{n\in\mathbb{Z}}$ and transformed path $TS=((TS)_n)_{n\in\mathbb{Z}}$ by setting
\begin{equation}\label{wtsdef}
W_n=M_n-S_n,\qquad (TS)_n=2M_n-S_n-2M_0,
\end{equation}
as is motivated by Lemmas \ref{wlem} and \ref{tlem} respectively. We then have that $TS\in\mathcal{S}^0$, and we can define the transformed particle configuration $((T\eta)_n)_{n\in\mathbb{Z}}$ by setting
\begin{equation}\label{tetadef}
(T\eta)_n=\mathbf{1}_{\{(TS)_n=(TS)_{n-1}-1\}},\qquad \forall n\in\mathbb{Z}.
\end{equation}
However, it is not a priori clear that these definitions are justified, since one can not simply start the carrier at $-\infty$. To provide this justification, we show that the definitions are consistent with taking the limit of a sequence of BBSs on half-lines, which are well-defined.

We start by introducing notation for the approximating sequence of BBSs. In particular, we write $\eta^{[k]}=(\eta^{[k]}_n)_{n\in\mathbb{Z}}$ for the particle configuration truncated at $k$ and below. That is,
\[\eta^{[k]}_n=\eta_n\mathbf{1}_{\{n>k\}}.\]
Write $S^{[k]}=(S^{[k]}_n)_{n\in\mathbb{Z}}$ for the corresponding path encoding (again, defined by setting $S^{[k]}_0=0$ and defining increments as at (\ref{SRWrep})). Note that for $n\leq k$, it holds that $S^{[k]}_n-S^{[k]}_{n-1}=1$, and so (\ref{sfinlim}) is satisfied by $S^{[k]}$. This implies that the corresponding maximum process $M^{[k]}$ is well-defined. Now, let $W^{[k]}=(W^{[k]}_n)_{n\in\mathbb{Z}}$ be the carrier process corresponding to $\eta^{[k]}$. Clearly there is no problem with defining this process, since we know the carrier is empty up to location $k$. Moreover, proceeding as in Lemma \ref{wlem}, we have that
\begin{equation}\label{wkident}
W^{[k]}=M^{[k]}-S^{[k]}.
\end{equation}
Similarly, we can also define the transformed path $T(S^{[k]})$, and check, as in Lemma \ref{tlem}, that
\begin{equation}\label{tsindent}
T(S^{[k]})=2M^{[k]}-S^{[k]}-2M^{[k]}_0.
\end{equation}
By virtue of the following natural limit result, we can understand $W$ and $TS$, defined at (\ref{wtsdef}), as the carrier process and transformed path for the two-sided particle configuration. Moreover, we check that the update rule presented at (\ref{updaterule}) is still valid in this more general setting.

{\lem \label{twosided} Suppose $S$ is an element of $\mathcal{S}^0$ satisfying \eqref{sfinlim}, then
\begin{equation}\label{firstclaims}
S^{[k]}\rightarrow S,\qquad M^{[k]}\rightarrow M,\qquad W^{[k]}\rightarrow W,\qquad T(S^{[k]})\rightarrow TS,
\end{equation}
as $k\rightarrow-\infty$, where $M$, $W$ and $TS$ are defined at (\ref{mdef}) and (\ref{wtsdef}). Moreover,
\begin{equation}\label{secondclaims}
\eta^{[k]}\rightarrow \eta,\qquad T(\eta^{[k]})\rightarrow T\eta,
\end{equation}
where $T\eta$ is defined at (\ref{tetadef}), and
\begin{equation}\label{twosidedupdaterule}
(T\eta)_{n}=\min\left\{1-\eta_{n},W_{n-1}\right\},\qquad \forall n\in\mathbb{Z}.
\end{equation}}
\begin{proof} Since we are dealing with discrete time processes, it suffices to show convergence pointwise. Since $S^{[k]}_n= S_n$ for $n\geq k$ (assuming $k$ is negative), we readily deduce that $S^{[k]}_n\rightarrow S_n$ as $k\rightarrow-\infty$, which establishes the first claim. For the second claim, we start by noting the obvious inequality $S_n^{[k]}\leq S_n$ for all $n\in\mathbb{Z}$, $k\leq 0$. This implies that $M_n^{[k]}\leq M_n$ for all $n\in\mathbb{Z}$, $k\leq 0$. To obtain the opposite inequality in the limit, observe that (\ref{sfinlim}) yields, for a given $n$, the existence of a finite $n_1\leq n$ such that $M_n=S_{n_1}$. Moreover, for $k\leq n_1$,
\[M_{n}^{[k]}=\max_{k\leq m\leq n}S_{m}^{[k]}\geq S_{n_1}^{[k]}.\]
Since $S_{n_1}^{[k]}\rightarrow S_{n_1}=M_n$ as $k\rightarrow -\infty$ by the previous part of the proof, we have thus demonstrated that $M_{n}^{[k]}\rightarrow M_n$ as $k\rightarrow -\infty$, as desired. The remaining two claims of \eqref{firstclaims} are now easy consequences of the definition of $W$ and $TS$, and equations (\ref{wkident}) and (\ref{tsindent}). To complete the proof, we simply observe that the first claim of (\ref{secondclaims}) is clear by definition, the results at (\ref{firstclaims}) imply
\[(T(\eta^{[k]}))_n=\mathbf{1}_{\{(T(S^{[k]}))_n=(T(S^{[k]}))_{n-1}-1\}}\rightarrow \mathbf{1}_{\{(TS)_n=(TS)_{n-1}-1\}}=(T\eta)_n,\]
as $k\rightarrow-\infty$, and also
\[(T\eta)_{n}=\lim_{k\rightarrow-\infty}(T(\eta^{[k]}))_n=
\lim_{k\rightarrow-\infty}\min\left\{1-\eta^{[k]}_{n},W^{[k]}_{n-1}\right\}=
\min\left\{1-\eta_{n},W_{n-1}\right\},\]
where for the second inequality we apply the update rule from (\ref{updaterule}) that holds in the one-sided setting.
\end{proof}

{\rem In the two-sided case, it is further straightforward to check that the identity of Lemma \ref{llem} can be adapted to
\[S_n=\ell_n-W_n+W_0,\qquad \forall n\in\mathbb{Z},\]
if the definition of the local time $\ell$ is extended as follows:
\begin{equation}\label{localdef}
\ell_n=\left\{\begin{array}{ll}
              \sum_{m=1}^{n}\mathbf{1}_{\{W_{m-1}=W_{m}=0\}}, & \mbox{if }n>0, \\
              0 & \mbox{if }n=0,\\
               -\sum_{m=n+1}^{0}\mathbf{1}_{\{W_{m-1}=W_{m}=0\}}, & \mbox{if }n<0.
            \end{array}\right.
            \end{equation}}

\subsection{Defining the carrier process and dynamics uniquely}\label{uniquesection}
In Section \ref{extensionsec}, we gave a definition of the dynamics of the BBS for two-sided infinite particle configurations satisfying (\ref{sfinlim}). In this section, we show that for no other configurations can the BBS dynamics be reasonably defined, in the sense that one can not construct a carrier process which corresponds to the configuration in the natural way. Moreover, we will show that the carrier process defined in the previous section is unique in a certain sense, which we relate to excluding the possibility of particles coming into the system from $-\infty$. The discussion presented here will be useful in subsequent sections when it comes to defining the inverse of the dynamics, and proving Theorem \ref{mr1}.

We start by introducing the space of carrier paths
\[ \mathcal{Y}:=\left\{Y: \mathbb{Z} \to \mathbb{Z}_+\::\: |Y_n-Y_{n-1}|=1\mbox{ or }Y_n=Y_{n-1}=0,\: \forall n \in \mathbb{Z}\right\}.\]
Given a configuration $\eta\in\{0,1\}^\mathbb{Z}$, the associated carrier path, which is formally given by $Y_n=\sum_{k=-\infty}^{n}(\eta_k-T\eta_k)$, should satisfy the update rule at (\ref{wupdate}) (for all $n\in\mathbb{Z}$). Equivalently, we require $Y$ to satisfy $\eta_n=\mathbf{1}_{\{Y_{n}=Y_{n-1}+1\}}$. This being the case, we define a map $\Phi : \mathcal{Y} \to \{0,1\}^{\mathbb{Z}}$ by setting
\[ (\Phi Y)_n = \mathbf{1}_{\{Y_{n}=Y_{n-1}+1\}}.\]
Since the map from $\eta \in \{0,1\}^{\mathbb{Z}}$ to $S \in \mathcal{S}^0$ is one-to-one, this map may equivalently viewed as a map from $\mathcal{Y}$ to $\mathcal{S}^0$. Henceforth, in a slight abuse of notation, we will use $\Phi$ for both versions of the map, with it being clear from the context whether we are mapping to $\{0,1\}^\mathbb{Z}$ or $\mathcal{S}^0$. In fact, $\Phi:\mathcal{Y}\rightarrow\mathcal{S}^0$ is given explicitly by
\begin{equation}\label{phidefY}
\Phi Y_n =\ell(Y)_n -Y_n+Y_0,
\end{equation}
where $\ell : \mathcal{Y} \to \mathcal{A}^0$ is defined as same way as (\ref{localdef}), and
\[ \mathcal{A}^0 =\left\{A : \mathbb{Z} \to \mathbb{Z}\::\: A_0=0, \: A_n-A_{n-1} \in \{0,1\},\: \forall n \in \mathbb{Z} \right\}\]
is a set of non-decreasing functions passing through the origin. In Section \ref{extensionsec}, we showed the following result.

\begin{lem} Assume $S \in \mathcal{S}^0$ satisfies (\ref{sfinlim}), and let $W=M-S$. It is then the case that $W \in \mathcal{Y}$ and $\Phi W=S$. In particular, $S \in \Phi(\mathcal{Y})$. Furthermore, $\lim_{k \to -\infty}T \eta^{[k]}$ exists and is equal to $T\eta$ given by (\ref{tetadef}), and satisfies (\ref{twosidedupdaterule}), i.e.\ $T\eta_n=\min \{1-\eta_n, W_{n-1}\}$.
\end{lem}

The above lemma tells that (\ref{sfinlim}) is a sufficient condition for $S$ to have an associated carrier path. The next lemma establishes (\ref{sfinlim}) is in fact a necessary condition for this. In fact, this result provides an intuitive explanation for the dynamics of the BBS, suggesting that when $\limsup_{n \to -\infty}S_n=\infty$ we have a carrier bringing infinite particles from $-\infty$, which fill all the holes, and transports all the particles to $\infty$. The limits in Lemma \ref{twosided} can also be understood in this case, although those for $M$ and $W$ will be $\infty$, and $TS$ no longer can be defined via \eqref{wtsdef}.

\begin{lem} \label{wbadlem} Assume $S \in \mathcal{S}^0$ and $\limsup_{n \to -\infty}S_n=\infty$. It is then the case that $S \notin \Phi(\mathcal{Y})$. In particular, $\lim_{k \to -\infty} W^{[k]}_n= \infty$, and it moreover holds that $\lim_{k \to -\infty} T\eta^{[k]}=1-\eta$.
\end{lem}
\begin{proof} Suppose $S = \Phi Y$ for some $Y \in \mathcal{Y}$. Applying \eqref{phidefY} yields that, for any $n\ge 0$,
\begin{align*}
S_0-S_{-n} & =\sum_{k=-n+1}^{0}(S_{k}-S_{k-1})\\
&=\sum_{k=-n+1}^0 (\ell(Y)_k -Y_k-\ell(Y)_{k-1} +Y_{k-1}) \\
& \ge \sum_{k=-n+1}^0 (Y_{k-1}-Y_k)\\
&=Y_{-n}-Y_0\\
& \ge -Y_0.
\end{align*}
Namely $S_{-n} \le Y_0$. However, this contradicts the assumption that $\limsup_{n \to -\infty}S_n=\infty$, and so it must be the case that $S \notin \Phi(\mathcal{Y})$.

For the remaining claims, we note that, for any $n \in \mathbb{Z}$ and $k \le n$,
\begin{align*}
W_n^{[k]}=M_n^{[k]}-S_n=\max_{k\leq m\leq n}S_m-S_n,
\end{align*}
and so $\lim_{k \to -\infty} W_n^{[k]} = \lim_{k \to -\infty}  (\max_{k\leq m\leq n}S_m-S_n) =\infty$. Since $T\eta^{[k]}_n=\min\{1-\eta_n^{[k]}, W_{n-1}^{[k]}\}$, it also follows that $\lim_{k \to -\infty}T\eta^{[k]}_n=1-\eta_n$.
\end{proof}

Now, if the BBS dynamics are to be defined in terms of the carrier, then it is natural to define the domain of $T$ to be the set where we have a carrier process, i.e.\
\[\mathcal{S}^{T}:=\Phi(\mathcal{Y}).\]
Moreover, the previous two lemmas yield the following alternative expression for $\mathcal{S}^T$.

\begin{cor}\label{domaincor} It holds that
\[\mathcal{S}^{T}=\left\{S \in \mathcal{S}^{0}\::\:\limsup_{n \to \infty} S_n  < \infty \right\} =\left\{ S \in \mathcal{S}^{0}\::\:M_0\in\mathbb{R}\right\}.\]
\end{cor}

Whilst the above discussion might give a suitable domain for $T$, we further need to consider what the canonical action of the BBS should be. Indeed, given a particle configuration $\eta$ and carrier process $Y$, one might seek to define the dynamics in terms of the update rule at \eqref{twosidedupdaterule}. However, as we discuss further in Remark \ref{noninirem} below, $\Phi$ is not an injective map, and consequently the latter approach does not uniquely define the dynamics of the system. The next lemma is intended to give a solution to this problem. In particular, we will identify a subset of $\mathcal{Y}$ on which the map is a bijection, and show that the carriers contained in this set (which are in fact given by the process $W=M-S$ defined in the previous section) are the minimal carriers for the corresponding configuration. As we argue in Remark \ref{noninirem}, this minimal carrier can be seen as a natural choice, since it is the one that excludes the possibility of extra particles appearing in the system from $-\infty$. The following proposition may be considered the main result of this subsection.

\begin{prop}\label{phiproperty}
Define
\[\mathcal{Y}^-:=\left\{ Y \in \mathcal{Y}\::\:\liminf_{n \to -\infty}Y_n=0\right\}.\]
It is then the case that $\Phi|_{\mathcal{Y}^-} :\mathcal{Y}^- \to \mathcal{S}^T$ is a bijection, and its inverse map is given by $\Phi^{-1}S=M-S$. Furthermore,
\[(\Phi^{-1}S)_n=\min \left\{Y_n\::\:\Phi Y = S,\: Y \in \mathcal{Y} \right\}.\]
\end{prop}
\begin{proof} Since $M-S \in \mathcal{Y}^-$ and $\Phi (M-S)=S$ for any $S \in \mathcal{S}^T$, we only need to show that $\Phi|_{\mathcal{Y}^-}$ is injective for the first statement. Let $Y, \tilde{Y} \in \mathcal{Y}^-$ satisfy $\Phi Y=\Phi \tilde{Y}$, and suppose $Y_{n_0} > \tilde{Y}_{n_0}$ for some $n_0$. Let $n_1:=\sup\{ n \le n_0:\: Y_n \le \tilde{Y}_n \}$, with the convention that $\sup \emptyset =-\infty$. If $n_1 \neq -\infty$, then $Y_{n_1} \le \tilde{Y}_{n_1}$ and $Y_{n_1+1} > \tilde{Y}_{n_1+1}$. Thus either
\[Y_{n_1 +1} - Y_{n_1}  =1\mbox{ and }\tilde{Y}_{n_1+1} - \tilde{Y}_{n_1}  \le 0,\]
or
\[Y_{n_1 +1} - Y_{n_1}   =0\mbox{ and }\tilde{Y}_{n_1+1} - \tilde{Y}_{n_1}  =-1,\]
must be satisfied. However, since $\mathbf{1}_{\{Y_{n_1}=Y_{n_1}+1\}}=\mathbf{1}_{\{\tilde{Y}_{n_1}=\tilde{Y}_{n_1}+1\}}$ and $Y_{n_1 +1} - Y_{n_1}   =0$ implies $Y_{n_1}=Y_{n_1 +1}=0$, neither of the above possibilities can occur. Hence, $n_1=-\infty$ and $Y_{n} > \tilde{Y}_n$ for all $n \le n_0$. In particular, this implies $Y_n \ge \tilde{Y}_n+1 \ge 1$ for all $n \le n_0$, and so $\liminf_{n \to -\infty}Y_n \ge 1$, which contradicts the assumption that $Y \in \mathcal{Y}^-$.

The claim that $(\Phi^{-1}S)_n=\min \{Y_n \ ; \ \Phi Y = S, \ Y \in \mathcal{Y} \}$ is shown by a similar argument. Suppose $Y \in \mathcal{Y}$ satisfies $Y_{n_0} < (\Phi^{-1}S)_{n_0}$ for some $n_0$, where $S= \Phi Y$. Then, since $\Phi Y= \Phi \Phi^{-1}S$, by the argument we have just given, $(\Phi^{-1}S)_n > Y_n$ for all $n \le n_0$. However, this contradicts the fact that $\Phi^{-1}S \in \mathcal{Y}^-$.
\end{proof}

As a simple corollary of this result, we have the following, which was essentially shown in Subsection \ref{extensionsec}.

\begin{cor}\label{corafter} For $S \in \mathcal{S}^{T}$, it holds that $W=\Phi^{-1}S$ and $M-M_0=\ell(\Phi^{-1}S)$.
\end{cor}
\begin{proof} The identity $W= \Phi^{-1}S$ follows by the definition of $W$ at (\ref{wtsdef}). Since
\begin{align*}
M&=W+S\\
&=\Phi^{-1}S +S \\
&= \Phi^{-1}S + \Phi\Phi^{-1}S\\
&=\Phi^{-1}S+\ell(\Phi^{-1}S)-\Phi^{-1}S+(\Phi^{-1}S)_0\\
&=\ell(\Phi^{-1}S) +(\Phi^{-1}S)_0,
\end{align*}
we have $M_0=(\Phi^{-1}S)_0$ and $M-M_0=\ell(\Phi^{-1}S)$.
\end{proof}

\begin{rem}\label{noninirem} To indicate why the process $\Phi^{-1}S=M-S$ is the natural carrier for a particle configuration, first consider the empty configuration $\eta$ given by setting $\eta_n=0$ for all $n\in\mathbb{Z}$. In this case, for each $N\in\mathbb{Z}$, if we set
\[Y^{N}_n=(N-n)_+,\qquad \forall n\in\mathbb{Z},\]
then we obtain a carrier $Y^N\in \mathcal{Y}$ such that $\Phi Y^N=\eta$. None of these functions are equal to the minimal carrier $W=M-S\equiv 0$. Now, if we define the BBS via the update rule (\ref{twosidedupdaterule}), then the carrier $Y^N$ yields a new configuration
\[\min\left\{1-\eta_{n},Y^N_{n-1}\right\}=\mathbf{1}_{\{n\leq N\}},\]
whereas the carrier $W$ yields $\min\{1-\eta_{n},W_{n-1}\}=\eta$. (See Figure \ref{nonunicarrier}.) In particular, we see that the carriers $Y^N$ are transporting a semi-infinite line of particles from $-\infty$. Without good reason for wanting to create such particles, it is thus apparently more natural to take as a carrier $W$. We note that in this case, the latter process is the only one for which $\lim_{k \to -\infty}T \eta^{[k]}$ is equal to the output of (\ref{twosidedupdaterule}), which also gives an understanding that, in this case, the system is not picking up information from $-\infty$. A similar argument can be given for other configurations where $W$ has flat segments (i.e.\ locations where $W_n=W_{n+1}=0$), see Remark \ref{needcond} for a particular example. In the same spirit, one might think about the case when $\eta\not \in \Phi(\mathcal{Y})$ and $\lim_{k \to -\infty}T \eta^{[k]}=1-\eta$, as described in Lemma \ref{wbadlem}, as having a carrier which transports an infinite number of particles from $-\infty$ to fill all the holes with particles, and which transports all the particles to $+\infty$.

\begin{figure}[!htb]
\vspace{0pt}
\centering
\scalebox{0.3}{\hspace{250pt}\includegraphics{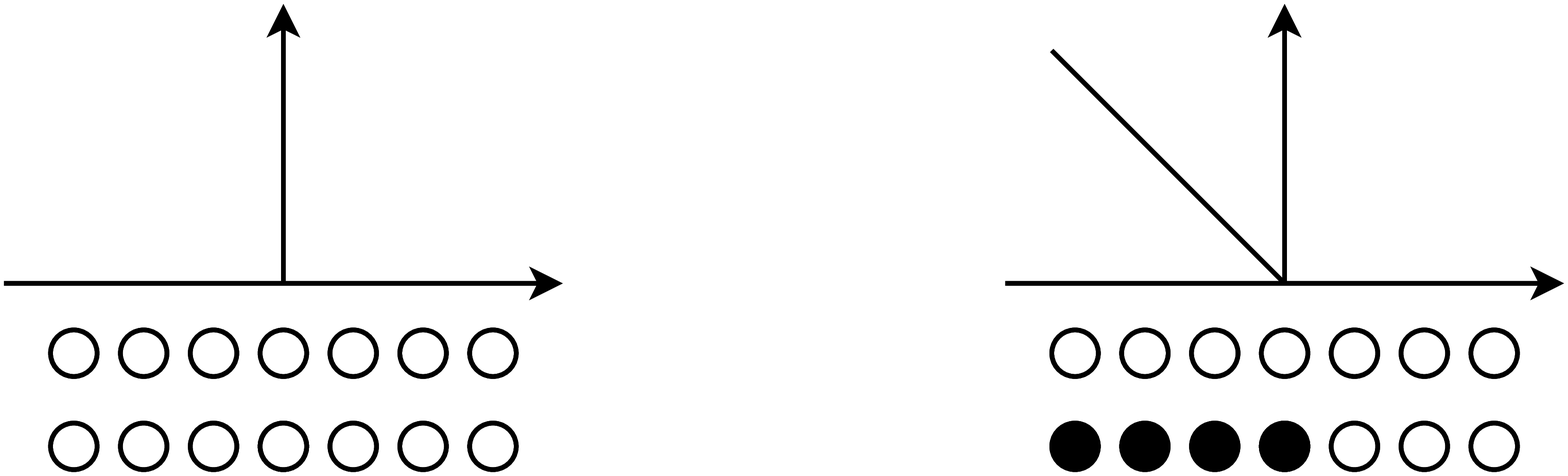}}
\rput(-6.9,2.6){$W$}
\rput(-0.3,2.6){$Y^0$}
\rput(-10.3,0.7){$\eta_n$}
\rput(-11.5,0.15){$\min\{1-\eta_{n},W_{n-1}\}$}
\rput(-3.9,0.7){$\eta_n$}
\rput(-5.0,0.15){$\min\{1-\eta_{n},Y^0_{n-1}\}$}
\vspace{0pt}
\caption{Particles transported from $-\infty$ by non-minimal carrier.}\label{nonunicarrier}
\end{figure}

As a slightly cautionary example about using the condition that $\lim_{k \to -\infty}T \eta^{[k]}$ is equal to the right-hand side of (\ref{twosidedupdaterule}) to pick the carrier, however, let us consider $\eta$ given by $\eta_n=\mathbf{1}_{\{n\mbox{ odd}\}}$. (See Figure \ref{nonunicarrier2}.) In this case, the minimal carrier $W$ is given by $W_n=\mathbf{1}_{\{n\mbox{ odd}\}}$. Other carriers are given by $W+N$ for any $N\in\mathbb{Z}_+$. However, if we define the BBS via the update rule (\ref{twosidedupdaterule}), then all of these carriers give rise to the configuration $1-\eta$, and this in turn is equal $\lim_{k \to -\infty}T \eta^{[k]}$. This example is somewhat subtle, however. One could interpret the carrier $W+N$ as taking $N$ particles from $-\infty$ to $\infty$. So, although the resulting dynamics are the same, it might be considered reasonable to rule this situation out. Note that this argument applies to any configuration where $W$ does not have flat segments.
\end{rem}

\begin{figure}[!htb]
\vspace{0pt}
\centering
\scalebox{0.3}{\includegraphics{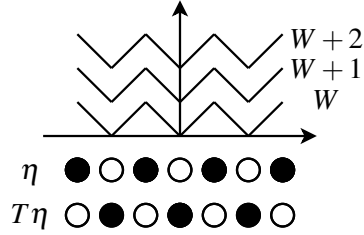}}
\rput(-0,1.7){$W$}
\rput(-0,2.1){$W+1$}
\rput(-0,2.5){$W+2$}
\rput(-3.9,0.7){$\eta$}
\rput(-3.9,0.15){$T \eta$}
\vspace{0pt}
\caption{Example showing multiple carriers giving rise to same dynamics.}\label{nonunicarrier2}
\end{figure}

\subsection{Inverse of the action of the carrier}\label{inversesec} Now we have identified $\mathcal{S}^T=\Phi(\mathcal{Y})$ as the domain where the BBS is well-defined, and presented justification for the dynamics of the system being given by the definitions in Section \ref{extensionsec} (equations (\ref{wtsdef}) and (\ref{tetadef}) in particular), a natural counterpart to study is the inverse of the action of the carrier. Or, to state this as a question, if we see the particle configuration $\eta$ left by the carrier, then can we identify the initial configuration $T^{-1}\eta$? In the finite particle case, this question was answered affirmatively in the original paper of Takahashi and Satsuma \cite{takahashi1990}. Indeed, in this case, they showed that the BBS is reversible, with the inverse dynamics given by the action of a carrier which moves from right to left, rather than left to right, as appears in the original definition. More explicitly, this action can be written
\begin{equation}\label{inverseaction}
T^{-1}\eta=\overleftarrow{T\overleftarrow{\eta}}.
\end{equation}
where $\overleftarrow{\eta}$ is the reversed configuration defined as at (\ref{revcon}). In particular, this identity tells us that, at least for the class of particle configurations considered, the backwards (in time) evolution of the BBS is described as follows: first reverse the configuration space according to \eqref{revcon}, apply the usual (forwards in time) BBS, and then reverse the space again.

To describe the inverse action of the carrier in terms of path encodings, let us first introduce the operator $R:\mathcal{S}^0\rightarrow\mathcal{S}^0$, defined by setting $RS=((RS)_n)_{n\geq 0}$ to be the path given by
\[(RS)_n=-S_{-n}.\]
Note that if $S$ is the path encoding of $\eta$, then
\[(RS)_n-(RS)_{n-1}=1-2\eta_{-(n-1)},\]
and so $RS$ is the path encoding of particle configuration  $\overleftarrow\eta$. Moreover, using this notation, it is clear that the action at \eqref{inverseaction} can be reexpressed as
\begin{equation}\label{inverseRTR}
T^{-1}S=R T RS.
\end{equation}
Alternatively, by applying the definition of $R$ and $T$, we see that this operation describes the dual of Pitman's transformation given by `reflecting in the future minimum'. Indeed, one can check that
\begin{equation}\label{Treverse}
(T^{-1}S)_n=2I_n-S_n-2I_0,
\end{equation}
where we define
\[I_n=\inf_{m \geq n}S_m.\]
In this section, we seek to extend the above observations to the case when the number of particles is infinite. For this, we take (\ref{Treverse}) as a definition of $T^{-1}$, and verify that this can also be written in terms of the formula at (\ref{inverseRTR}), see Lemma \ref{rtrlem}. The main goal of the section, however, is to prove Theorem \ref{mr1}(a), confirming the expression that was stated there for $\mathcal{S}^{rev}$. (Recall $\mathcal{S}^{rev}$, defined at (\ref{srevdef}), is the set upon which the forward and backward dynamics are well-defined and reversible, in that $T^{-1}TS=TT^{-1}S=S$ holds.)

We first introduce the domain of $T^{-1}$, which by applying the obvious symmetry and comparing with Corollary \ref{domaincor} we can suppose is given by
\[\mathcal{S}^{T^{-1}}=\left\{S \in \mathcal{S}^{0}\::\:T^{-1}{S}\mbox{ well-defined}\right\} =\left\{S\in\mathcal{S}^{0}\::\:I_0 \in \mathbb{R}\right\}.\]
Since \[I^{RS}_0=\inf_{n \geq 0} (RS)_n =-\sup_{n \leq 0}S_{n}=-M_0\]
and $R$ is a bijection on $\mathcal{S}^0$, we have $R\mathcal{S}^{T}=\mathcal{S}^{T^{-1}}$ and $R\mathcal{S}^{T^{-1}}=\mathcal{S}^{T}$. We are now ready to check (\ref{inverseRTR}) in our more general setting.

\begin{lem}\label{rtrlem} It holds that $T^{-1}=RTR$ on $\mathcal{S}^{T^{-1}}$.
\end{lem}
\begin{proof} The comments preceding the lemma give that both operators have domain $\mathcal{S}^{T^{-1}}$. Moreover, since $R$ is a bijection on $\mathcal{S}^0$, we only need to show that $RT^{-1}=TR$. This can be done directly as follows:
\begin{eqnarray*}
(RT^{-1}S)_n&=&-(2I_{-n}-S_{-n}-2I_0)\\
&=&-2 \inf_{m \geq -n}S_m+S_{-n}+2\inf_{m \geq 0}S_m\\
&=& 2\sup_{m \leq n}(-S_{-m})-(-S_{-n})-2\sup_{m \leq 0}(-S_{-m})\\
&=&(TRS)_n.
\end{eqnarray*}
\end{proof}

Since the operation of $T^{-1}$ is given by the carrier moving from right to left, it is also natural to consider a carrier path $V$ corresponding to $\eta$, or equivalently $S$. By applying the arguments of Section \ref{extensionsec}, for any $S \in \mathcal{S}^{T^{-1}}$, we can see that $V=S-I$ and $T^{-1}\eta_n=\min \{ 1-\eta_n, V_{n+1}\}$. We note this process is the natural generalisation of $V_n=\sum_{k=n}^{\infty}(\eta_k -T^{-1}\eta_k)$ from the finite to the infinite particle case, where the latter expression was the one that appeared in discussion of the time-reversed dynamics in the original paper of Takahashi and Satsuma \cite{takahashi1990}. Note that $V \in \mathcal{Y}^+$, where
\[\mathcal{Y}^+:=\left\{ Y \in \mathcal{Y}\::\: \liminf_{n \to \infty}Y_n=0\right\}.\]
The various notions introduced in Section \ref{uniquesection} are also naturally extended to their dual versions. In particular, let $\Psi : \mathcal{Y} \to \{0,1\}^{\mathbb{Z}}$ be the map
\[ \Psi Y_n = \mathbf{1}_{\{Y_{n}=Y_{n-1}-1\}}.\]
The map $\Psi: \mathcal{Y} \to \mathcal{S}^0$ that this induces is given by
\[\Psi Y_n =\ell(Y)_n +Y_n-Y_0.\]
We then have the following adaptation of Corollary \ref{domaincor}, Proposition \ref{phiproperty} and Corollary \ref{corafter}.

\begin{prop}\label{adapprop} It holds that $\Psi (\mathcal{Y})=\mathcal{S}^{T^{-1}}$, and $\Psi|_{\mathcal{Y}^+} : \mathcal{Y}^+ \to \mathcal{S}^{T^{-1}}$ is a bijection with inverse map given by $\Psi^{-1}S=V$. Moreover,
\[(\Psi^{-1}S)_n=\min\left\{Y_n\::\:\Psi Y=S,\:Y \in \mathcal{Y} \right\}.\]
Also, $I-I_0=\ell (\Psi^{-1}S)$ for any $S \in \mathcal{S}^{T^{-1}}$.
\end{prop}

With the above notation, for $S \in \mathcal{S}^T$,
\[(T\eta_n)=\mathbf{1}_{\{W_{n}=W_{n-1}-1\}}=(\Psi W)_n =(\Psi \Phi^{-1} \eta)_n,\]
and so $T=\Psi \Phi^{-1}$. In the same way, we have $T^{-1} = \Phi \Psi^{-1}$. From this, we can easily see that
\begin{equation}\label{easily}
T(\mathcal{S}^T)  \subseteq \Psi (\mathcal{Y}) = \mathcal{S}^{T^{-1}},\qquad T^{-1}(\mathcal{S}^{T^{-1}})  \subseteq \Phi (\mathcal{Y}) = \mathcal{S}^{T}.
\end{equation}

We now come to the main result of the section, which gives the characterisation of $\mathcal{S}^{rev}$ stated in Theorem \ref{mr1}(a). We will break the problem into two parts, namely writing $\mathcal{S}^{rev}=\mathcal{S}^{T^{-1}T}\cap\mathcal{S}^{TT^{-1}}$, where
\[\mathcal{S}^{T^{-1}T} =\left\{S \in \mathcal{S}^{T}\::\:T^{-1}TS=S\right\},\qquad
\mathcal{S}^{TT^{-1}}=\left\{S \in \mathcal{S}^{T^{-1}}\::\:TT^{-1}S=S\right\}.\]
For future use, we also summarise equalities involving the future minimum of $TS$ and past maximum of $T^{-1}S$ that arise in the proof.

\begin{thm}\label{goodsetthm}
It holds that
\[\mathcal{S}^{T^{-1}T} =\left\{ S \in \mathcal{S}^{0}\::\:M_0 \in \mathbb{R},\:\limsup_{n \to \infty}S_n=M_{\infty}\right\},\]
\[\mathcal{S}^{TT^{-1}} =\left\{ S \in \mathcal{S}^{0}\::\:I_0 \in \mathbb{R},\: \liminf_{n \to -\infty }S_n=I_{-\infty}\right\},\]
and so
\[\mathcal{S}^{rev}=\left\{S\in \mathcal{S}^0\::\:M_0, I_0\in\mathbb{R},\: \limsup_{n\rightarrow\infty}S_n=M_\infty,\;\liminf_{n\rightarrow-\infty}S_n=I_{-\infty}\right\}.\]
Moreover, for $S\in\mathcal{S}^{T^{-1}T}$,
\[I^{TS}=M-2M_0, \qquad  TV=W,\]
and, for $S\in\mathcal{S}^{TT^{-1}}$,
\[M^{T^{-1}S}=I-2I_0, \qquad T^{-1}W=V,\]
where $I^{TS}_n=\inf_{m \geq n}(TS)_m$ and $M^{T^{-1}S}_n=\sup_{m \leq n}(T^{-1}S)_m$.
\end{thm}

\begin{rem} We can also characterize the reversible set as $\mathcal{S}^{rev}=\Phi(\mathcal{Y}^{rev}) \cap \Psi( \mathcal{Y}^{rev})$ where $\mathcal{Y}^{rev}=\mathcal{Y}^- \cap \mathcal{Y}^+$.
\end{rem}

To prove the above result, we prepare a simple lemma.

\begin{lem}\label{phipsiy+y-} It holds that $\Phi^{-1} \Phi (\mathcal{Y}^+) \subseteq \mathcal{Y}^+$ and $\Psi^{-1} \Psi (\mathcal{Y}^-) \subseteq \mathcal{Y}^-$.
\end{lem}
\begin{proof}
We only give a proof for $\Phi$. Let $Y \in \mathcal{Y}^+$ and $\tilde{Y}:= \Phi^{-1} \Phi Y$. Since $\Phi \tilde{Y}= \Phi Y$, Proposition \ref{phiproperty} gives that $\tilde{Y}_n \le Y_n$ for any $n \in \mathbb{Z}$. In particular, $\liminf_{n \to \infty}\tilde{Y}_n \le \liminf_{n \to \infty}Y_n =0$.
\end{proof}

\begin{proof}[Proof of Theorem \ref{goodsetthm}]
First note that
\[\left\{ S \in \mathcal{S}^{0}\::\:M_0 \in \mathbb{R},\:\limsup_{n \to \infty}S_n=M_{\infty}\right\} = \left\{ S \in \mathcal{S}^{T}\::\: \liminf_{n \to \infty} W_n=0\right\}.\]
Therefore, we only need to show that $\mathcal{S}^{T^{-1}T} = \{ S \in \mathcal{S}^{T}\::\: \Phi^{-1}S \in \mathcal{Y}^+\}$.

We first show that $\mathcal{S}^{T^{-1}T} \subseteq \{ S \in \mathcal{S}^{T}\::\: \Phi^{-1}S \in \mathcal{Y}^+\}$. Let $S \in \mathcal{S}^{T^{-1}T}$. Then, $S=T^{-1}TS$, so $S=\Phi \Psi^{-1} \Psi \Phi^{-1}S$. From (\ref{easily}), we further have that $Y:=\Psi^{-1} \Psi \Phi^{-1}S \in \Psi^{-1}(\mathcal{S}^{T^{-1}})=\mathcal{Y}^+$.
Hence $\Phi^{-1}S=\Phi^{-1}\Phi Y  \in \mathcal{Y}^+$ from Lemma \ref{phipsiy+y-}.

Next, we prove $\mathcal{S}^{T^{-1}T} \supseteq \{ S \in \mathcal{S}^{T}\::\: \Phi^{-1}S \in \mathcal{Y}^+\}$. Assume $S \in \mathcal{S}^{T}$ and $\Phi^{-1}S \in \mathcal{Y}^+$. Since $\Psi|_{\mathcal{Y}^+}$ is a bijection, it is then the case that $\Psi^{-1} \Psi \Phi^{-1}S = \Phi^{-1}S$. Consequently $T^{-1}TS=\Phi \Psi^{-1} \Psi \Phi^{-1}S = \Phi \Phi^{-1}S=S$.

Finally, if $S \in \mathcal{S}^{T^{-1}T}$, then as we have seen $\Psi^{-1} TS= \Psi^{-1} \Psi \Phi^{-1}S = \Phi^{-1}S$ holds. This expression can be rewritten as $TS-I^{TS}=M-S$, which implies $TV=TS-I^{TS}=M-S=W$. In particular, $-I^{TS}_0=M_0$ holds. Applying Corollary \ref{corafter} and Proposition \ref{adapprop}, we further obtain that $I^{TS}-I^{TS}_0=\ell (\Psi^{-1} TS)=\ell(\Phi^{-1}S)=M-M_0$. Hence $I^{TS}=M-M_0+I^{TS}_0=M-2M_0$.

Noting the relation
\[\left\{ S \in \mathcal{S}^{0}\::\:I_0 \in \mathbb{R},\:\liminf_{n \to -\infty}S_n=I_{-\infty}\right\} = \left\{ S \in \mathcal{S}^{T^{-1}}\::\: \liminf_{n \to -\infty} V_n=0\right\},\]
we can prove the rest of the claim in the same way.
\end{proof}

Finally, we summarise some of the duality relations that hold for the various maps that we have introduced. To state the result, we further define $\tilde{R}: \mathcal{Y} \to \mathcal{Y}$ by setting $(\tilde{R}Y)_n=Y_{-n}$. The proof, which is straightforward, is omitted.

\begin{lem}\label{dualityrel} It holds that
\[ R\Psi=\Phi\tilde{R}, \qquad R\Phi=\Psi\tilde{R}.\]
Moreover, $\tilde{R}(\mathcal{Y}^-)=\mathcal{Y}^+$, $\tilde{R}(\mathcal{Y}^+)=\mathcal{Y}^-$ and the maps in the following diagram are all bijections and commutative.
\[\xymatrix@C+30pt@R+10pt
{      \mathcal{Y}^- \ar@<.5ex>[r]^\Phi \ar[d]  &  \mathcal{S}^T \ar[d]^-{R} \ar@<.5ex>[l]^{\Phi^{-1}}\\
                  \mathcal{Y}^+ \ar[u]^-{\tilde{R}} \ar@<.5ex>[r]^\Psi   &  \mathcal{S}^{T^{-1}} \ar[u] \ar@<.5ex>[l]^{\Psi^{-1}}
                  }\]
Also, the following diagram satisfies the same property.
\[\xymatrix@C+30pt@R+10pt
{       \mathcal{Y}^{rev} \ar@<.5ex>[r]^-\Phi \ar[d]  &  \mathcal{S}^{rev} \ar[d]^-{R} \ar@<.5ex>[l]^-{\Phi^{-1}}\\
                  \mathcal{Y}^{rev}  \ar[u]^-{\tilde{R}} \ar@<.5ex>[r]^-\Psi   &  \mathcal{S}^{rev} \ar[u] \ar@<.5ex>[l]^-{\Psi^{-1}}
                  }\]
\end{lem}

{\rem \label{needcond} As an example of a particle configuration for which the BBS dynamics are well-defined, but whose path encoding does not satisfy $T^{-1}TS=S$, or indeed $TT^{-1}S=S$, consider ${\eta}$ as given by $\eta_n=\mathbf{1}_{\{n\leq 0,\:n\mbox{ even}\}}+\mathbf{1}_{\{n\geq 1,\:n\mbox{ odd}\}}$. Then we find that the particle configuration encoded by $T^{-1}TS$ is equal to ${\eta}'$, where ${\eta}'_n=\mathbf{1}_{\{n\leq -2,\:n\mbox{ even}\}}+\mathbf{1}_{\{n\geq 1,\:n\mbox{ odd}\}}$, and the particle configuration encoded by $TT^{-1}S$ is equal to ${\eta}''$, where ${\eta}''_n=\mathbf{1}_{\{n\leq 0,\:n\mbox{ even}\}}+\mathbf{1}_{\{n\geq 3,\:n\mbox{ odd}\}}$. (The relevant path encodings are shown in Figure \ref{needcondfig}.) In essence, the operation ${S}\mapsto TS$ sends one particle (from 0) to $\infty$, and this is not recovered by $T^{-1}$. Similarly, ${S}\mapsto T^{-1}S$ sends one particle (from 1) to $-\infty$. The conditions required for a function $S$ to be in $\mathcal{S}^{T^{-1}T}$ or $\mathcal{S}^{TT^{-1}}$ prevent this happening, by ensuring that the carrier must empty itself infinitely often in the relevant direction.

\begin{figure}[!htb]
\vspace{0pt}
\centering
\scalebox{0.2}{\includegraphics{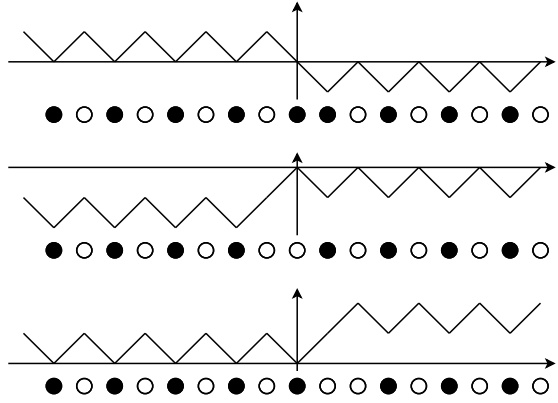}}
\vspace{0pt}
\caption{Path encodings of ${S}$, $T^{-1}TS$ and $TT^{-1}S$, as described in Remark \ref{needcond}.}\label{needcondfig}
\end{figure}

Continuing from the discussion in Remark \ref{noninirem}, we note that it is possible to recover the missing particle from a reservoir at $\pm\infty$ by using a different carrier. For instance, two possible carriers for $T^{-1}S$, $W^{T^{-1}S}$ and $Y^{T^{-1}S}$ say, are shown in Figure \ref{needcondfig2}. If we use $W^{T^{-1}S}$ (in the right-hand side of \eqref{twosidedupdaterule}), then we arrive at $TT^{-1}S$ as described above. If we use $Y^{T^{-1}S}$, then we recover the missing particle, and return $S$. More generally, for a carrier $Y\in \mathcal{Y}^{\pm}$, $\liminf_{n\rightarrow\pm\infty} Y_n$ represents the number of particles returned from $\pm\infty$, and $\liminf_{n\rightarrow\mp\infty} Y_n$ represents the number of particles carried to $\mp\infty$. Such observations could motivate a more general model with reservoirs at infinity, where any carrier $Y\in\mathcal{Y}$ is allowed to determine the dynamics, but we do not pursue that here.

\begin{figure}[!htb]
\vspace{0pt}
\centering
\scalebox{0.2}{\includegraphics{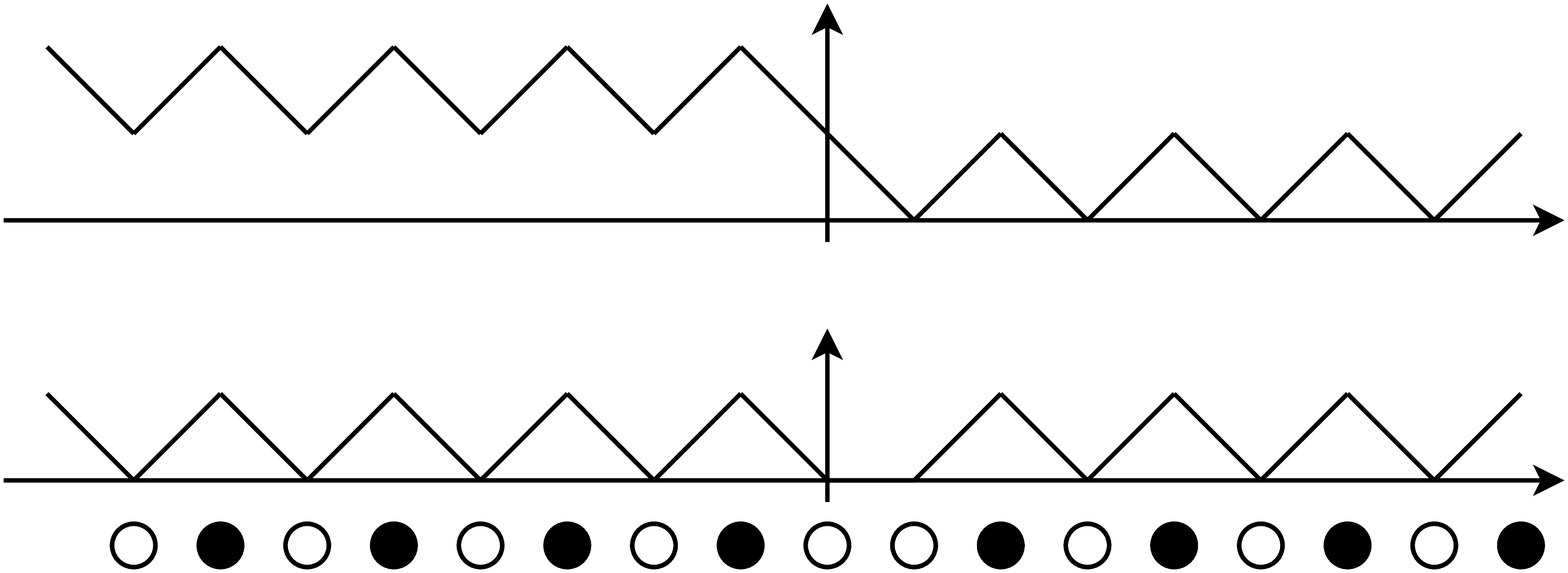}}
\vspace{0pt}
\caption{Carriers for $T^{-1}S$, as described in Remark \ref{needcond}. (The lower figure shows $W^{T^{-1}S}$, and the upper shows $Y^{T^{-1}S}$.)}\label{needcondfig2}
\end{figure}}

\subsection{Invariant set of initial conditions}\label{invariantsetsec} The aim of this section is to characterise the invariant set of initial configurations, $\mathcal{S}^{inv}$ (defined at (\ref{sinv})), for which repeated passes of the carrier are possible from left and right and they are compatible. In particular, we will complete the proof of Theorem \ref{mr1}.

To describe the set $\mathcal{S}^{inv}$, we start by introducing some definitions. For any strictly increasing function $F : \mathbb{Z} \to \mathbb{R}$ satisfying $\lim_{n \rightarrow \infty} F(n)=\infty$, we introduce the following subset of $\mathcal{S}^0$,
\begin{equation}\label{sfudef}
\mathcal{S}^+_F:=\left\{S\in\mathcal{S}^0:\:\:\lim_{n\rightarrow \infty}\frac{S_n}{F(n)}=1 \right\},
\end{equation}
and for any strictly increasing function $F : \mathbb{Z} \to \mathbb{R}$ satisfying $\lim_{n \rightarrow -\infty} F(n)=-\infty$,
\begin{equation}\label{sfldef}
\mathcal{S}^-_F:=\left\{S\in\mathcal{S}^0:\:\:\lim_{n\rightarrow -\infty}\frac{S_n}{F(n)}=1 \right\}.
\end{equation}
Also, for any $K\in\mathbb{N}$, we introduce the following subsets of $\mathcal{S}^0$,
\[\mathcal{S}^+_K:=\left\{S \in \mathcal{S}^0\::\:  \sup_{n \in \mathbb{Z}} (\sup_{m \le n}S_m -\inf_{m \ge n}S_m) =K, \: \limsup_{n\to \infty}S_n - \liminf_{n \to \infty}S_n=K \right\},\]
\[\mathcal{S}^-_K:=\left\{S \in \mathcal{S}^0\::\: \sup_{n \in \mathbb{Z}} (\sup_{m \le n}S_m -\inf_{m \ge n}S_m) =K, \: \limsup_{n\to -\infty}S_n - \liminf_{n \to -\infty}S_n=K \right\}.\]
Note that $S \in\mathcal{S}^+_F$ for some $F$ only if $\lim_{n \rightarrow \infty}S_n=\infty$, and $S \in\mathcal{S}^+_K$ for some $K$ only if it holds that $\limsup_{n \rightarrow \infty}S_n <\infty$. Moreover, $S \in\mathcal{S}^+_K$ for some $K$ only if $M_0, I_0 \in \mathbb{R}$. The same is true for $\mathcal{S}^-_F$ and $\mathcal{S}^-_K$.

The main result of the subsection is the following characterization of the invariant set. Clearly this gives Theorem \ref{mr1}(b).

\begin{thm}\label{characterizationinv} For $S \in \mathcal{S}^0$, $S \in \mathcal{S}^{inv}$ if and only if $S \in \mathcal{S}^-_{*_1} \cap \mathcal{S}^+_{*_2}$, where $*_1$ and $*_2$ are some $F$ or $K$. Moreover, if the condition holds, then $T^kS \in \mathcal{S}^-_{*_1} \cap \mathcal{S}^+_{*_2}$ for any $k\in \mathbb{Z}$.
\end{thm}

{\rem The final statement of the above theorem can be understood as meaning the BBS preserves the asymptotic density profile. Indeed, if $S\in \mathcal{S}_F^+$, then
\[n^{-1}\sum_{m=0}^{n-1}(T^k\eta)_m\approx\frac12(1-n^{-1}F(n))\]
for each $k\in\mathbb{Z}$. In particular, if $F(n)=cn$ for some $c\in(0,1]$, then we have a constant asymptotic density given by $\frac12(1-c)$. Furthermore, if $S\in \mathcal{S}_K^+$, then $n^{-1}\sum_{m=0}^{n-1}(T^k\eta)_m\rightarrow\frac12$ for all $k\in\mathbb{Z}$.}

\begin{rem} We can also characterize the invariant set in terms of the carrier by the identity $\mathcal{S}^-_{*_1} \cap \mathcal{S}^+_{*_2} = \Phi(\mathcal{Y}^-_{*_1} \cap \mathcal{Y}^+_{*_2}) \cap \Psi(\mathcal{Y}^-_{*_1} \cap \mathcal{Y}^+_{*_2}) $  where $*_1$ and $*_2$ are critical or sub-critical where
\[\mathcal{Y}^{\pm}_{sub-critical}:=\left\{Y \in \mathcal{Y}^{rev} \:: \: \lim_{n \to \pm \infty}\frac{Y_n}{\ell(Y)_n}=0\right\},\]
\[\mathcal{Y}^{\pm}_{critical}:=\left\{Y \in \mathcal{Y}^{rev} \:: \:  \lim_{n \to \pm \infty}|\ell(Y)_n| < \infty,\: \limsup_{n \to \pm \infty}Y_n=\sup_n Y_n <\infty \right\}.\]
\end{rem}

The proof of Theorem \ref{characterizationinv} is divided into four lemmas.

\begin{lem}\label{cha:goodset} It holds that $\mathcal{S}^+_K \subseteq \mathcal{S}^{T^{-1}T}$ and $\mathcal{S}^-_K \subseteq \mathcal{S}^{TT^{-1}}$. Also, $\mathcal{S}^-_{*_1} \cap \mathcal{S}^+_{*_2} \subseteq \mathcal{S}^{rev}$ for any $*_1$ and $*_2$.
\end{lem}

\begin{lem}\label{cha:invariant} The set $\mathcal{S}^-_{*_1} \cap \mathcal{S}^+_{*_2} $ is invariant under $T$ and $T^{-1}$ for any $*_1$ and $*_2$.
\end{lem}

\begin{lem}\label{cha:key1} The following inclusions hold:
\[\mathcal{S}^{inv} \cap \left\{S \in \mathcal{S}^0\::\: \limsup_{n \to \infty}S_n=\infty \right\} \subseteq \bigcup_{F}\mathcal{S}^+_F,\]
\[\mathcal{S}^{inv} \cap \left\{S \in \mathcal{S}^0\::\: \liminf_{n \to -\infty}S_n=-\infty \right\} \subseteq \bigcup_{F}\mathcal{S}^-_F.\]
\end{lem}

\begin{lem}\label{cha:key2} The following inclusions hold:
\[\mathcal{S}^{inv} \cap \left\{S \in \mathcal{S}^0\::\: \limsup_{n \to \infty}S_n<\infty \right\} \subseteq \bigcup_{K}\mathcal{S}^+_K,\]
\[\mathcal{S}^{inv} \cap \left\{S \in \mathcal{S}^0\::\: \liminf_{n \to -\infty}S_n > -\infty \right\} \subseteq \bigcup_{K}\mathcal{S}^-_K.\]
\end{lem}

Before giving proofs of the above lemmas, we prove Theorem \ref{characterizationinv} assuming that they hold.

\begin{proof}[Proof of Theorem \ref{characterizationinv}]
From Lemmas \ref{cha:key1} and \ref{cha:key2}, for any $S \in \mathcal{S}^{inv}$, there exists some $*_1$ and $*_2$ such that $S \in \mathcal{S}^-_{*_1} \cap \mathcal{S}^+_{*_2}$. Conversely, if $S \in \mathcal{S}^-_{*_1} \cap \mathcal{S}^+_{*_2}$ for some $*_1$ and $*_2$, then Lemma \ref{cha:invariant} gives $T^kS \in \mathcal{S}^-_{*_1} \cap \mathcal{S}^+_{*_2}$ for any $k \in \mathbb{Z}$. Since Lemma \ref{cha:goodset} gives $\mathcal{S}^-_{*_1} \cap \mathcal{S}^+_{*_2} \subset \mathcal{S}^{rev}$, we deduce that $T^kS \in \mathcal{S}^{rev}$ for any $k \in \mathbb{Z}$. Hence $S \in \mathcal{S}^{inv}$.
\end{proof}

\begin{proof}[Proof of Lemma \ref{cha:goodset}] Since $S \in \mathcal{S}^+_K$ implies $M_0 \in \mathbb{R}$, to check that $S\in \mathcal{S}^{T^{-1}T}$, we only need to show that $\limsup_{n \rightarrow \infty}S_n=\sup_n S_n$. Suppose $\limsup_{n \rightarrow \infty}S_n < \sup_n S_n$. Thus there exists an $n_0$ such that $S_n \le S_{n_0}-1$ for all $n \ge n_0+1$. Since $S \in \mathcal{S}^+_K$, $S_{n_0}-I_{n_0} \le M_{n_0}-I_{n_0} \le K$. Therefore, $\limsup_{n \rightarrow \infty}S_n -\liminf_{n \rightarrow \infty}S_n \le S_{n_0}-1-I_{n_0} \le K-1$, which contradicts the condition that $\limsup_{n \rightarrow \infty}S_n -\liminf_{n \rightarrow \infty}S_n=K$. The claim that $\mathcal{S}^-_K \subseteq \mathcal{S}^{TT^{-1}}$ is shown in the same way.

Next, we prove $\mathcal{S}^-_{*_1} \cap \mathcal{S}^+_{*_2} \subset \mathcal{S}^{rev}$. For $\mathcal{S}^-_{K} \cap \mathcal{S}^+_{K'}$, it is clear from the fact just we have shown. If $S \in  \mathcal{S}^-_{F}\cap\mathcal{S}^+_{K}$, then $M_0,I_0 \in \mathbb{R}$ and $\limsup_{n \rightarrow \infty}S_n=\sup_n S_n$. So, we only need to show that $\liminf_{n \rightarrow -\infty}S_n=\inf_n S_n$ but since $\lim_{n \rightarrow -\infty}S_n=-\infty$, this is clear. The case $\mathcal{S}^-_{K}\cap \mathcal{S}^+_{F}$ is exactly the same. Finally, if $S \in \mathcal{S}^-_{F} \cap \mathcal{S}^+_{F'}$, then $\lim_{n \rightarrow \pm \infty}S_n=\infty$, so $M_0,I_0 \in \mathbb{R}$ and the other conditions also clearly hold.
\end{proof}

\begin{rem} It is obvious that $\mathcal{S}^-_{K} \cap \mathcal{S}^+_{K'} \neq \emptyset$ if and only if $K=K'$.
\end{rem}

\begin{proof}[Proof of Lemma \ref{cha:invariant}] We first consider the case $S \in \mathcal{S}^-_{*}\cap \mathcal{S}^+_F$. For any $S \in \mathcal{S}_F^+$ and $\epsilon >0$, there exists $N_{\epsilon}$ so that for any $n \ge N_{\epsilon}$, $(1-\epsilon) F(n) \le S_n \le (1+\epsilon)F(n)$. In particular, for any $n \ge N_{\epsilon}$,
$(1-\epsilon)F(n) \le I_n \le (1+\epsilon)F(n)$ and so $\lim_{n\rightarrow \infty}\frac{I_n}{F(n)}=1$. Moreover, for $S \in \mathcal{S}^-_{*}\cap\mathcal{S}^+_F $, we have $M_0 \in \mathbb{R}$ and $\lim_{n \to \infty} S_n=\infty$. Thus there exists an increasing subsequence $\{n_k\}_{k \in \mathbb{N}}$ satisfying $n_k \to \infty \ (k\to \infty)$ and $S_{n_k}=M_{n_k}$ for all $k$. For each $\epsilon$, let $L_{\epsilon}:=\min\{n_k:\:k \in \mathbb{N},\:n_k \ge N_{\epsilon}\}$. Then, for any $n \ge L_{\epsilon}$, $(1-\epsilon)F(n) \le S_n \le M_n$ and
\[M_n = \max\left\{ M_{L_{\epsilon}},\max\{ S_k:\: L_{\epsilon} \le k \le n\}\right\} = \max\{S_k:\: \ L_{\epsilon} \le k \le n\} \le (1+\epsilon)F(n).\]
Hence $\lim_{n\rightarrow \infty}\frac{M_n}{F(n)}=1$. Therefore,
\[\lim_{n\rightarrow  \infty}\frac{TS_n}{F(n)}=\lim_{n\rightarrow \infty}\frac{2M_n-S_n-2M_0}{F(n)}=1,\]
and
\[\lim_{n\rightarrow \infty}\frac{T^{-1}S_n}{F(n)}=\lim_{n\rightarrow \infty}\frac{2I_n-S_n-2I_0}{F(n)}=1.\]
Therefore, $TS, T^{-1}S \in \mathcal{S}^+_F$. In the same way, if $S \in \mathcal{S}^-_{F}\cap\mathcal{S}^+_*$, then $TS, T^{-1}S \in \mathcal{S}^-_{F}$.

Next, we prove that $\mathcal{S}^+_K$ is invariant under $T$ and $T^{-1}$. Let $S \in \mathcal{S}^+_K$. Since $\limsup_{n \rightarrow \infty}S_n \in \mathbb{R}$ and $\liminf_{n \rightarrow \infty}S_n \in \mathbb{R}$ and $\limsup_{n \rightarrow \infty}S_n=M_{\infty}$ (from the proof of Lemma \ref{cha:goodset}), there exists an $n_0$ such that $M_n=M_{\infty}$ and $I_n=I_{\infty}$ for all $n \ge n_0$. Therefore,
\[\limsup_{n \rightarrow \infty}TS_n = \limsup_{n \rightarrow \infty} (2M_n -S_n-2M_0)= 2M_{\infty}-2M_0-\liminf_{n \rightarrow \infty}S_n= 2M_{\infty}-2M_0-I_{\infty}\]
and similarly $ \liminf_{n \rightarrow \infty}TS_n=2M_{\infty}-2M_0-M_{\infty}$. Therefore,
\[\limsup_{n \rightarrow \infty}TS_n -\liminf_{n \rightarrow \infty}TS_n = M_{\infty}-I_{\infty}=K.\]
In the same way, $\limsup_{n \rightarrow \infty}T^{-1}S_n -\liminf_{n \rightarrow \infty}T^{-1}S_n=K$. Also,
\begin{eqnarray*}
\sup_{m \le n}TS_m &=& \sup_{m \le n} (2M_m -S_m-2M_0) \\
&=& -2M_0+\sup_{m \le n}(M_m +M_m-S_m)\\
& \le& -2M_0+M_n +\sup_{m \le n} (M_m-I_m)\\
& \le& -2M_0+M_n +K,
\end{eqnarray*}
and
\[
\inf_{m \ge n}TS_m = \inf_{m \ge n} (2M_m -S_m-2M_0) = -2M_0+\inf_{m \ge n}(M_m +M_m-S_m) \ge -2M_0+M_n.
\]
Therefore, $\sup_{m \le n}TS_m -\inf_{m \ge n}TS_m \le -2M_0+M_n +K - ( -2M_0+M_n) =K$. Moreover, letting $n_1\geq n_0$ be such that $S_{n_1}=I_\infty$, we have that
\[\sup_{m \le n_1}TS_m -\inf_{m \ge n_1}TS_m \geq (2M_{n_1}-S_{n_1}-2M_0)-\inf_{m\geq n_1}(2M_\infty-S_n-2M_0)=M_\infty-I_\infty=K.\]
In a similar way, we can show that $\sup_{m \le n}T^{-1}S_m -\inf_{m \ge n}T^{-1}S_m =K$. Therefore, $TS, T^{-1}S \in \mathcal{S}^+_K$. The proof of the invariance of $\mathcal{S}^-_K$ under $T$ and $T^{-1}$ is the same as that for  $\mathcal{S}^+_K$.
\end{proof}

\begin{proof}[Proof of Lemma \ref{cha:key1}] We first check that if $S \in \mathcal{S}^{inv} \cap \{S \in \mathcal{S}^0: \limsup_{n \to \infty}S_n=\infty \}$, then $\lim_{n \to \infty}S_n=\infty$. Actually, if we suppose $S  \in \mathcal{S}^{inv}$, $\limsup_{n \to \infty}S_n=\infty$ and $\liminf_{n \to \infty}S_n =I_{\infty}< \infty$, then
\[\inf_{n \ge 0}(T^{-1}S)_n =\inf_{n \ge 0} (2I_n -S_n -2I_0) \le \inf_{n \ge 0} (2I_{\infty} -S_n -2I_0)= 2I_{\infty}-2I_0 -\sup_{n \ge 0}S_n =-\infty.\]
This implies $T^{-1}S \notin \mathcal{S}^{T^{-1}}$, and so $S \notin \mathcal{S}^{inv}$ in particular.

So, we only need to show that $\mathcal{S}^{inv} \cap \{S \in \mathcal{S}^0: \lim_{n \to \infty}S_n=\infty\} \subseteq \cup_F \mathcal{S}^+_F$. This will be done if we can prove for any $S \in \mathcal{S}^{inv} \cap \{S \in \mathcal{S}^0: \lim_{n \to \infty}S_n=\infty\} $, $S \in \mathcal{S}^+_M$ where $M_n=\sup_{m \le n}S_m$, since $M$ is increasing and under the assumption, $\lim_{n \to \infty}M_n=\infty$. We will show that this is indeed the case in Lemma \ref{maxlem}.
\end{proof}

\begin{lem}\label{maxlem} It holds that
\[\mathcal{S}^{inv} \cap \left\{S \in \mathcal{S}^0: \lim_{n \to \infty}S_n=\infty\right\}  \subseteq \left\{S \in \mathcal{S}^0\::\: \ \lim_{n \to \infty} S_n =\infty,\:M_0 <\infty,\: \lim_{n \to \infty}\frac{M_n}{I_n} =1 \right\}.\]
\end{lem}
\begin{proof} Let $S \in \mathcal{S}^{inv} \cap \{S \in \mathcal{S}^0: \lim_{n \to \infty}S_n=\infty\}$. For the argument, it will be convenient to note that $\lim_{n\to\infty}T^{-k}S_n=\infty$ for all $k\geq 0$. Indeed Theorem \ref{goodsetthm} yields $M^{T^{-1}S}_n=I_n^S-2I_0^S\rightarrow\infty$, and so $\limsup_{n\to\infty}T^{-1}S_n=\infty$. Hence by the argument used in the proof of Lemma \ref{cha:key1}, we have that $\lim_{n\to\infty}T^{-1}S_n=\infty$. Iterating this establishes the claim.

To complete the proof, we will suppose that $\limsup_{n \to \infty}{M_n}/{I_n} \ge 1+\delta$ for some $\delta \in(0,1)$ and show that it gives a contradiction. Under the latter assumption, for any $0 < \delta' <\delta$, there exists an increasing subsequence $\{n_k\}_{k \in \mathbb{N}}$, $n_k \to \infty \ (k \to \infty)$ such that for all $k \in \mathbb{N}$,
\[ \frac{M_{n_k}}{I_{n_k}} > 1+\delta', \qquad I_{n_k}>0.\]
Define $\tilde{n}_1 = \sup \{ n \in \mathbb{Z}:\: n \le n_1, \: S_{n}=M_{n_1} \}$ and $\tilde{n}_k = \sup \{ n \in \mathbb{Z}:\: \tilde{n}_{k-1} \le n \le n_k, \:   S_{n}=M_{n_k} \}$ recursively.
Then, $\tilde{n}_k$ is non-decreasing and satisfies $\tilde{n}_k \to \infty \ (k \to \infty)$ since $M_{n_k} \to \infty$. In particular, for large enough $k$, $I_{\tilde{n}_k} >0$, and so we can assume $I_{\tilde{n}_k} >0$ for all $k$. By the definition of $\tilde{n}_k$, we have $S_{\tilde{n}_k}=M_{n_k}$ and $I_{\tilde{n}_k} \le I_{n_k}$. Therefore, for all $k$, it holds that $S_{\tilde{n}_k}/I_{\tilde{n}_k} > 1+\delta'$, and $I_{\tilde{n}_k}>0$. Hence
\[I^{T^{-1}S}_{\tilde{n}_k}=\inf_{n \ge \tilde{n}_k} (T^{-1}S)_n =\inf_{n \ge \tilde{n}_k} ( 2I_n-S_n-2I_0 ) \le 2I_{\tilde{n}_k}-S_{\tilde{n}_k}-2I_0 \le (1-\delta') I_{\tilde{n}_k} -2I_0.\]
From this, we conclude
\begin{equation}\label{stat}
\liminf_{n \to \infty} \frac{I^{T^{-1}S}_n}{I_n} \le \liminf_{k \to \infty} \frac{I^{T^{-1}S}_{\tilde{n}_k}}{I_{\tilde{n}_k}} \le 1-\delta'.
\end{equation}
Since $\delta'\in(0, \delta)$ was arbitrary, this implies $\liminf_{n \to \infty} \frac{I^{T^{-1}S}_n}{I_n} \le 1-\delta$. Recalling again from Theorem \ref{goodsetthm} that $M^{T^{-1}S}_n=I^S_n-2I^S_0$, this gives
\[ \limsup_{n \to \infty} \frac{M^{T^{-1}S}_n}{I^{T^{-1}S}_n}=\limsup_{n \to \infty} \frac{I^S_n}{I^{T^{-1}S}_n}\times\frac{I^S_n-2I^S_0}{I^S_n}=\limsup_{n \to \infty} \frac{I^S_n}{I^{T^{-1}S}_n} \ge \frac{1}{1-\delta} > 1.\]
Following the same argument recursively, we thus obtain that
\[ \limsup_{n \to \infty} \frac{M^{T^{-k}S}_n}{I^{T^{-k}S}_n} \ge 1+\delta_k,\]
where $\delta_0=\delta$ and $1+\delta_k=\frac{1}{1-\delta_{k-1}}$ for $k\geq 1$. From this relation, $\{\delta_k\}$ is well-defined and strictly increasing up to $k_0:=\inf\{l:\delta_{l}>1\}$. If $k_0=\infty$, then $\delta_{\infty}:=\lim_{k \to \infty}\delta_k\in[\delta,1]$ must solve $1+\delta_{\infty}=\frac{1}{1-\delta_{\infty}}$. However, the latter equation implies $\delta_{\infty}=0$, which is a contradiction to $\delta_{\infty} \ge \delta$. Thus, we must have that $k_0<\infty$, and $\delta_{k_0}>1$. Now, repeating the part of the argument leading to (\ref{stat}), we find that
\[\liminf_{n \to \infty} \frac{I^{T^{-(k_0+1)}S}_n}{I^{T^{-k_0}S}_n} \le 1-\delta_{k_0}<0.\]
However, this contradicts the fact that both $I^{T^{-k_0}S}_n$ and $I^{T^{-(k_0+1)}S}_n$ diverge (to $+\infty$), as was noted in the previous paragraph. Thus the proof is complete.
\end{proof}

\begin{proof}[Proof of Lemma \ref{cha:key2}] For any $S \in \mathcal{S}^{inv} \cap \{S \in \mathcal{S}^0: \limsup_{n \to \infty}S_n < \infty\}$, it is clear that  $I_{\infty}, M_{\infty} \in \mathbb{R}$. Moreover, $\liminf_{n \to \infty} S_n=I_\infty$ is obvious by definition, and $\limsup_{n \to \infty} S_n =M_{\infty}$ holds by Theorem \ref{goodsetthm}. We define $K=M_{\infty}-I_{\infty}$ and $L:=\sup_{n \in \mathbb{Z}} (M_n-I_n)$, where a priori $L$ can be $\infty$. It is clear that $L\geq K$, because $M_n=M_\infty$, $I_n=I_\infty$ for large $n$. Thus, to establish the lemma, we only need to show that $L \le K$.

To this end, we first note that the argument used in the proof of Lemma \ref{cha:invariant} yields, for every $k\in\mathbb{Z}$,
\[\limsup_{n \to \infty}T^kS_n-\liminf_{n \to \infty}T^kS_n=K.\]
Here, since $S \in \mathcal{S}^{inv}$, $T^kS$ is well-defined for all $k$.

Next, we consider $\tilde{M}_n:=M_{\infty}-M_n$ and consider the pair $(\tilde{M}_n, W_n) = (\tilde{M}_n, M_n-S_n)$. If there exists an $n_0$ such that $\tilde{M}_{n_0}+ K < W_{n_0}$, then
\begin{align*}
TS_{n_0} -\limsup_{n \to \infty}TS_n & = 2M_{n_0}-S_{n_0} - \limsup_{n \to \infty} (2M_n -S_n) \\
&=M_{n_0}+W_{n_0}-M_{\infty}-(M_{\infty}-I_{\infty})\\
&=W_{n_0}-\tilde{M}_{n_0}-K\\
& >0,
\end{align*}
and so $TS \notin \mathcal{S}^{rev}$ (by Theorem \ref{goodsetthm}). This implies $S \notin \mathcal{S}^{inv}$. Thus, in combination with the conclusion of the previous paragraph, we have established that if $S \in  \mathcal{S}^{inv} \cap \{S \in \mathcal{S}^0: \limsup_{n \to \infty}S_n < \infty\}$, then for any $k \in \mathbb{Z}$, $T^kS$ must satisfy $\tilde{M}^{T^kS}_n+ K \ge W^{T^kS}_n$ for all $n \in \mathbb{Z}$, or equivalently $K \ge \sup_{n \in \mathbb{Z}}(W^{T^kS}_n-\tilde{M}^{T^kS}_n)$.

Our next goal is to show that if $L >K$, then we can construct an increasing sequence $\{n_k\}$ such that
\[W_{n_{k+1}}^{T^{k+1}S}-\tilde{M}_{n_{k+1}}^{T^{k+1}S} \ge W_{n_k}^{T^kS}-\tilde{M}_{n_k}^{T^kS}  +1.\]
If we have this, then $\sup_{n \in \mathbb{Z}}(W^{T^kS}_n-\tilde{M}^{T^kS}_n) > K$ for large enough $k$, and hence $S \notin \mathcal{S}^{inv}$. Suppose $L > K$. It is then the case that there exists an $\tilde{n}_0$ such that $M_{\tilde{n}_0}-I_{\tilde{n}_0} \ge K+1$, and so there exists an $n_0 \ge \tilde{n}_0$ such that $W_{n_0}=M_{n_0}-S_{n_0} \ge K+1$. Thus, since
\[M^{TS}_{n_0} \ge (TS)_{n_0}=2M_{n_0}-S_{n_0}-2M_0=M_{n_0}+W_{n_0}-2M_0\]
and $I^{TS}_{n_0}=M_{n_0}-2M_0$ (again by Theorem \ref{goodsetthm}), we obtain $M^{TS}_{n_0} -I^{TS}_{n_0} \ge W_{n_0} \ge K+1$. From this, we deduce there exists an $n_1 \ge n_0$ such that $W^{TS}_{n_1}=M^{TS}_{n_1}-(TS)_{n_1} \ge K+1$. Moreover,
\begin{align*}
\tilde{M}^{TS}_{n_1}&=M_{\infty}^{TS}-M^{TS}_{n_1} \\
&\le (M_{\infty}^{TS}-I_{\infty}^{TS})+I_{\infty}^{TS}-M^{TS}_{n_0} \\
&\leq K+(M_{\infty}-2M_0)-(2M_{n_0}-S_{n_0}-2M_0)\\
&=K+\tilde{M}_{n_0}-W_{n_0}.
\end{align*}
Therefore,
\[W^{TS}_{n_1}-\tilde{M}^{TS}_{n_1} \ge K+1 - (K+\tilde{M}_{n_0}-W_{n_0}) =W_{n_0}-\tilde{M}_{n_0}+1.\]
Since $W^{TS}_{n_1} \ge K+1$, we can repeat the procedure to obtain an $n_2 \ge n_1$ such that $W_{n_2} \ge K+1$ and
\[W^{T^2S}_{n_2}-\tilde{M}^{T^2S}_{n_2} \ge W^{TS}_{n_1}-\tilde{M}^{TS}_{n_1}+1,\]
and so we are able to find the desired sequence ${n_k}$. Thus the proof is complete.
\end{proof}

\subsection{Correspondence between particle configuration and current}\label{currentsec}

As we observed in the introduction, $W_0$ represents the number of particles moved by the carrier across the origin on the first evolution of the BBS (or, more precisely, from $\{\dots,-1,0\}$ to $\{1,2,\dots\}$). Similarly, $(T^{k-1}W)_0$ represents the particles moved by the carrier across the origin on the $k$th evolution of the system. It is natural to ask how much information about the initial particle configuration $\eta$ we can extrapolate from observing the particle current  $((T^kW)_0)_{k\in\mathbb{Z}}$. The goal of this section is to demonstrate that, at least when restricted to a suitable domain $\mathcal{S}_{sub-critical}^*\subseteq\mathcal{S}_{sub-critical}$ (see \eqref{ssubcritstar} for a precise definition) and corresponding codomain, the map from $(\eta_n)_{n\in\mathbb{Z}}$ to $((T^kW)_0)_{k\in\mathbb{Z}}$ is a bijection. We note that the proof of this result depends on the simple condition under which $(\eta_n)_{n\leq 0}$ can be recovered from $((T^kW)_0)_{k\geq0}$ that is provided by Proposition \ref{infoddlem}. Note that, until we note otherwise, we restrict our attention to configurations $\eta$ for which $S\in\mathcal{S}^{inv}$ so that the BBS dynamics are well-defined for all time. The results of this section will be useful in Section \ref{invarsec} when studying the properties of invariant measures and the issue of ergodicity for the map $\eta\mapsto T\eta$.

For convenience, in this section we will use the notation
\[\eta^k_n:=(T^k\eta)_n,\qquad w^k_n:=(T^kW)_n,\qquad \forall k,n\in\mathbb{Z}.\]
Now, viewing the BBS as a cellular automaton, we can describe the dynamics as in the following diagram.
\[\xymatrix@C-15pt@R-15pt{ & \eta^{k+1}_n & \\
            w_{n-1}^k \ar[rr]& & w_{n}^k\\
             & \eta^k_n \ar[uu]&}\]
That is, we have a system with input $(\eta^k_n,w^k_{n-1})$, which returns an output of $(\eta^{k+1}_n,w^k_{n})$, where the value of $\eta^{k+1}_n$ is given by the BBS update rule of (\ref{twosidedupdaterule}), and the value of $w^k_{n}$ is given by the update rule at (\ref{wupdate}). In particular, there are three basic patterns that can appear in this system, which can be represented as follows.
\begin{equation}\label{patterns}
\xymatrix@C-15pt@R-15pt{ & 0 & \\
            0 \ar[rr]& & 0\\
             & 0 \ar[uu]&}\qquad
\xymatrix@C-15pt@R-15pt{ & 1 & \\
            w>0 \ar[rr]& & w-1\\
             & 0 \ar[uu]&}\qquad
\xymatrix@C-15pt@R-15pt{ & 0 & \\
            w \ar[rr]& & w+1\\
             &  1\ar[uu]&}
\end{equation}
From this simple observation, we obtain the following result.

\begin{lem}\label{infzerolem}
(a) If $w^{k}_{n}=0$ for some $k\in\mathbb{Z}_+$, $n\in\mathbb{Z}$, then the values of $(\eta^l_{n})_{l=0}^k$ are uniquely determined by $(w^l_{n})_{l=0}^k$.\\
(b) If $w^{k}_{n}=0$ infinitely often as $k\rightarrow\infty$ for some $n\in\mathbb{Z}$, then the values of $(\eta^k_{n})_{k\geq0}$ are uniquely determined by $(w^k_{n})_{k\geq0}$.
\end{lem}
\begin{proof} Suppose $w^{k}_{n}=0$ for some $k\in\mathbb{Z}_+$. We see from the first and second patterns shown in \eqref{patterns} that it must then be the case that $\eta^{k}_{n}=0$. Now, if $w^{k-1}_n,w^{k-2}_n,\dots>0$, then we see from the second and third pattern of (\ref{patterns}) that $\eta^{k-1}_n,\eta^{k-2}_n,\dots$ must alternate between $1$ and $0$. Of course, if we eventually come to an $l<k$ with $w^{l}_n=0$, then we return to setting $\eta^{l}_{n}=0$ and the alternation starts again. This argument, which completes the proof of part (a) is summarised in the following diagram.
\[\begin{array}{cc}
   & w^{k}_{n}=0 \\
  \eta^{k}_{n}=0 &  \\
   & w^{k-1}_n >0\\
  \eta^{k-1}_{n}=1 &  \\
   & w^{k-2}_n >0\\
  \eta^{k}_{n}=0 &  \\
   & w^{k-3}_n >0\\
  \eta^{k-2}_{n}=1 & \vdots \\
  \vdots& w^{l}_n=0\\
  \eta^{l}_{n}=0 &
\end{array}\]
The proof of (b) is now obvious, since we can apply the argument of (a) from an arbitrarily large value of $k$.
\end{proof}

\begin{rem}\label{aboverem}
In fact, the condition given in part (b) of the above lemma is also necessary for the result. Indeed, if $w^{k}_{n}>0$ for $k\geq k_0$, then one can only deduce from $(w^k_{n})_{k\geq0}$ that
$(\eta^k_{n})_{k\geq k_0}$ are alternating between 0 and 1, but not the actual values. For example, consider the periodic configuration $\eta$ given by $(\dots,\eta_0=1,0,0,1,\dots)$. It is then the case that $w_0^k=1$ for all $k\geq 0$. However, the same is true if $\eta$ is given by $(\dots,\eta_0=0,1,1,0,\dots)$, and so the current does not determine $\eta_0$.
\end{rem}

We next look to extend the previous result to recover from $(w^k_{n})_{k\geq0}$ not just the values of $(\eta^k_{n})_{k\geq0}$, but the entire array $(\eta^k_{m})_{k\geq0,m\leq n}$. However, $(w^k_{n})_{k\geq0}$ having infinitely many zeros does not imply that $(w^k_{n-1})_{k\geq0}$ does, and so we can not immediately iterate the argument of Lemma \ref{infzerolem} to obtain the result. (Indeed, consider $(w^k_{1})_{k\geq0}$ and $(w^k_{0})_{k\geq0}$ for the example configurations discussed in Remark \ref{aboverem}.) We overcome this problem by more carefully considering the spacing between zeros. In particular, for a given $n$, define a sequence $\sigma_n=(\sigma_n^i)_{i\geq 1}$ by setting
\[\sigma_n^i=\inf\left\{k>\sigma_n^{i-1}:\:w^k_n=0\right\},\]
where we fix by convention $\sigma_n^{0}=-1$, and set $\sigma^i_n=\infty$ when it is not well-defined by the above equation. We will say that $(w^k_{n})_{k\geq0}$ has `infinitely many odd gaps between zeros' if the sequence in question has infinitely many zeros, and $\sigma^{i+1}_n-\sigma^i_n$ is odd infinitely often. A key result of this section is the following. Whilst the assumption might not immediately seem natural, it becomes clearer why it is relevant when we apply the result in the case of random initial configurations in Section \ref{invarsec}. In particular, from Corollary \ref{oddgapandflatpath} and Lemma \ref{particleflatdensity} as presented below, we will see how it is related to having a sub-critical density of particles.

\begin{prop}\label{infoddlem}
(a) If $(w^k_{n})_{k\geq0}$ has infinitely many odd gaps between zeros, then so does the sequence $(w^k_{n-1})_{k\geq0}$.\\
(b) If $(w^k_{n})_{k\geq0}$ has infinitely many odd gaps between zeros, then the values of $(\eta^k_{m})_{k\geq0,m\leq n}$ are uniquely determined by $(w^k_{n})_{k\geq0}$.
\end{prop}
\begin{proof} Suppose $(w^k_{n})_{k\geq0}$ has infinitely many odd gaps between zeros, and that $\sigma^{i+1}_n-\sigma^{i}_n$ is odd. We then note that the argument used in the proof of Lemma \ref{infzerolem} implies that both $\eta^{\sigma^i_n}_{n}=0$ and $\eta^{\sigma^i_n+1}_{n}=0$. Since we also have that $w^{\sigma^i_n}_n=0$, we can deduce from the first pattern of \eqref{patterns} that $w^{\sigma^i_n}_{n-1}=0$.

Now, let $j$ be the smallest integer strictly greater than $i$ for which $\sigma^{j+1}_n-\sigma^{j}_n$ is odd. Note that
\[\sigma^j_n-\sigma^i_n=\sum_{l=i}^{j-1}(\sigma^{l+1}_n-\sigma^{l}_n)\]
is odd, since all the summands are even apart from the first one. Moreover, by the argument of the previous paragraph, we have that $w^{\sigma^j_n}_{n-1}=0$. Of course, $w^{\sigma^i_n}_{n-1}=0$ and $w^{\sigma^j_n}_{n-1}=0$ might not be consecutive zeros of the sequence $(w^k_{n-1})_{k\geq0}$. However, because they are separated by an odd number, there must be a pair of consecutive zeros contained within the interval $[\sigma^i_n,\sigma^j_n]$ that are separated by an odd number.

Hence for each odd interval between consecutive zeros of $(w^k_{n})_{k\geq0}$, we have deduced the existence of an odd interval between consecutive zeros of $(w^k_{n-1})_{k\geq0}$. Since these are distinct by construction, we have completed the proof of part (a). The proof of part (b) is now immediate given part (a) and Lemma \ref{infzerolem}(b).
\end{proof}

As a corollary, we also have a two-sided version of Proposition \ref{infoddlem}(b). For this, it is useful to reverse space and appeal to a corresponding result for the reverse carrier. In particular, recall the notation $V=S-I$ from Section \ref{inversesec}, and note that we have by definition $V_0=W_0^{RS}$. In the following proof, we will write $v_0^k:=(T^kV)_0$.

\begin{cor}\label{keybijection}
If $(w^k_{0})_{k\in\mathbb{Z}}$ has infinitely many odd gaps between zeros in both directions (as $k\rightarrow\pm\infty$), then the values of $(\eta^k_n)_{k,n\in\mathbb{Z}}$ are uniquely determined by $(w^k_{0})_{k\in\mathbb{Z}}$.
\end{cor}
\begin{proof} Suppose $(w^k_{0})_{k\in\mathbb{Z}}$ has infinitely many odd gaps between zeros in both directions. By the same argument as in the proof of Proposition \ref{infoddlem} (b), the values of $(\eta^k_{m})_{k\in \mathbb{Z},m\leq 0}$ are uniquely determined by $(w^k_{0})_{k\in\mathbb{Z}}$. Also, since we are assuming $S \in \mathcal{S}^{inv}$, from Theorem  \ref{goodsetthm} we have that
\[w^k_0=(T^kW)_0=(T^{k+1}V)_0=v^{k+1}_0\]
for any $k \in \mathbb{Z}$. Therefore, $(v^{k}_{0})_{k\in\mathbb{Z}}$ are uniquely determined by $(w^k_{0})_{k\in\mathbb{Z}}$. Moreover,  since $w^k_0=0$ is equivalent to $v^{k+1}_0=0$, $(v^k_{0})_{k\in\mathbb{Z}}$ also has infinitely many odd gaps between zeros in both directions. In particular, $(v^k_{0})_{k\geq 0}$ has infinitely many odd gaps, and, by the symmetry described in the paragraph preceding the result, so does $(T^kW^{RS}_0)_{k\geq 0}$. Since $RS$ is the path encoding of the reversed configuration $\overleftarrow{\eta}$, it follows that $(\eta^k_{m})_{k\in \mathbb{Z},m\geq 1}$ is uniquely determined by $(v^k_{0})_{k\geq 0}$, and so by $(w^k_{0})_{k\in\mathbb{Z}}$.
\end{proof}

\begin{rem}
By re-centering the relevant path encodings, the previous corollary is easily generalized to the statement for $(w^k_{n})_{k\in\mathbb{Z}}$ instead of $(w^k_{0})_{k\in\mathbb{Z}}$ for any $n \in \mathbb{Z}$.
\end{rem}

Before proceeding, we summarise some further useful properties that follow from the above argument. To this end, for a given $n\in\mathbb{Z}$, we define a sequence $\tilde{\sigma}_n=(\tilde{\sigma}_n^i)_{i\geq 1}$ by setting
\[\tilde{\sigma}_n^i=\inf\left\{\sigma_n^{k}>\tilde{\sigma}_n^{i-1}:\: k \ge 0, \ \sigma_n^{k+1}-\sigma_n^k : \text{odd} \right\},\]
where we fix by convention $\tilde{\sigma}_n^{0}=-1$, and set $\tilde{\sigma}^i_n=\infty$ when it is not well-defined by the above equation.

\begin{cor}\label{oddgapandflatpath} The sequence $(w^k_{n})_{k\geq0}$ has infinitely many odd gaps between zeros if and only if $w^k_n=w^k_{n-1}=0$ infinitely often as $k \to \infty$. If this is the case, for any $k \ge 0$, $w^k_n=w^k_{n-1}=0$ if and only if $k \in \{ \tilde{\sigma}_n^i : \: i \ge 0\}$, and so
\begin{equation}\label{tildesig}
\tilde{\sigma}_n^i = \inf\left\{ k >\tilde{\sigma}_n^{i-1}:\: w^k_n=w^k_{n-1}=0 \right\}
\end{equation}
for all $i \ge 0$. Moreover, it holds that $\tilde{\sigma}_{n}^{i-1} \le \tilde{\sigma}_{n-1}^i < \tilde{\sigma}_{n}^{i+1}$ for all $i \ge 0$.
\end{cor}
\begin{proof} We observed in the proof of Proposition \ref{infoddlem}(a) that if $\sigma^{i+1}_n-\sigma^{i}_n$ is odd, then $w^{\sigma^i_n}_{n}=w^{\sigma^i_n}_{n-1}=0$. Hence if the sequence $(w^k_{n})_{k\geq0}$ has infinitely many odd gaps between zeros, then $w^k_n=w^k_{n-1}=0$ infinitely often as $k \to \infty$. In the converse direction, if $w^k_n=w^k_{n-1}=0$ infinitely often as $k\rightarrow\infty$, then clearly $\sigma^i_n<\infty$ for each $i$. Moreover, if $w^k_n=w^k_{n-1}=0$ for some $k=\sigma^i_n$, then the argument used in the proof of Lemma \ref{infzerolem} also yields $\sigma^{i+1}_n-\sigma^{i}_n$ is odd, as desired. This completes the first part of the proof, and also establishes (\ref{tildesig}).

Next, suppose $\sigma^{i+1}_n-\sigma^{i}_n$ is odd, and let $j$ be the smallest integer strictly greater than $i$ for which $\sigma^{j+1}_n-\sigma^{j}_n$ is odd. We will show that there exists exactly one pair of consecutive zeros of $(w_{n-1}^k)_{k\geq 0}$ contained within the interval $[\sigma^i_n,\sigma^j_n]$ that is separated by an odd number. In fact, $\sigma^i_n=\sigma^{a_i}_{n-1}$ and $\sigma^j_n=\sigma^{a_j}_{n-1}$ for some $a_i < a_j$, and for any $a_i < \ell <a_j$, $w^{\sigma^{\ell}_{n-1}}_{n-1}=0$ and $w^{\sigma^{\ell}_{n-1}}_{n}=1$. Since $(\eta^{k}_{n})_{\sigma^i_n < k \leq \sigma^j_n}$ is alternating, $\sigma^{\ell+1}_{n-1}-\sigma^{\ell}_{n-1}$ must be an even number if $a_i \leq \ell < \ell +1 < a_j$. In particular, the only odd gap within the interval $[\sigma^i_n,\sigma^j_n]$ is given by $\sigma^{a_j}_{n-1}-\sigma^{\ell}_{n-1}$, where $\ell$ is the greatest integer strictly smaller than $a_j$ for which $\sigma^{\ell}_{n-1}=0$.
\end{proof}

The previous corollary demonstrated the relevance of empty periods for the carrier. In the next part of this subsection, we study the relations between the latter and the boundary conditions of particle configurations, which will be useful when it comes to proving our main result (Theorem \ref{bijectionthm}). Note that we now drop the a priori assumption that $\eta\in\mathcal{S}^{inv}$, and explicitly state in which set $\eta$ is contained in the individual results. We start by introducing some notation concerning the intervals the carrier spends with no particles. In particular, we define a map $N^{\pm} : \mathcal{Y} \to \mathbb{Z} \cup \{-\infty\}\cup\{\infty\}$ by setting
\[N^{-}(Y):=\inf\left\{n \in \mathbb{Z}\::\: Y_{n-1}=Y_n=0 \right\},\qquad N^{+}(Y):=\sup\left\{n \in \mathbb{Z}\::\: Y_{n-1}=Y_n=0 \right\},\]
with the convention that $\inf \emptyset= \infty$, $\sup \emptyset=-\infty$. Moreover, for $\eta \in \mathcal{S}^T$, we define $N^{\pm}_W$ by $N^{\pm}_W=N^{\pm}(W)$, or equivalently
\[N_W^-(\eta)=\inf\left\{n \in \mathbb{Z}\::\: \eta \in A_n\right\},\qquad N_W^+(\eta)=\sup\left\{n \in \mathbb{Z}\::\: \eta \in A_n\right\},\]
where $A_n=\{\eta:\: W_n=W_{n-1}=0\}$. We similarly define $N^{\pm}_V=N^{\pm}(V)$ for $\eta \in \mathcal{S}^{T^{-1}}$, and note that this can be explicitly expressed as
\[N_V^-(\eta)=\inf\left\{n \in \mathbb{Z}\::\: \eta \in B_n\right\},\qquad N_V^+(\eta)=\sup\left\{n \in \mathbb{Z}\::\: \eta \in B_n\right\},\]
where $B_n=\{\eta:\: V_n=V_{n-1}=0\}$. We have the following basic observations.

\begin{lem}\label{NWboundary}
For $\eta \in \mathcal{S}^{inv}$, the following holds.\\
(a) $N_W^+=\infty \Leftrightarrow N_V^{+}=\infty \Leftrightarrow \eta \in \mathcal{S}^+_{sub-critical}$\\
(b) $N_W^+ < \infty \Leftrightarrow N_V^{+} < \infty \Leftrightarrow \eta \in \mathcal{S}^+_{critical}$ \\
(c) $N_W^-=-\infty \Leftrightarrow N_V^{-}=-\infty \Leftrightarrow \eta \in \mathcal{S}^-_{sub-critical}$\\
(d) $N_W^- > -\infty \Leftrightarrow N_V^{-}> -\infty \Leftrightarrow \eta \in \mathcal{S}^-_{critical}$
\end{lem}
\begin{proof}
Since $M_n-M_0=\ell(W)_n$, $N_W^+=\infty$ is equivalent to $\lim_{n \to \infty}M_n=\infty$. Moreover, for $S \in \mathcal{S}^{inv}$, $\lim_{n \to \infty}M_n=\infty$ is equivalent to $\lim_{n \to \infty}I_n=\infty$, and these are also both equivalent to $\eta \in \mathcal{S}^+_{sub-critical}$. From these observations, we readily obtain claim (a). The other claims are similarly straightforward.
\end{proof}

We next show that $N_W^-$ and $N_W^+$ are strictly increasing under the action of $T$.

\begin{lem}\label{NMlemma}
If $\eta, T\eta \in \mathcal{S}^{T}$, then $N_W^-(\eta)+1 \le N_W^-(T\eta)$ and $N_W^+(\eta) +1 \le N_W^+(T\eta)$.
\end{lem}
\begin{proof}
We only show $N_W^-(\eta)+1 \le N_W^-(T\eta)$, since the other claim follows from the same argument. If $N_W^-(\eta)=-\infty$, the claim trivially holds. If $N_W^-(\eta)=\infty$, then $\left\{n \in \mathbb{Z}\::\: \eta \in A_n\right\}=\emptyset$, and so $M_{-\infty}=M_{\infty}$. In particular, $T\eta=1-\eta$ and $TS=-S$. Since $T\eta \in \mathcal{S}^{T}$, $I_{-\infty} > -\infty$ and $M_{-\infty}^{TS}=-I_{-\infty}=M_{\infty}^{TS}$, so $N_W^-(T\eta)=\infty$. Finally, suppose $N_W^-(\eta) \in \mathbb{Z}$. It then holds that $-\infty < M_{-\infty}=M_{N_W^-(\eta)-1}=M_{N_W^-(\eta)}-1$. Since $T\eta \in \mathcal{S}^{T}$, $I_{-\infty} > -\infty$ and $M_{-\infty}^{TS}=2M_{-\infty}-I_{-\infty}-2M_0$. In particular, $M_{-\infty}^{TS} \ge 2M_n-S_n-2M_0=TS_n$ for any $n \le N_W^-(\eta)-1$. Also, if $n = N_W^-(\eta)$, then
\[M_{-\infty}^{TS} -(TS)_n=(2M_{-\infty}-I_{-\infty})-(2M_n-S_n)=(2M_{-\infty}-I_{-\infty})-M_n=M_{-\infty}-I_{-\infty}-1  \ge 0
\]
since $M_{-\infty}-I_{-\infty} \ge 1$, which follows from the fact that the increments of $S$ take values in the set $\{\pm 1\}$. Therefore we have that $M^{TS}_n = M_{-\infty}^{TS}$ for $n \le  N_W^-(\eta)$, and so $N_W^-(\eta)+1 \le N_W^-(T\eta)$.
\end{proof}

\begin{rem} If we start from a configuration in $\mathcal{S}_{critical}^-\cap\mathcal{S}_{sub-critical}^+$, then we know that $-\infty<N_W^-< N_W^+=\infty$. Informally, we can view $N_W^-$ as the boundary between the critical and sub-critical sections of the configuration. The previous result shows that $N_W^-(T^k\eta)$ diverges to $+\infty$ as $k\rightarrow+\infty$, and so locally the configuration eventually looks critical. A similar observation can be made for a configuration in $\mathcal{S}_{sub-critical}^-\cap\mathcal{S}_{critical}^+$, with the boundary between the two regimes being $N_W^+$ in this case.
\end{rem}

The following lemma studies the relation between $N^-_W$ for the minimal carrier $W$, and other carriers describing the same particle configuration.

\begin{lem}\label{ywrelation} Let $\eta \in \mathcal{S}^T$ and $Y \in \mathcal{Y}$ satisfy $\Phi(Y)=\eta$. The following then hold.\\
(i) If $N_W^-(\eta)=-\infty$ and $\Psi(Y) \in \mathcal{S}^{T}$, then $Y=W$. \\
(ii) If $N_W^-(\eta) > -\infty$ and $\Psi(Y) \in \mathcal{S}^T$, then $N^-_W(\Psi(Y)) \ge N_W^-(\eta)$.
\end{lem}
\begin{proof} From Proposition \ref{phiproperty}, $\Phi(Y)=\eta=\Phi(W)$ implies $W_n \le Y_n$ for all $n$. In particular, for $n_0:=\inf\{ n : W_n=Y_n\}$, we have $W_n=Y_n$ for all $n \ge n_0$, and $W_n \le Y_n-1$ for all $n < n_0$. Note that $W=Y$ is equivalent to $n_0=-\infty$.

Suppose $N_W^-(\eta)=-\infty$ and $n_0 > -\infty$. Then $\lim_{n \to -\infty}(Y_n-W_n) =\infty$, since $Y_n-Y_{n-1}=-1$ for any $n <n_0$ satisfying $W_n=W_{n-1}=0$, and otherwise $W_n-W_{n-1}=Y_n-Y_{n-1}$. Therefore $\lim_{n \to -\infty}Y_n=\infty$. Thus, to complete the proof of (i), it remains to show that, in this case, $\Psi(Y) \notin \mathcal{S}^{T}=\Phi(\mathcal{Y})$. Suppose that there exists $\tilde{Y} \in \mathcal{Y}$ such that $\Phi(\tilde{Y})=\Psi(Y)$. It is then the case that, for any $n_1 \le n_0$ and $n \le n_1$,
\[\tilde{Y}_{n_1}-\tilde{Y}_n =\sum_{i=n}^{n_1-1} (\tilde{Y}_{i+1}-\tilde{Y}_{i}) \ge -\sum_{i=n}^{n_1-1} (Y_{i+1}-Y_{i})=Y_n-Y_{n_1},\]
since $Y_{i+1}-Y_i \in \{-1,1\}$ for all $i \le n_0-1$. Recalling that $\lim_{n \to -\infty}Y_n=\infty$, this implies $\tilde{Y}_{n_1}=\infty$, which can not be the case. Thus we have shown $\Psi(Y) \notin \mathcal{S}^{T}$, as desired.

Next, we assume $N_W^-(\eta) > -\infty$ and $\Psi(Y)  \in \mathcal{S}^T$, and denote $A=N_W^-(\eta)$. Then, for any $n \le A-1$, $W_n-W_{n-1} \in \{-1,1\}$, and so $Y_n-Y_{n-1}=W_n-W_{n-1}$. In particular $(T\eta)_n=\Psi(W)_n=\Psi(Y)_n$ for all $n \le A-1$. Thus, since the path encoding of $\Psi(Y)$ is the translation of $TS$ on $n \le A-1$, $\Psi(Y)  \in \mathcal{S}^T$ implies $T\eta \in \mathcal{S}^T$. In particular, $\Phi^{-1}(T\eta)_n=\Phi^{-1}(\Psi(Y))_n$ on $n \le A-1$. Moreover, by Lemma \ref{NMlemma}, $N_W^-(T\eta) \ge N_W^-(\eta)+1=A+1$. Therefore, if $n \le A-1$, then $\Phi^{-1}(T\eta)_{n-1}=\Phi^{-1}(T\eta)_n=0$ does not occur, and so neither does $\Phi^{-1}(\Psi(Y))_{n-1}=\Phi^{-1}(\Psi(Y))_n=0$. Hence we conclude $N_W^-(\Psi(Y)) \ge A$.
\end{proof}

With the above preparations in place, we are now ready to study the map
\begin{eqnarray*}
\Lambda : \mathcal{S}^{inv} &\to& \mathbb{Z}_+^{\mathbb{Z}}\\
\eta&\mapsto&\left((T^kW)_0\right)_{k \in \mathbb{Z}}.
\end{eqnarray*}
In particular, we will describe a restriction of this map which is a bijection. For this, we introduce a ``good'' subset of $\mathbb{Z}_+^{\mathbb{Z}}$ by setting
\[(\mathbb{Z}_+^{\mathbb{Z}})^*:=\{y=(y^k)_{k} \in \mathbb{Z}_+^{\mathbb{Z}} : y \text{ has infinitely many odd gaps between zeros in both directions} \},\]
where we use the terminology for odd gaps between zeros from earlier in the section. Moreover, denote
\begin{equation}\label{ssubcritstar}
\mathcal{S}_{sub-critical}^*:=\Lambda^{-1}((\mathbb{Z}_+^{\mathbb{Z}})^*).
\end{equation}
(We note that Lemma \ref{currenttoparticle} below shows $\Lambda^{-1}((\mathbb{Z}_+^{\mathbb{Z}})^*) \subseteq \mathcal{S}_{sub-critical}$.) In Corollary \ref{keybijection}, we already showed that $\Lambda|_{\mathcal{S}_{sub-critical}^*}$ is injective. Our main result demonstrates that, taking $(\mathbb{Z}_+^{\mathbb{Z}})^*$ as the codomain, the latter map is actually a bijection.

\begin{thm}\label{bijectionthm} The map $\Lambda : \mathcal{S}_{sub-critical}^* \to (\mathbb{Z}_+^{\mathbb{Z}})^*$ is a measurable bijection.
\end{thm}

This result is a consequence of the following lemma.

\begin{lem}\label{currenttoparticle} The map $\Lambda : \mathcal{S}_{sub-critical}^* \to (\mathbb{Z}_+^{\mathbb{Z}})^*$ is surjective, and it moreover holds that $\mathcal{S}_{sub-critical}^* \subseteq \mathcal{S}_{sub-critical}$.
\end{lem}
\begin{proof}
For any $(y^k)_{k \in \mathbb{Z}} \in (\mathbb{Z}_+^{\mathbb{Z}})^*$, we can construct $(y^k_n)_{k \in \mathbb{Z}, n \in \mathbb{Z}}$ and $(\eta^k_n)_{k \in \mathbb{Z}, n \in \mathbb{Z}}$ satisfying $(y^k_0)_k=(y^k)_k$ uniquely by applying the basic patterns given in \eqref{patterns}. In particular, they satisfy $\Phi(y^k)=\eta^k$ and $\Psi(y^{k-1})=\eta^{k}$, and so $\eta^k \in \mathcal{S}^T \cap \mathcal{S}^{T^{-1}}$ for any $k\in\mathbb{Z}$. Moreover, $(y^k_n)_k \in (\mathbb{Z}_+^{\mathbb{Z}})^*$ for all $n \in \mathbb{Z}$. Our goal is to prove that $\Lambda(\eta)=(y^k)_{k\in\mathbb{Z}}$ for $\eta:=(\eta^0_n)_n$, and also that $\eta \in \mathcal{S}_{sub-critical}$. For this, it is enough to check that $\eta^k \in \{\lim_{n \to-\infty}M_n=- \infty, \lim_{n \to \infty}I_n=\infty\}$ and $T^kW=y^k$ for all $k\in\mathbb{Z}$, because $\{\lim_{n \to-\infty}M_n=- \infty, \lim_{n \to \infty}I_n=\infty\} \subseteq \mathcal{S}^{rev}$.

First, we suppose $\eta^k \in \{\lim_{n \to-\infty}M_n=- \infty, \lim_{n \to \infty}I_n=\infty\}$ for all $k\in\mathbb{Z}$, and show that $T^kW=y^k$ and $T^k\eta=\eta^k$ by induction. To begin with we show that $W=y^0$. In particular, we have that $\eta\in \mathcal{S}^T$ and $y^0$ satisfies $\Phi(y^0)=\eta=\Phi(W)$. Moreover, by assumption, it holds that $N^-_W(\eta)=-\infty$. Hence, because it also holds that $\Psi(y^0)=\eta^1\in \mathcal{S}^T$, Lemma \ref{ywrelation}(i) yields that $W=y^0$. It moreover follows that $T\eta = \eta^1$. Since by assumption we also have that $\eta^1 \in \{\lim_{n \to-\infty}M_n=- \infty, \lim_{n \to \infty}I_n=\infty\}$, iterating the argument gives $T^kW=y^k$ and $T^k \eta = \eta^k \in \mathcal{S}^{rev}$ for all $k \ge 0$. By symmetry, we can also show that $T^{k+1}V=y^k$ and $T^k \eta = \eta^k \in \mathcal{S}^{rev}$ for all $k \le -1$. From this, we have $\eta \in \mathcal{S}_{sub-critical}$, and so $T^kW=T^{k+1}V$ for all $k$, which confirms $T^kW=y^k$ for all $k$.

Next, we suppose $\eta^k \notin \{\lim_{n \to-\infty}M_n=- \infty, \lim_{n \to \infty}I_n=\infty\}$ for some $k\in\mathbb{Z}$. Without loss of generality, we can assume $\lim_{n \to-\infty}M_n(\eta^k) > -\infty$, hence $N_W^-(\eta^k) > -\infty$. Denote $N_W^-(\eta^k) = A \in \mathbb{Z} \cup \{\infty\}$. From now on, we derive that $N^-(y^{\ell}) \ge A$ for all $\ell \ge k$, which contradicts the condition $(y^{\ell}_{A-1})_{\ell} \in (\mathbb{Z}_+^{\mathbb{Z}})^*$. For this, it is enough to show that $N^-_W(\eta^{\ell}) \ge A$ for all $\ell \ge k$, since $\Phi(y^{\ell})=\Phi(\Phi^{-1}\eta^{\ell})$, and so $y^{\ell}_n \ge (\Phi^{-1}\eta^{\ell})_n$ for all $n$. The claim $N^-_W(\eta^{\ell}) \ge A$ for all $\ell \ge k$ can be shown by induction on $\ell$. Indeed, since $\eta^{\ell}=\Phi(y^{\ell})$ and $\Psi(y^{\ell})=\eta^{\ell+1} \in \mathcal{S}^T$, $N^-_W(\eta^{\ell}) \ge A$ implies $N^-_W(\eta^{\ell+1}) \ge A$ by Lemma \ref{ywrelation}(ii).
\end{proof}

\begin{rem}
Using the notation $A_n^k=T^{-k}A_n=\{\eta:\:T^k\eta \in A_n\}$, the set $\mathcal{S}_{sub-critical}^*$ can alternatively be characterised as
\[\mathcal{S}_{sub-critical}^*=\mathcal{S}_{sub-critical} \cap \left( \bigcap_{n \in \mathbb{Z}} \left(\limsup_{k \to \infty}A_n^k \cap \limsup_{k \to -\infty}A_n^k\right)\right).\]
Indeed, for $\eta \in \mathcal{S}^{inv}$, by appealing to the two-sided extensions of Proposition \ref{infoddlem} and Corollary \ref{oddgapandflatpath} (which are straightforward to deduce by applying similar arguments to above), we have that $\eta \in \cap_{n\in\mathbb{Z}}\left(\limsup_{k \to \infty}A_n^k \cap \limsup_{k \to -\infty}A_n^k\right)$ if and only if $((T^lW)_0)_{l \in \mathbb{Z}} \in (\mathbb{Z}_+^{\mathbb{Z}})^*$.
\end{rem}

\section{Random initial configurations}\label{probsec}

In this section, we turn our attention to the case when the initial configuration is random. The starting point will be that $\eta=(\eta_n)_{n\in\mathbb{Z}}$ is a sequence of Bernoulli random variables, built on a probability space with probability measure $\mathbf{P}$, whose corresponding path encoding has distribution supported in $\mathcal{S}^{rev}$. It is then the case that $T\eta$ is well-defined, $\mathbf{P}$-a.s. Going beyond this, it is a natural for random initial configurations to ask whether the law of $\eta$ is preserved by $T$, that is, is it the case that $T\eta \buildrel{d}\over{=}\eta$? As we noted in the introduction, one way in which we are able to answer this question is in terms of the particle current. In particular, in Section \ref{invarsec}, we prove Theorems \ref{mrb} and \ref{mrc}, which discuss the situation in the critical and sub-critical cases, respectively. We also establish Theorem \ref{mrd}, which gives simple sufficient conditions for invariance based on the symmetry of $\eta$ and $W$. Moreover, we check the basic properties of invariant measures stated as Theorem \ref{mra}.

Our next observation concerns the case when $\eta=(\eta_n)_{n\in\mathbb{Z}}$ is a stationary, ergodic sequence. In particular, if we assume that the density of this sequence satisfies
\begin{equation}\label{p0}
\rho=\mathbf{P}\left(\eta_0=1\right)<\frac12,
\end{equation}
then ergodicity implies that $S$, as defined by \eqref{SRWrep}, satisfies
\[\frac{S_n}n=\frac{\sum_{m=1}^n(S_m-S_{m-1})}{n}=\frac{\sum_{m=1}^n(1-2\eta_m)}{n}\rightarrow 1-2\rho>0,\qquad \mathbf{P}\mbox{-a.s.}\]
Similarly, $S_n/n\rightarrow 1-2\rho>0$ as $n\rightarrow-\infty$, $\mathbf{P}$-a.s. Thus Theorem \ref{mr1} gives the following result, which yields in turn that $(T^kS)_{k\in\mathbb{Z}}$ is well-defined, $\mathbf{P}$-a.s.

{\lem \label{slinrand} If $\eta$ is a stationary, ergodic sequence satisfying (\ref{p0}), then $S\in \mathcal{S}_{F_{1-2\rho}}^-\cap\mathcal{S}_{F_{1-2\rho}}^+$, $\mathbf{P}$-a.s., where $F_{1-2\rho}(n):=(1-2\rho )n$. (Recall the notation for $\mathcal{S}_F^{\pm}$ from \eqref{sfudef} and \eqref{sfldef}.) In particular, $S\in \mathcal{S}_{sub-critical}$ (where the latter set was defined at (\ref{subcritdef})), $\mathbf{P}$-a.s.}
\medskip

As introduced in Theorem \ref{mre}, within the class of stationary, ergodic sequences $\eta$, we are able to establish invariance in distribution under $T$ for a number of specific examples: when the initial configuration is independent and identically distributed (i.i.d.); when the initial configuration is Markov; an example with bounded solitons obtained by conditioning the i.i.d.\ initial configuration; and an example with bounded solitons for which the carrier satisfies a strong symmetry condition. These are introduced in Section \ref{examplessec}, which is where we prove Theorem \ref{mre}. Actually, the Markov initial configuration case includes the i.i.d.\ one, but we prefer to separate these, as many properties of the model are simpler in the i.i.d.\ case, which enables us to derive more detailed results in this setting. As well as checking the invariance of the aforementioned examples, we prove that these examples are the only distributionally invariant (under $T$) configurations with $S\in \mathcal{S}^{rev}$, $\mathbf{P}$-a.s., for which $\eta$ or $W$ is a two-sided stationary Markov chain. In Section \ref{currentclt}, we study the current across the origin, proving Theorem \ref{mrf} and Corollary \ref{ergcormrf} in particular. Finally, the section is completed by an investigation into the distance travelled by a tagged particle (see Section \ref{distancesec}, which is where Theorem \ref{mrg} is established).

\subsection{Invariance in distribution and ergodicity for random particle configurations}\label{invarsec} In this section, we study the properties of invariance and ergodicity for initial configuration $\eta$ under the action of the BBS. In particular, we prove Theorems \ref{mra}, \ref{mrb}, \ref{mrc} and \ref{mrd}. Most of the results are stated under the assumptions of Theorem \ref{mra}, namely that $\eta$ is a random particle configuration such that the distribution of the corresponding path encoding $S$ is supported on $\mathcal{S}^{rev}$, and $T\eta\buildrel{d}\over{=}\eta$ holds. Note that we do not restrict to  stationary, ergodic sequences in this section.

To begin with, we prepare a simple, but useful, lemma that gives a relation between the probability of having a particle at $n$ and the probability of seeing a flat segment in the carrier path at $n$ that holds for any invariant measure.

\begin{lem}\label{particleflatdensity} Under the assumptions of Theorem \ref{mra}, for any $n \in \mathbb{Z}$,
\[\mathbf{P}(\eta_n=1)=\frac{1}{2}\left(1-\mathbf{P}(W_n=W_{n-1}=0)\right).\]
In particular, $\mathbf{P}(\eta_n=1)=\frac{1}{2}$ is equivalent to $\mathbf{P}(W_n=W_{n-1}=0)=0$. Also, for any $n \in \mathbb{Z}$, $\mathbf{P}(\eta_n=1) \le \frac{1}{2}$.
\end{lem}
\begin{proof}
By the invariance of the measure under $T$,
\begin{align*}
\mathbf{P}(\eta_n=1)& =\frac{1}{2}\left(\mathbf{P}(\eta_n=1)+\mathbf{P}(T\eta_n=1)\right) \\
& = \frac{1}{2}\left(\mathbf{P}(W_n-W_{n-1}=1)+ \mathbf{P}(W_n-W_{n-1}=-1)\right)\\
&=\frac{1}{2}\left(1-\mathbf{P}(W_n=W_{n-1}=0)\right).
\end{align*}
The other claims are obvious from the equation.
\end{proof}

Next, we give an important characterization of the support of invariant measures.

\begin{lem}\label{Npmlem}
Under the assumptions of Theorem \ref{mra}, it holds that, $\mathbf{P}$-a.s.,
\[\eta \in \left\{N_W^-=-\infty,\:N_W^+=\infty\right\} \cup \{N_W^-=\infty,\:N_W^+=-\infty\}.\]
\end{lem}
\begin{proof}
Since $T\eta\buildrel{d}\over{=}\eta$, we have $\eta \in \mathcal{S}^{inv}$, $\mathbf{P}$-a.s. Now, denoting $c_n=\mathbf{P}(N^+_W(\eta) \le n)$ for $n \in \mathbb{Z}$, by definition we have $c_n \ge c_{n-1}$, and Lemma \ref{NMlemma} yields
\[c_n= \mathbf{P}(N^+_W(T\eta) \le n) \le \mathbf{P}(N^{+}_W(\eta) \le n-1) =c_{n-1}.\]
Hence $c_n=c_{n-1}$ for all $n \in \mathbb{Z}$, and so $\mathbf{P}(N^+_W(\eta)=  n)=c_n-c_{n-1}=0$ for all $n \in \mathbb{Z}$. A similar argument for $N^{-}_W$ shows that $N^{-}_W(\eta) \in \{\infty,-\infty\}$, $\mathbf{P}$-a.s. Finally, since $N^{+}_W(\eta)=-\infty$ if and only if $N^{-}_W(\eta)=\infty$ (on the event that neither take a value in $\mathbb{Z}$), the proof is complete.
\end{proof}

From Lemmas \ref{NWboundary} and \ref{Npmlem}, we see that any invariant measure must satisfy $\eta \in \mathcal{S}_{sub-critical}\cup \mathcal{S}_{critical}$, $\mathbf{P}$-a.s. In the following, we show that we can say even more, specifically that $\eta \in \mathcal{S}_{sub-critical}^*\cup \mathcal{S}_{critical}^*$, $\mathbf{P}$-a.s., where $\mathcal{S}_{sub-critical}^*$ is the set defined in Section \ref{currentsec} (at \eqref{ssubcritstar}),
and
\[\mathcal{S}_{critical}^*:=\mathcal{S}_{critical} \cap \left( \bigcap_{n \in \mathbb{Z}} A_n^c \right)=\left\{ S \in \mathcal{S}^0 \:: \: M_{-\infty}=M_{\infty} \in \mathbb{R},\: I_{-\infty}=I_{\infty} \in \mathbb{R} \right\}.\]
We first deal with the sub-critical case.

\begin{lem}\label{infoftenflat} Under the assumptions of Theorem \ref{mra}, it is the case that
\[\mathbf{P}\left(S\in\mathcal{S}_{sub-critical}\backslash\mathcal{S}_{sub-critical}^*\right)=0.\]
\end{lem}
\begin{proof} For $S \in \mathcal{S}_{sub-critical}$, $\lim_{n \to \infty}S_n=\infty$, and so $\lim_{n \to \infty}\ell(W)_n= \infty$. Therefore, $W_n=W_{n-1}=0$ for arbitrarily large $n$. Namely, $S\in \limsup_{n \to \infty}A_n$. Now, by Poincar\'{e}'s recurrence theorem \cite[Theorem 1.4]{Walters},
\[\mathbf{P}\left(S\in A_n\backslash \limsup_{k \to \infty}A_n^k \right)=0,\qquad\forall n\in \mathbb{Z}.\]
Moreover, from Proposition \ref{infoddlem}(a),
\[\bigcap_{m \le n} \limsup_{k \to \infty}A_m^k =\limsup_{k \to \infty}A_n^k,\]
and so
\[\mathbf{P}\left(S\in A_n\backslash \bigcap_{m \le n}\limsup_{k \to \infty}A_m^k \right)=0,\qquad\forall n\in \mathbb{Z}.\]
Hence, since $\mathcal{S}_{sub-critical}\subseteq\limsup_{n \to \infty}A_n$, it follows that
\[{\mathbf{P}\left(S\in \mathcal{S}_{sub-critical}\backslash\bigcap_{n \in\mathbb{Z}}\limsup_{k \to \infty}A_n^k \right)}\leq \mathbf{P}\left(S\in \limsup_{n\rightarrow\infty}\left(A_n\backslash \bigcap_{m \le n}\limsup_{k \to \infty}A_m^k\right)\right)=0.\]
By a symmetric argument for the time-reversed process, we can similarly conclude
\[\mathbf{P}\left(S\in \mathcal{S}_{sub-critical}\backslash\bigcap_{n \in\mathbb{Z}}\limsup_{k \to {-\infty}}B_n^k \right)=0,\]
where $B_n:=\{\eta:\: V_n=V_{n+1}=0\}$ and $B_n^k=\{\eta:\: T^k\eta \in B_n\}$. On the other hand, since $S \in \mathcal{S}^{inv}$, $\mathbf{P}$-a.s., we have from Theorem \ref{goodsetthm} that $T^{k+1}V=T^kW$ holds for any $k$, and so
\[\bigcap_{n \in \mathbb{Z}} \limsup_{k \to -\infty}B_n^k = \bigcap_{n \in \mathbb{Z}} \limsup_{k \to -\infty}A_n^k,\]
which completes the proof.
\end{proof}

We now show that the support of any invariant measure is restricted to the union of $\mathcal{S}_{sub-critical}^*$ and $\mathcal{S}_{critical}^*$, thus demonstrating that there must be common boundary conditions at $\pm\infty$.

\begin{prop} \label{critsubcrit}
Under the assumptions of Theorem \ref{mra}, it holds that $S \in \mathcal{S}_{sub-critical}^* \cup \mathcal{S}_{critical}^*$, $\mathbf{P}$-a.s.
\end{prop}
\begin{proof} From Lemma \ref{Npmlem}, we know that $S$ is supported on
\[\mathcal{S}^{inv} \cap \left(\left\{N^-_W=-\infty,\:N^+_W=\infty\right\} \cup \{N^-_W=\infty,\:N^+_W=-\infty\}\right).\]
As $\mathcal{S}^{inv} \cap \left\{N^-_W=-\infty,\:N^+_W=\infty\right\} \subseteq \mathcal{S}_{sub-critical}$
and $\mathcal{S}^{inv} \cap \left\{N^-_W=\infty,\:N^+_W=-\infty\right\} \subseteq  \mathcal{S}_{critical}\cap \left( \bigcap_{n \in \mathbb{Z}} A_n^c \right)$, the result follows from Lemma \ref{infoftenflat}.
\end{proof}

Since $\mathcal{S}_{sub-critical}^*$ and $\mathcal{S}_{critical}^*$ are invariant under $T$, any invariant measure can be decomposed into the parts supported on each of the sets $\mathcal{S}_{sub-critical}^*$ and on $\mathcal{S}_{critical}^*$. Therefore, from now on, we study the properties of invariant measures supported only on $\mathcal{S}_{sub-critical}^*$ or $\mathcal{S}_{critical}^*$. We start by characterising the invariant measures supported on $\mathcal{S}_{critical}^*$.

\begin{prop}\label{critprop} Suppose $\eta$ is a random particle configuration such that the distribution of the corresponding path encoding $S$ is supported on $\mathcal{S}^{rev}$. The following conditions are then equivalent.\\
(i)  $T\eta\buildrel{d}\over{=}\eta$ and $S \in \mathcal{S}_{critical}^*$, $\mathbf{P}$-a.s.\\
(ii) $T\eta\buildrel{d}\over{=}\eta$ and $P(\eta_n=1)=\frac{1}{2}$ for all $n \in \mathbb{Z}$.\\
(iii) $\eta \buildrel{d}\over{=} 1-\eta$ and $S\in \mathcal{S}_{critical}^*$, $\mathbf{P}$-a.s.\\
(iv) $S\buildrel{d}\over{=} -S$ and $S \in \mathcal{S}_{critical}^*$, $\mathbf{P}$-a.s.\\
(v) $S \in \cup_{K \in \mathbb{N} } \mathcal{S}_{K}$, $\mathbf{P}$-a.s., and for each positive integer $K$, $W \buildrel{d}\over{=} K-W$, $\mathbf{P}$-a.s.\ on $\mathcal{S}_{K}$, where we write $\mathcal{S}_K=\mathcal{S}_K^- \cap \mathcal{S}_K^+$.
\end{prop}
\begin{proof}
From Lemma \ref{Npmlem}, (i) implies $P(W_n=W_{n-1}=0)=0$ for all $n$, and so (ii) follows directly from Lemma \ref{particleflatdensity}. Conversely (ii) implies $P(W_n=W_{n-1}=0)=0$ for all $n$, and so (i). Also, (i) and (ii) imply $T\eta= 1-\eta$, $\mathbf{P}$-a.s., and so (iii) follows. The condition $S \in \mathcal{S}_{critical}^*$, $\mathbf{P}$-a.s.\ in (iii) also implies $T\eta = 1-\eta$, $\mathbf{P}$-a.s., and so (i) follows. The equivalence between (iii), (iv) and (v) are straightforward.
\end{proof}

The following lemma will allow us to replace $\mathcal{S}_{critical}^*$ with $\mathcal{S}_{critical}$ in $\mathbf{P}$-a.s.\ statements for invariant measures.

\begin{lem} \label{critstarremove}
Suppose $\eta$ is a random particle configuration such that the distribution of the corresponding path encoding $S$ is supported on $\mathcal{S}_{critical}$. Then, $\eta \buildrel{d}\over{=} 1-\eta$ implies $S \in \mathcal{S}_{critical}^*$, $\mathbf{P}$-a.s.
\end{lem}
\begin{proof}
For any $S \in \mathcal{S}_{critical} \setminus \mathcal{S}_{critical}^*$, we have \[-\infty < \limsup_{n \to -\infty}S_n=M_{-\infty} < M_{\infty}=\sup_{n}S_n < \infty.\]
Since $S^{1-\eta}=-S$, where $ S^{1-\eta}$ is the path encoding for $1-\eta$, it follows that
\[\liminf_{n \to -\infty}S^{1-\eta}_n=\liminf_{n \to -\infty}(-S_n) > \inf_{n}(-S_n)=I_{-\infty}^{1-\eta}.\]
Thus we obtain from Theorem \ref{mr1} that $S^{1-\eta} \notin \mathcal{S}^{rev}$. As $\eta \buildrel{d}\over{=} 1-\eta$ implies that $S^{1-\eta} \in \mathcal{S}^{rev}$, $\mathbf{P}$-a.s., we can conclude that $\mathbf{P}(S \in \mathcal{S}_{critical} \setminus \mathcal{S}_{critical}^*)= 0$, as desired.
\end{proof}

We proceed to turn our attention to invariant measures supported on $\mathcal{S}_{sub-critical}^*$. From Theorem \ref{bijectionthm}, we readily deduce the equivalence of the invariance of the configuration $\eta$ under $T$ and the invariance of the current $((T^kW)_0)_{k \in \mathbb{Z}}$ under the shift $\theta$.

\begin{prop}\label{subcritprop}
Suppose $\eta$ is a random particle configuration such that the distribution of the corresponding path encoding $S$ is supported on $\mathcal{S}_{sub-critical}^*$. It then holds that $T\eta\buildrel{d}\over{=}\eta$ if and only if $ ((T^kW)_0)_{k \in \mathbb{Z}} \buildrel{d}\over{=} \theta ((T^kW)_0)_{k \in \mathbb{Z}}$.
\end{prop}
\begin{proof} We clearly have by definition that $\theta \circ \Lambda=\Lambda \circ T$. Since Theorem \ref{bijectionthm} gives that $\Lambda$ is a measurable bijection on $\mathcal{S}_{sub-critical}^*$, the result follows.
\end{proof}

The preceding proposition (combined with Theorem \ref{bijectionthm}) implies that there is one-to-one relationship between the invariant measures for $T$ supported on $\mathcal{S}_{sub-critical}^*$ and the invariant measures for $\theta$ supported on $(\mathbb{Z}_+^{\mathbb{Z}})^*$. In particular, for any probability measure $Q$ on $\mathbb{Z}_+$ with $Q(\{0\}) >0$, $Q^{\otimes \mathbb{Z}} \circ \Lambda$ is invariant under $T$. Moreover, $Q^{\otimes \mathbb{Z}} \circ \Lambda$ is an example of a configuration distribution which is ergodic under $T$ (this is a consequence of Theorem \ref{mrc}, which we prove below).
Note that, if defined in this way, the configuration distribution is stationary under spatial shifts only when $Q$ is the geometric distribution (see Corollary \ref{iidmarkovcharacterization}). On the other hand, the next proposition shows that, under any invariant measure supported on $\mathcal{S}_{sub-critical}^*$, the density profile must be spatially stationary, namely a constant.

\begin{prop}\label{densityconstant}
Suppose $\eta$ is a random particle configuration such that the distribution of the corresponding path encoding $S$ is supported on $\mathcal{S}_{sub-critical}^*$, and $T\eta\buildrel{d}\over{=}\eta$ holds. Then there exists a constant $\rho \in [0, \frac{1}{2})$ such that $\mathbf{P}(\eta_n=1)=\rho$ for all $n \in \mathbb{Z}$.
\end{prop}
\begin{proof} By the ergodic decomposition theorem \cite[Theorem 4.2]{Varad}, we only need to show the result when $\eta$ is ergodic under $T$. Moreover, from Lemma \ref{particleflatdensity}, it will be sufficient to show that $\mathbf{P}(W_n=W_{n-1}=0)$ is constant for $n \in \mathbb{Z}$. Now, if $\eta$ is ergodic under $T$, then we have that
\[\lim_{k \to \infty}\frac{1}{k}\sum_{l=1}^k f_n(T^l \eta) = \mathbf{P}(W_n=W_{n-1}=0),\qquad \mathbf{P}\mbox{-a.s.},\]
where we define $f_n(\eta):=\mathbf{1}_{\{W_n=W_{n-1}=0\}}$. With the notation $\tilde{\sigma}^i_n$ introduced in Section \ref{currentsec}, since $(T^kW_n)_{k \ge0}$ has infinitely many odd gaps between zeros for any $n$, $\mathbf{P}$-a.s., Corollary \ref{oddgapandflatpath} yields that
\[\sum_{l=1}^k f_n(T^l \eta)=\sum_{i=1}^{\infty} \mathbf{1}_{\{\tilde{\sigma}^i_n \le k\}},\qquad\forall k\geq1,\:n\in\mathbb{Z},\qquad \mathbf{P}\mbox{-a.s.}\]
The latter result also gives
\[\left|\sum_{i=1}^{\infty} \mathbf{1}_{\{\tilde{\sigma}^i_{n-1} \le k\}} - \sum_{i=1}^{\infty} \mathbf{1}_{\{\tilde{\sigma}^i_n \le k\}}\right| \leq 1,\qquad\forall k\geq1,\:n\in\mathbb{Z},\qquad \mathbf{P}\mbox{-a.s.},\]
and so we deduce
\[\lim_{k \to \infty}\frac{1}{k}\sum_{l=1}^k f_n(T^l \eta)=\lim_{k \to \infty}\frac{1}{k}\sum_{l=1}^k f_{n-1}(T^l \eta),\qquad \mathbf{P}\mbox{-a.s.}\]
Thus we obtain $\mathbf{P}(W_n=W_{n-1}=0) =\mathbf{P}(W_{n-1}=W_{n-2}=0)$ for all $n \in \mathbb{Z}$, as desired.
\end{proof}

Together with Lemma \ref{critstarremove}, the following result will enable us to replace the set $\mathcal{S}_{sub-critical}^*$ by $\mathcal{S}_{sub-critical}$ when proving Theorem \ref{mrb}. In the proof, we use the notation $T^k\eta_n=\eta^k_n$ and $T^kW_n=w^k_n$ as in Section \ref{currentsec}.

\begin{lem}\label{subcritstarremove} Suppose $\eta$ is a random particle configuration such that the distribution of the corresponding path encoding $S$ is supported on $\mathcal{S}_{sub-critical}$. Then $ ((T^kW)_0)_{k \in \mathbb{Z}} \buildrel{d}\over{=} \theta ((T^kW)_0)_{k \in \mathbb{Z}}$ implies $S \in \mathcal{S}_{sub-critical}^*$, $\mathbf{P}$-a.s.
\end{lem}
\begin{proof}
Suppose $((T^kW)_0)_{k \in \mathbb{Z}} \buildrel{d}\over{=} \theta ((T^kW)_0)_{k \in \mathbb{Z}}$ holds. Applying this property in an argument similar to that of Lemma \ref{Npmlem}, it is possible to deduce that, for $\mathbf{P}$-a.e.\ $\eta$, precisely one of the following holds:
\begin{enumerate}
  \item[(i)] $(w^k_0)_{k \in \mathbb{Z}}$ has infinitely many odd gaps between zeros in both directions;
  \item[(ii)] $(w^k_0)_{k \in \mathbb{Z}}$ has no odd gap between zeros, but has infinitely many zeros in both directions;
  \item[(iii)] $(w^k_0)_{k \in \mathbb{Z}}$ has no zeros.
\end{enumerate}
Since the conditions (i), (ii) and (iii) are invariant under the shift operator $\theta$, and (i) implies $\eta \in \mathcal{S}_{sub-critical}^*$, we only need to show that there is no probability measure satisfying $\mathcal{S}_{sub-critical}$, $\mathbf{P}$-a.s., $((T^kW)_0)_{k \in \mathbb{Z}} \buildrel{d}\over{=} \theta ((T^kW)_0)_{k \in \mathbb{Z}}$, and either (ii) or (iii) holds, $\mathbf{P}$-a.s. Suppose such a probability measure exists. It must then be the case that $(\eta_0^k)_{k \in \mathbb{Z}}$ is alternating $\mathbf{P}$-a.s., since otherwise $(w^k_0)_{k \in \mathbb{Z}}$ must have at least one zero, and moreover at least one odd gap between zeros under the condition that it has infinitely many zeros. By considering the patterns \eqref{patterns}, it follows that we have $w^k_{-1}=w^k_0+1$ and $w^k_{-1}= w^k_0-1$ alternately as $k$ varies. In particular, we obtain that $(T^kW_{-1})_{k \in \mathbb{Z}} \buildrel{d}\over{=} \theta^2 (T^kW_{-1})_{k \in \mathbb{Z}}$. Hence we can conclude for the sequence $(w^k_{-1})_{k \in \mathbb{Z}}$ that (i), (ii) or (iii) holds, $\mathbf{P}$-a.s. However, if (i) holds and $S \in \mathcal{S}^{inv}$, then from the proof of Corollary \ref{keybijection}, we deduce that $(w^k_0)_{k \in \mathbb{Z}}$ also satisfies (i), and so we conclude that either (ii) or (iii) holds for $(w^k_{-1})_{k \in \mathbb{Z}}$, $\mathbf{P}$-a.s. By iterating the argument, we obtain that $(T^kW_{n})_{k \in \mathbb{Z}} \buildrel{d}\over{=} \theta^2 (T^kW_{n})_{k \in \mathbb{Z}}$ for all $n \le 0$ and $(\eta_n^k)_{k \in \mathbb{Z}}$ is alternating for all $n \le 0$, $\mathbf{P}$-a.s. In particular, this implies $W_n \neq W_{n-1}$ for all $n \le 0$, $\mathbf{P}$-a.s., and so $\ell(W)_n =0$ for all $n \le 0$, $\mathbf{P}$-a.s. As a consequence, we find that $\limsup_{n \to -\infty} S_n > -\infty$, and so $S \notin \mathcal{S}_{sub-critical}$, $\mathbf{P}$-a.s., which contradicts the assumption.
\end{proof}

We are now ready to complete the proofs of Theorem \ref{mra}, \ref{mrb} and \ref{mrc}. Apart from combining the various results we have already proved, we check the claims relating to ergodicity.

\begin{proof}[Proof of Theorem \ref{mra}] This a straightforward consequence of Propositions \ref{critsubcrit}, \ref{critprop} and \ref{densityconstant}.
\end{proof}

\begin{proof}[Proof of Theorem \ref{mrb}] Part (a) follows from Proposition \ref{critprop} and Lemma \ref{critstarremove}. We now move to part (b), and so assume that $S\in\mathcal{S}_{critical}$, $\mathbf{P}$-a.s., and also $T\eta\buildrel{d}\over{=}\eta$ holds. Since by Proposition \ref{critsubcrit}, $S\in\mathcal{S}_{critical}^*$, $\mathbf{P}$-a.s., we know that $T\eta=1-\eta$, $\mathbf{P}$-a.s. Hence, the support of $\eta$ must contain at least two points. Clearly, if the support of $\eta$ contains exactly two points, then this set must be of the form $\{\eta^{(0)},1-\eta^{(0)}\}$ for some $\eta^{(0)}\in\{0,1\}^\mathbb{Z}$. Moreover, for $\eta$ to be invariant under $T$, we must have
\[\mathbf{P}\left(\eta=\eta^{(0)}\right)=\frac12=\mathbf{P}\left(\eta=1-\eta^{(0)}\right).\]
Clearly the only two invariant sets in this setting are $\{\eta^{(0)},1-\eta^{(0)}\}$ and the empty set. Since these have probabilities $1$ and $0$, the system is ergodic. Next, suppose the support of $\eta$ contains three distinct sequences $\eta^{(0)},\eta^{(1)}=1-\eta^{(0)},\eta^{(2)}$. In particular, there exists an integer $n$ such that $(\eta^{(i)}_m)_{m=-n}^n$, $i=0,1,2$, are distinct, and also
\[\mathbf{P}\left((\eta_m)_{m=-n}^n=(\eta^{(i)}_m)_{m=-n}^n\right)>0,\qquad \forall i=0,1,2.\]
Now, for any function $f$ it $\mathbf{P}$-a.s.\ holds that
\[\frac{1}{k}\sum_{l=0}^{k-1}f\left(T^l\eta\right)\rightarrow
\frac{1}{2}\left(f(\eta)+f(1-\eta)\right).\]
Hence, if $f=\mathbf{1}_{\{(\eta^{(2)}_m)_{m=-n}^n\}}$, then the above limit is not $\mathbf{P}$-a.s.\ constant, and so $\eta$ is not ergodic under $T$, which completes the proof of (b).
\end{proof}

\begin{proof}[Proof of Theorem \ref{mrc}] Part (a) is a consequence of Propositions \ref{critsubcrit} and \ref{subcritprop}, and Lemma \ref{subcritstarremove}. For part (b), we assume $S\in\mathcal{S}_{sub-critical}$, $\mathbf{P}$-a.s., and also $T\eta\buildrel{d}\over{=}\eta$ holds. From part (a) and Lemma \ref{subcritstarremove}, we can in fact suppose $S\in\mathcal{S}_{sub-critical}^*$, $\mathbf{P}$-a.s., which allows us to apply Theorem \ref{bijectionthm}. Given the latter result and recalling the identity $\theta\circ\Lambda=\Lambda\circ T$, the proof is straightforward.
\end{proof}

Before continuing, we make an additional observation about the boundary conditions of ergodic measures. In particular, Theorems \ref{mrb} and \ref{mrc} give characterisations of all the ergodic measures that are supported on either $\mathcal{S}_{critical}$ or $\mathcal{S}_{sub-critical}$, and the following lemma confirms that for no other distributions on particle configurations can $T$ be ergodic.

{\lem Under the assumptions of Theorem \ref{mra}, if $\eta$ is ergodic under $T$, then the support of $\eta$ is contained within $\mathcal{S}_{sub-critical}$ or $\mathcal{S}_K=\mathcal{S}_{K}^-\cap\mathcal{S}_{K}^+$ for some $K\in\mathbb{N}$.}
\begin{proof} By Propositions \ref{critsubcrit} and \ref{critprop}, we know that the asymptotic behaviour of $\eta$ at $-\infty$ and $+\infty$ must be the same. Note that, for $*$ equal to `sub-critical' or some $K\in\mathbb{N}$, $f=\mathbf{1}_{\mathcal{S}_{*}}$ is a bounded measurable function of $\eta$. Hence, ergodicity implies
\[\frac{1}{k}\sum_{l=0}^{k-1}f\left(T^l\eta\right)\rightarrow\mathbf{P}\left(\eta\in\mathcal{S}_{*}\right).\]
On the other hand, Theorem \ref{characterizationinv} yields
\[\frac{1}{k}\sum_{l=0}^{k-1}f\left(T^l\eta\right)=f\left(\eta\right)\in\{0,1\}.\]
Since the collection of subsets considered is countable, the result follows.
\end{proof}

To complete the section, we establish Theorem \ref{mrd}. The notation $\Psi$, $\Phi$, $R$ and $\tilde{R}$ should be recalled from Section \ref{pathsec}.

\begin{proof}[Proof of Theorem \ref{mrd}] Suppose that $S \in \mathcal{S}^{rev}$, $\mathbf{P}$-a.s. Note that the three conditions of the theorem can be restated as the following:
\begin{equation}\label{cond1}
S\buildrel{d}\over{=}RS,
\end{equation}
\begin{equation}\label{cond2}
\Phi^{-1}S\buildrel{d}\over{=}\tilde{R} \Phi^{-1}S\qquad \Leftrightarrow\qquad \Phi^{-1}S\buildrel{d}\over{=} \Psi^{-1} RS,
\end{equation}
\begin{equation}\label{cond3}
S\buildrel{d}\over{=} \Psi \Phi^{-1}S \qquad\Leftrightarrow\qquad \Psi^{-1}S\buildrel{d}\over{=} \Phi^{-1}S,
\end{equation}
respectively.

Firstly, suppose (\ref{cond1}) and (\ref{cond2}) are satisfied, then
\[TS=\Psi \Phi^{-1}S\buildrel{d}\over{=} \Psi\Psi^{-1} RS= RS\buildrel{d}\over{=}S,\]
which is (\ref{cond3}). Secondly, suppose (\ref{cond1}) and (\ref{cond3}) are satisfied, then
\[\bar{W}=\tilde{R} \Phi^{-1}S=\Psi^{-1}RS\buildrel{d}\over{=}\Psi^{-1}S\buildrel{d}\over{=}\Phi^{-1}S=W,\]
which is condition (\ref{cond2}). Thirdly, suppose (\ref{cond2}) and (\ref{cond3}) are satisfied, then
\[RS\buildrel{d}\over{=} R\Psi \Phi^{-1}S=\Phi \tilde{R}\Phi^{-1}S\buildrel{d}\over{=}\Phi \Phi^{-1}S=S,\]
which is (\ref{cond1}). Hence, any two of the conditions implies the third. Moreover, since under any two of the three conditions we know that $TS\buildrel{d}\over{=}S$, it follows that  $T^kS\buildrel{d}\over{=}S$ for any $k\in \mathbb{Z}$. From this, we can conclude that $S \in \mathcal{S}^{inv}$, $\mathbf{P}$-a.s.
\end{proof}

{\rem\label{noninv} To highlight that the three conditions assumed in Theorem \ref{mrd} are independent, we present some simple examples for which only one of the three conditions is satisfied. We note the examples are stationary, ergodic configurations satisfying \eqref{p0}, and so Lemma \ref{slinrand} implies that the relevant path encodings meet the requirement that $S\in\mathcal{S}^{rev}$, $\mathbf{P}$-a.s. Firstly, consider $\eta$ to be uniformly distributed on the 16 distinct shifts of the repeated concatenations of  \[(1,1,0,1,0,0,0,0,1,0,1,1,0,0,0,0);\]
a section of the corresponding carrier process $W$ is shown in Figure \ref{noninvfig} (along with the other examples discussed here). Clearly $\eta$ satisfies, $\overleftarrow\eta\buildrel{d}\over{=}\eta$, but $\bar{W}\buildrel{d}\over{\neq}W$, and so it must also be the case that $TS\buildrel{d}\over{\neq}S$. Secondly, consider $\eta$ to be uniformly distributed on the 8 distinct shifts of
\[(1,1,0,1,0,0,0,0)\]
(i.e.\ the first half of the configuration described previously). Then $\bar{W}\buildrel{d}\over{=}W$, but $\overleftarrow\eta\buildrel{d}\over{\neq}\eta$, and so $TS\buildrel{d}\over{\neq}S$. Finally, suppose $\eta$ is uniformly distributed on the 9 distinct shifts of \[(1,0,1,0,0,1,0,0,0),\]
then $\eta$ is invariant under $T$, but neither $\overleftarrow\eta\buildrel{d}\over{=}\eta$ nor $\bar{W}\buildrel{d}\over{=}W$ are satisfied.}

\begin{figure}[!htb]
\vspace{0pt}
\centering
\scalebox{0.3}{\includegraphics{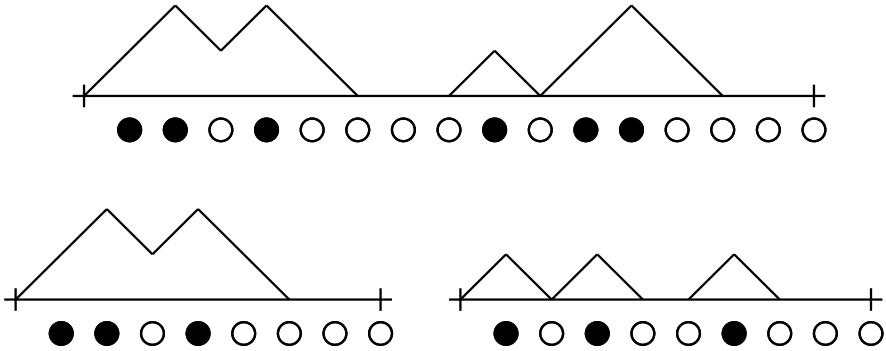}}
\vspace{0pt}
\caption{Configurations discussed in Remark \ref{noninv}.}\label{noninvfig}
\end{figure}

\subsection{Examples of invariant initial configurations}\label{examplessec}

In this section, we introduce the examples described in Theorem \ref{mre}. The proof of the latter result appears in Section \ref{proofsec}, which is where we also check the claims of Remark \ref{mrem}.

\subsubsection{Independent and identically distributed initial configuration}\label{iidsec}
Suppose  $\eta=(\eta_n)_{n\in\mathbb{Z}}$ is given by a sequence of i.i.d.\ Bernoulli($p$) random variables with $p\in[0,\frac12)$. It is then the case that \eqref{p0} is satisfied with $\rho=p$. Furthermore, $S$ is a two-sided simple random walk path satisfying $S_0=0$ and
\[\mathbf{P}\left(S_{n}-S_{n-1}=-1\right)=p=1-\mathbf{P}\left(S_{n}-S_{n-1}=+1\right),\qquad \forall n\in\mathbb{Z},\]
where the increments of $S$ are independent. NB.\ Figure \ref{bbsfig} actually shows $S$ and $TS$ for a (one-sided) realisation of such an $\eta$ with $p=0.45$. By Lemma \ref{slinrand}, we have that the carrier $W=M-S$ is well-defined, $\mathbf{P}$-a.s. Moreover, it is possible to describe the distribution of the carrier explicitly as a reflected random walk; this is the content of the following lemma. Note that, since $p=0$ trivially gives the empty configuration, the associated path encoding obviously satisfies $S\in \mathcal{S}^{rev}$, $\mathbf{P}$-a.s., and the configuration is invariant under $T$; we henceforth exclude this case.

{\lem\label{Wprops} If $\eta$ is given by a sequence of i.i.d.\ Bernoulli($p$) random variables with $p\in(0,\frac12)$, then $W$ is a two-sided stationary Markov chain with transition probabilities given by
\begin{equation}\label{Wprobs}
\mathbf{P}\left(W_{n}=W_{n-1}+j\:\vline\: W_{n-1}\right)=\left\{\begin{array}{ll}
                                                p, & \mbox{if }j=1,\\
                                                1-p, & \mbox{if }W_{n-1}>0\mbox{ and }j=-1,\\
                                                1-p, & \mbox{if }W_{n-1}=0\mbox{ and }j=0.\\
                                              \end{array}\right.
\end{equation}
The stationary distribution of this chain is given by $\pi=(\pi_x)_{x\in\mathbb{Z}_+}$, where
\begin{equation}\label{pidef}
\pi_x=\left(\frac{1-2p}{1-p}\right)\left(\frac{p}{1-p}\right)^x,\qquad \forall x\in\mathbb{Z}_+.
\end{equation}
In particular, the mean and variance of $\pi$ are computed to be equal to $\mu_p$ and $\sigma^2_p$ (see \eqref{meanvar}), respectively.}
\begin{proof} From our assumptions on $\eta$ and Lemmas \ref{twosided} and \ref{slinrand}, we have that $W^{[k]}\rightarrow W$ as $k\rightarrow-\infty$, $\mathbf{P}$-a.s. Moreover, it is clear that $W^{[k]}$ is a Markov chain on $\mathbb{Z}_+$ with transition matrix $P=(P(x,y))_{x,y\in\mathbb{Z}_+}$, as defined by (\ref{Wprobs}), started from $W_{k}^{[k]}=0$. Now, the stationary probability distribution $\pi=(\pi_x)_{x\in\mathbb{Z}_+}$ for $W^{[k]}$ is obtained by solving the detailed balance equations:
\[p \pi_x=(1-p)\pi_{x+1},\qquad \forall x\in\mathbb{Z}_+.\]
In particular, we immediately see that the solution of these equations is given by the formula at \eqref{pidef}. Hence, we obtain for any $x_{-n},\dots,x_n\in \mathbb{Z}_+$ that
\begin{eqnarray*}
\mathbf{P}\left(W_{-n}=x_{-n},\dots,W_n=x_n\right)
&=&\lim_{k\rightarrow-\infty}\mathbf{P}\left(W^{[k]}_{-n}=x_{-n},\dots,W^{[k]}_n=x_n\right)\\
&=&\lim_{k\rightarrow-\infty}P^{-k-n}(0,x_{-n})\prod_{i=-n+1}^nP(x_{i-1}x_i)\\
&=&\pi_{x_{-n}}\prod_{i=-n+1}^nP(x_{i-1}x_i),
\end{eqnarray*}
which yields that $W$ is indeed the relevant two-sided stationary Markov chain, with stationary probability measure given by $\pi$. Finally, the mean and variance of the geometric distribution $\pi$ are easily checked to be equal to the expressions at \eqref{meanvar} by direct computation.
\end{proof}

Since the Markov chain $W$ is reversible (indeed, one can readily verify it satisfies the detailed balance equations), we immediately obtain the following as a simple corollary of this lemma and Theorem \ref{mrd}.

{\cor \label{revcor} If $\eta$ is a sequence of i.i.d.\ Bernoulli($p$) random variables with $p\in(0,\frac12)$, then the three conditions of (\ref{threeconds}) are satisfied. In particular, $\eta$ is invariant in distribution under $T$.}

\subsubsection{Markov initial configuration}\label{markovsec}

Suppose $\eta=(\eta_n)_{n\in\mathbb{Z}}$ is given by a two-sided stationary Markov chain on $\{0,1\}$ with transition probabilities given by
\[\mathbf{P}\left(\eta_{n+1}=1\:\vline\:\eta_n=j\right)=p_j,\qquad j\in\{0,1\},\]
for some parameters $p_0\in (0,1)$, $p_1\in[0,1)$. Note that we recover the i.i.d.\ case of the previous section if $p_0=p_1=p<1/2$. An elementary computation yields that the stationary distribution of this chain is given by
\begin{equation}\label{markovdensity}
\rho=\mathbf{P}\left(\eta_0=1\right)=\frac{p_0}{1-p_1+p_0}.
\end{equation}
To ensure (\ref{p0}) is satisfied, we thus need to assume $p_0+p_1<1$. In particular, under this assumption, the conclusion of Lemma \ref{slinrand} holds, and so the evolution of the BBS is well-defined for all time, $\mathbf{P}$-a.s. As a result, we can define the process $W=M-S$. Whilst this is not a Markov process, we are still able to compute its one-dimensional marginal distribution.

{\lem\label{Wpropsmarkov} If $\eta$ is the two-sided stationary Markov chain described above with
$p_0\in (0,1)$, $p_1\in[0,1)$ satisfying $p_0+p_1<1$, then
\[\mathbf{P}\left(W_{0}=m\right)=\left\{\begin{array}{ll}
                                                \frac{1-p_0-p_1}{(1-p_0)(1+p_0-p_1)}, & \mbox{if }m=0,\\
                                                \frac{p_0(1-p_0+p_1)(1-p_0-p_1)}{(1-p_0)^2(1+p_0-p_1)}\left(\frac{p_1}{1-p_0}\right)^{m-1}, & \mbox{if }m\geq 1.\\
                                              \end{array}\right.\]
In particular, it follows that
\begin{equation}\label{wmarkovexpect}
\mathbf{E}W_0=\frac{p_0(1-p_0+p_1)}{(1+p_0-p_1)(1-p_0-p_1)}.
\end{equation}}
\begin{proof} It is easy to verify that the Markov chain $\eta$ is reversible, and so $\overleftarrow\eta\buildrel{d}\over{=}\eta$ holds. Hence $S\buildrel{d}\over{=}RS$. It follows that $W_0=M_0$ is distributed as $-I_0$, where $I_0$ is the future infimum of $S$. To compute the distribution of $I_0$, let us define
\[q_j:=\mathbf{P}\left(I_0\leq -1\:\vline\:\eta_0=j\right)\]
for $j=0,1$. By the strong Markov property, it is elementary to deduce that, for $m\geq 1$,
\begin{eqnarray*}
  \mathbf{P}\left(-I_0=m\right)&=&\rho\mathbf{P}\left(-I_0=m\:\vline\:\eta_0=1\right)+(1-\rho)\mathbf{P}\left(-I_0=m\:\vline\:\eta_0=0\right)\\
  &=&\rho q_1^m(1-q_1) +(1-\rho)q_0q_1^{m-1}(1-q_1).
\end{eqnarray*}
Hence, we need to compute $q_0$ and $q_1$. Observe that a first-step decomposition of the Markov chain yields
\[q_j=(1-p_j)q_0q_1+p_j\]
for $j=0,1$. Solving these equations gives
\[q_0=\frac{p_0}{1-p_1},\qquad q_1=\frac{p_1}{1-p_0}.\]
(Since $S\in\mathcal{S}_{sub-critical}$, $\mathbf{P}$-a.s., it is easy to exclude the solution $q_0=q_1=1$.) The result now follows by straightforward computations.
\end{proof}

Moreover, we have the following generalisation of Corollary \ref{revcor}.

{\cor\label{markovcor} If $\eta$ is the two-sided stationary Markov chain described above with
$p_0\in (0,1)$, $p_1\in[0,1)$ satisfying $p_0+p_1<1$, then the three conditions of (\ref{threeconds}) are satisfied. In particular, $\eta$ is invariant in distribution under $T$.}
\begin{proof} As noted in the previous proof, $\overleftarrow\eta\buildrel{d}\over{=}\eta$. Furthermore, we note that the process $W$ has the same law as the process $Q$ described in \cite{HMOC}. Hence the claim that $\bar{W}\buildrel{d}\over{=}W$ is \cite[Theorem 2]{HMOC}. The final condition is given by Theorem \ref{mrd}.
\end{proof}

\subsubsection{Conditioning the i.i.d.\ configuration to have bounded solitons}
\label{boundedsec}

In this section, we introduce a particle configuration with bounded solitons, obtained by conditioning the i.i.d.\ configuration of Section \ref{iidsec} to not have any solitons of size greater than $K$, for some fixed $K\in\mathbb{Z}_+$. Note that the event that $\eta$ forms no solitons of size strictly greater than $K$ can alternatively be expressed as the event that the carrier $W$ satisfies $\sup_{n\in\mathbb{Z}}W_n\leq K$. Of course, the latter is an event of 0 probability whenever $\eta$ is Bernoulli($p$), for any $p\in (0,1)$.  However, by taking limits of finite particle configurations, it is possible to make sense of the conditioning in terms of the classical theory of quasi-stationary distributions for Markov chains. In particular, we are able to show that the limiting configuration $\tilde{\eta}$ is stationary, ergodic, $\mathbf{P}$-a.s.\ satisfies (\ref{p0}), and the conditions at (\ref{threeconds}) hold.

We start by defining the limiting carrier process. Let $P=(P(x,y))_{x,y\in\mathbb{Z}_+}$ be the transition matrix of $W$, as defined in (\ref{Wprobs}) (where we now allow any $p\in (0,1)$).  For $K\in \mathbb{Z}_+$ fixed, let $P^{(K)}=(P^{(K)}(x,y))_{x,y\in\{0,\dots,K\}}$ be the restriction of $P$ to $\{0,\dots,K\}$. Since $P^{(K)}$ is an finite, irreducible, substochastic matrix, it admits (by the Perron-Frobenius theorem) a unique eigenvalue of largest magnitude, $\lambda_K$ say. Moreover, $\lambda_K\in(0,1)$ and has a unique (up to scaling) strictly positive eigenvector $h_K=(h_K(x))_{x\in\{0,\dots,K\}}$. Let $\tilde{P}^{(K)}=(\tilde{P}^{(K)}(x,y))_{x,y\in\{0,\dots,K\}}$ be defined by
\begin{equation}\label{tildepkdef}
\tilde{P}^{(K)}(x,y)=\frac{{P}^{(K)}(x,y)h_K(y)}{\lambda_Kh_K(x)},\qquad \forall x,y\in \{0,\dots,K\}.
\end{equation}
It is elementary to check that this is a stochastic matrix. Moreover, the associated Markov chain is reversible, with stationary probability measure given by $\tilde{\pi}^{(K)}=(\tilde{\pi}^{(K)}_x)_{x\in \{0,\dots,K\}}$, where
\begin{equation}\label{tildepi}
\tilde{\pi}^{(K)}_x = c_1 h_K(x)^2{\pi}_x
\end{equation}
for some constant $c_1\in(0,\infty)$ (which may depend on $K$), and $\pi$ is defined as at (\ref{pidef}). Thus the Markov chain in question admits a two-sided stationary version, and we will denote this by $\tilde{W}^{(K)}=(\tilde{W}^{(K)}_n)_{n\in\mathbb{Z}}$.

We will view $\tilde{W}^{(K)}$ as a random carrier process, and write the associated particle configuration $\tilde{\eta}^{(K)}=(\tilde{\eta}^{(K)}_n)_{n\in\mathbb{Z}}$. To justify the claim that $\tilde{\eta}^{(K)}$ is the i.i.d.\ configuration of Section \ref{iidsec} conditioned to have solitons of size no greater than $K$, we have the following result.

{\lem \label{condconst} Fix $K\in \mathbb{Z}_+$. Let ${\eta}=({\eta}_n)_{n\in\mathbb{Z}}$ be an i.i.d.\ Bernoulli($p$) particle configuration for some $p\in (0,1)$. Write $\eta^{[-N,N]}=(\eta^{[-N,N]}_n)_{n\in\mathbb{Z}}$ for the truncated configuration given by $\eta^{[-N,N]}_n=\eta_n\mathbf{1}_{\{-N<n\leq N\}}$. If $W^{[-N,N]}$ is the associated carrier process, then we have the following convergence of conditioned processes:
\[W^{[-N,N]}\:\vline\: \left\{\sup_{n\in\mathbb{Z}}W^{[-N,N]}_n\leq K\right\}\rightarrow\tilde{W}^{(K)}\]
in distribution as $N\rightarrow \infty$. In particular, this implies
\[\eta^{[-N,N]}\:\vline\: \left\{\sup_{n\in\mathbb{Z}}W^{[-N,N]}_n\leq K\right\}\rightarrow\tilde{\eta}^{(K)}\]
in distribution as $N\rightarrow \infty$.}
\begin{proof} The main ingredient for the proof are the following asymptotic result: for any $x,y\in\{0,\dots,K\}$, it holds that
\begin{equation}\label{exitasym}
\mathbf{P}\left(\sup_{0\leq n\leq N}W_n\leq K\:\vline\: W_0=x\right)\sim c_2\lambda_K^Nh_K(x)
\end{equation}
\begin{equation}\label{distasym}
\mathbf{P}\left(W_N=y\:\vline\: \sup_{0\leq n\leq N}W_n\leq K,\:W_0=x\right)\rightarrow c_1c_2^{-1}h_K(x)^{-1}\tilde{\pi}^{(K)}_x
\end{equation}
as $N\rightarrow\infty$, where $c_2\in(0,\infty)$ is a constant (which may depend on $K$), and $c_1$ is the constant from   \eqref{tildepi}. See \cite[Proposition 1]{GT}, for example. Indeed, for any $x_{-n},\dots,x_n\in \{0,\dots,K\}$, we can use this to deduce
\begin{eqnarray}
\lefteqn{\mathbf{P}\left(W^{[-N,N]}_i=x_i,\:i\in\{-n,\dots,n\}\:\vline\: \sup_{m\in\mathbb{Z}}W^{[-N,N]}_m\leq K\right)}\nonumber\\
&=&\mathbf{P}\left(W^{[-N,N]}_{-n}=x_{-n}\:\vline\: \sup_{m\in\mathbb{Z}}W^{[-N,N]}_m\leq K\right) \times \prod_{i=-n+1}^nP(x_{i-1},x_i)\nonumber\\
&&\hspace{-20pt}\times \frac{\mathbf{P}\left(\sup_{-N\leq m\leq -n}W^{[-N,N]}_m\leq K \:\vline\: W^{[-N,N]}_{-n}=x_{-n}\right)
\mathbf{P}\left(\sup_{n\leq m\leq N}W^{[-N,N]}_m\leq K \:\vline\: W_{n}^{[-N,N]}=x_{n}\right)}
{\mathbf{P}\left(\sup_{-N\leq m\leq N}W^{[-N,N]}_m\leq K\:\vline\: W_{-n}^{[-N,N]}=x_{-n}\right)}\label{aaa1}
\end{eqnarray}
Now, since $(W^{[-N,N]}_m)_{m\leq -n}$ and $(W^{[-N,N]}_m)_{m\geq -n}$ are conditionally independent given $W^{[-N,N]}_{-n}$, we have from \eqref{exitasym} and \eqref{distasym} that
\begin{eqnarray}
\lefteqn{\mathbf{P}\left(W^{[-N,N]}_{-n}=x_{-n}\:\vline\: \sup_{m\in\mathbb{Z}}W^{[-N,N]}_m\leq K\right)}\nonumber\\
&=&\mathbf{P}\left(W_{N-n}=x_{-n}\:\vline\: W_0=0,\:\sup_{0\leq m\leq N-n}W_m\leq K\right)\nonumber\\
&&\times \frac{\mathbf{P}\left(\sup_{0\leq m\leq N-n}W_m\leq K\:\vline\: W_0=0\right)
\mathbf{P}\left(\sup_{0\leq m\leq N+n}W_m\leq K\:\vline\: W_0=x_{-n}\right)}
{\mathbf{P}\left(\sup_{0\leq m\leq 2N}W_m\leq K\:\vline\: W_0=0\right)}\nonumber\\
&\rightarrow &\tilde{\pi}^{(K)}_{x_{-n}}.\label{aaa2}
\end{eqnarray}
Moreover, we similarly have that
\begin{eqnarray}
\lefteqn{ \frac{\mathbf{P}\left(\sup_{-N\leq m\leq -n}W^{[-N,N]}_m\leq K \:\vline\: W^{[-N,N]}_{-n}=x_{-n}\right)
\mathbf{P}\left(\sup_{n\leq m\leq N}W^{[-N,N]}_m\leq K \:\vline\: W_{n}^{[-N,N]}=x_{n}\right)}
{\mathbf{P}\left(\sup_{-N\leq m\leq N}W^{[-N,N]}_m\leq K\:\vline\: W_{-n}^{[-N,N]}=x_{-n}\right)}}\nonumber\\
&=&
 \frac{\mathbf{P}\left(\sup_{n\leq m\leq N}W_m\leq K \:\vline\: W_{n}=x_{n}\right)}
{\mathbf{P}\left(\sup_{-n\leq m\leq N}W_m\leq K\:\vline\: W_{-n}=x_{-n}\right)}\nonumber\\
&\rightarrow& \frac{h_K(x_n)}{\lambda_K^{2n}h_K(x_{-n})}.\hspace{280pt}\label{aaa3}
\end{eqnarray}
Combining \eqref{aaa1}, \eqref{aaa2} and \eqref{aaa3}, we thus obtain
\[\mathbf{P}\left(W^{[-N,N]}_i=x_i,\:i\in\{-n,\dots,n\}\:\vline\: \sup_{m\in\mathbb{Z}}W^{[-N,N]}_m\leq K\right)\rightarrow \tilde{\pi}^{(K)}_{x_{-n}}\prod_{i=-n+1}^n
\tilde{P}^{(K)}(x_{i-1},x_i),\]
which completes the proof.
\end{proof}

\begin{rem} For the i.i.d.\ configuration with $p\in(0,\frac{1}{2})$, we have from Lemma \ref{Wprops} that it is possible to define a two-sided stationary version of the carrier process $W$. Straightforward adaptations of the previous proof allow us to alternatively obtain:
\[\eta\:\vline\: \left\{\sup_{n\in[-N,N]}W_n\leq K\right\}\rightarrow\tilde{\eta}^{(K)}.\]
That is, obtain $\tilde{W}^{(K)}$ and $\tilde{\eta}^{(K)}$ by starting with the two-sided infinite configuration, and conditioning on a decreasing sequence of events. We choose to present Lemma \ref{condconst} as conditioning the truncated configuration, however, since this picture is valid throughout the range $p\in(0,1)$.
\end{rem}

As a further consequence of the construction of $\tilde{\eta}^{(K)}$, we have the following result. In the proof, we write $\tilde{S}^{(K)}$ for the path encoding of $\tilde{\eta}^{(K)}$.

{\cor\label{condcor} If $\tilde{\eta}^{(K)}$ and $\tilde{W}^{(K)}$ are as described above, then, for any $p\in (0,1)$, $K\in\mathbb{Z}_+$, $\tilde{\eta}^{(K)}$ is a stationary, ergodic process satisfying
\[\mathbf{P}\left(\tilde{\eta}^{(K)}_0=1\right)<\frac12,\]
and also the three conditions of (\ref{threeconds}). In particular, $\tilde{\eta}^{(K)}$ is invariant in distribution under $T$.}
\begin{proof} The stationarity and ergodicity of $\tilde{\eta}^{(K)}$ readily follow from the corresponding properties of $\tilde{W}^{(K)}$. Moreover, observe that $\mathbf{P}(\tilde{\eta}^{(K)}_0=1)$ is the probability that $\tilde{W}^{(K)}$ has an upcrossing at $0$, i.e.\ it is equal to $\mathbf{P}(\tilde{W}^{(K)}_0-\tilde{W}^{(K)}_{-1}=1)$. Since $\tilde{W}^{(K)}$ is reversible, this is also the probability that $\tilde{W}^{(K)}$ has a downcrossing at $0$, i.e.\ $\mathbf{P}(\tilde{W}^{(K)}_0-\tilde{W}^{(K)}_{-1}=-1)$. Hence we deduce
\begin{equation}\label{expdens}
\rho=\mathbf{P}\left(\tilde{\eta}^{(K)}_0=1\right)=\frac12\left(1-\mathbf{P}\left(\tilde{W}^{(K)}_0=\tilde{W}^{(K)}_1=0\right)\right)=\frac12\left(1-\tilde{\pi}^{(K)}_0\tilde{P}^{(K)}(0,0)\right),
\end{equation}
which is strictly less than $1/2$. For the reversibility of $\tilde{\eta}^{(K)}$, we note that
\begin{eqnarray*}
\tilde{S}^{(K)}&=&\lim_{N\rightarrow\infty}\left(S^{[-N,N]}\:\vline\: \left\{\sup_{n\in\mathbb{Z}}W^{[-N,N]}_n\leq K\right\}\right)\\
&=&\lim_{N\rightarrow\infty}\left(S^{[-N,N]}\:\vline\: \left\{\sup_{m\leq n}\left(S^{[-N,N]}_m-S^{[-N,N]}_n\right)\leq K\right\}\right)\\
&\buildrel{d}\over{=}&\lim_{N\rightarrow\infty}\left(R(S^{[-N,N]})\:\vline\: \left\{\sup_{m\leq n}\left(R(S^{[-N,N]})_m-R(S^{[-N,N]})_n\right)\leq K\right\}\right)\\
&=&\lim_{N\rightarrow\infty}\left(R(S^{[-N,N]})\:\vline\: \left\{\sup_{m\leq n}\left(S^{[-N,N]}_m-S^{[-N,N]}_n\right)\leq K\right\}\right)\\
&=&\lim_{N\rightarrow\infty}\left(R(S^{[-N,N]})\:\vline\: \left\{\sup_{n\in\mathbb{Z}}W^{[-N,N]}_n\leq K\right\}\right)\\
&=&R(\tilde{S}^{(K)}),
\end{eqnarray*}
where we write $S^{[-N,N]}$ for the path encoding of $\eta^{[-N,N]}$, and the limits are distributional ones. Hence we conclude that $\overleftarrow{\tilde{\eta}^{(K)}}\buildrel{d}\over{=}\tilde{\eta}^{(K)}$. The remaining claims are clear since, as already noted, $\tilde{W}^{(K)}$ is reversible, and Theorem \ref{mrd} gives the invariance result.
\end{proof}

As a simple, concrete example, consider the case when $K=1$, $p\in (0,1)$. We can then compute $\tilde{P}^{(1)}$ explicitly to be
\[\left(
    \begin{array}{cc}
      1-\tilde{p} & \tilde{p} \\
      1 & 0 \\
    \end{array}
  \right),\]
where
\begin{equation}\label{tildepdef}
\tilde{p}=1-\frac{2}{1+\sqrt{1+\frac{4p}{1-p}}}.
\end{equation}
Note that $\tilde{p}\rightarrow 0$ as $p\rightarrow 0$, and $\tilde{p}\rightarrow 1$ as $p\rightarrow 1$, and so any value of $\tilde{p}\in(0,1)$ can be obtained in this way. Moreover, observe that this case reduces to the Markov initial configuration example of Section \ref{markovsec} with $p_0=\tilde{p}$, $p_1=0$. For any $K\geq 2$, however, $\tilde{\eta}^{(K)}$ will not be Markov, and so the class of examples constructed in this section fall outside those introduced previously.

The comments of the previous paragraph imply that the processes $\tilde{W}^{(1)}$ obtained from values of $p\in (0,1)$ are the only two-sided stationary Markov chains on $\{0,1\}$ which represent a carrier process of a particle configuration satisfying (\ref{p0}) and which is invariant in distribution under $T$. In the remainder of this section, we aim to extend this remark. In particular, we will show that the processes $\tilde{W}^{(K)}$, $K\in\mathbb{Z}_+$, constructed here and $W$ from Section \ref{iidsec} are in fact the only two-sided stationary Markov carrier processes for which the associated particle configuration satisfies (\ref{p0}) and is reversible, or equivalently by Theorem \ref{mrd}, is invariant in distribution under $T$. See Proposition \ref{boundedonly} and Corollary \ref{unboundedonly} for the precise results. To this end, the following lemma concerning the behaviour of $\lambda_K$ will be useful.

{\lem\label{evaluebeh} For $K\geq 1$,
\[1-\frac{1-p}{\lambda_K}\rightarrow \left\{\begin{array}{cc}
                                       0, & \mbox{ as }p\rightarrow 0, \\
                                       1, & \mbox{ as }p\rightarrow 1.
                                     \end{array}\right. \]}
\begin{proof} As $p\rightarrow 0$, ${P}^{(K)}$ (defined near the start of this section) converges to a lower triangular matrix with diagonal elements $(1,0,\dots,0)$. Since eigenvalues are continuous functions of matrix entries, it follows that $\lambda_K\rightarrow 1$ in this case. This establishes the first limit result.

For $p\rightarrow 1$, we start by computing the determinant of ${P}^{(K)}$. In particular, for $K\geq 2$, we have
\begin{equation}\label{iii}
\det {P}^{(K)} =(1-p)\det Q_K -p(1-p)\det Q_{K-1},
\end{equation}
where $Q_i$ is the $i\times i$ tridiagonal matrix with entries $p$ in the row above the diagonal, 0 on the diagonal, and $1-p$ on the row below the diagonal. It is an elementary exercise to check that
\[\det Q_i = -p^{i/2}(1-p)^{i/2}\mathbf{1}_{\{i\mbox{ even}\}}.\]
Moreover, the identity at \eqref{iii} can be extended to $K=1$ if we, by convention, fix $\det Q_0 =1$. Hence
\[\det {P}^{(K)} = -p^{K/2}(1-p)^{1+K/2}\mathbf{1}_{\{K\mbox{ even}\}}
+p^{(K+1)/2}(1-p)^{(K+1)/2}\mathbf{1}_{\{K\mbox{ odd}\}}.\]
Since $\lambda_K\geq |\det {P}^{(K)}|^{1/(K+1)}$, it follows that
\[\frac{1-p}{\lambda_K}\leq C (1-p)^{\frac{1}{2}-\frac{1}{2(K+1)}}\rightarrow 0,\]
as $p\rightarrow 1$, which is enough to complete the proof.
\end{proof}

We next proceed to check the desired result in the bounded soliton case.

\begin{prop}\label{boundedonly} For each $K\in\mathbb{Z}_+$, the processes $\tilde{W}^{(K)}$ for $p\in (0,1)$ are the only two-sided stationary, irreducible Markov chains on $\{0,1,\dots,K\}$ that represent a carrier process for which the associate particle configuration satisfies (\ref{p0}) and $\overleftarrow{\tilde{\eta}^{(K)}}\buildrel{d}\over{=}\tilde{\eta}^{(K)}$.
\end{prop}
\begin{proof} The result for $K=0$ is obvious, and so for the remainder of the proof we fix $K\geq 1$. Suppose $Z=(Z_n)_{n\in\mathbb{Z}}$ is a Markov chain satisfying the desired properties. We will show that $Z$ must be given by $\tilde{W}^{(K)}$ for some $p\in (0,1)$. Clearly, it will be enough to show equality of transition matrices. To check this, let us start by observing that the general form of the transition matrix of $Z$ is as follows:
\begin{equation}\label{vmatrix}
\left(
    \begin{array}{cccccc}
      1-p_0 & p_0 & 0 & \dots & \dots & 0 \\
      1-p_1 & 0 & p_1 & & & \vdots \\
      0 & 1-p_2 & 0 & \ddots & & \vdots\\
      \vdots & \ddots & \ddots & \ddots &\ddots& 0 \\
      \vdots && \ddots & 1-p_{K-1} & 0 & p_{K-1} \\
      0 & \dots &\dots & 0 & 1 & 0 \\
    \end{array}
  \right),
\end{equation}
where $p_0\in (0,1]$, $p_i\in (0,1)$ for $i=1,\dots,K-1$, and we will also write $p_K=0$.

To show that the transition matrix of $Z$ comes from a one-parameter family, we consider the probabilities of seeing certain particle configurations, as shown in Figure \ref{config}. In particular, the first particle configuration we consider is $\eta_1=1,\dots,\eta_i=1,\eta_{i+1}=0,\eta_{i+2}=1,\dots,\eta_{K+2}=1$, where $i\in\{1,\dots K\}$. Since $Z$ never takes a value greater than $K$, if this configuration appears, then it must be the case that $Z_0=0$. Hence the probability of it occurring is given by
\begin{equation}\label{ppp1}
\pi^Z_0p_0p_1\dots p_{i-1}(1-p_i)p_{i-1}p_i\dots p_{K-1},
\end{equation}
where $\pi^Z$ is the stationary probability measure for $Z$. A similar argument shows the probability of seeing the reverse configuration, $\eta_1=1,\dots,\eta_{K-i+1}=1,\eta_{K-i+2}=0,\eta_{K-i+3}=1,\dots,\eta_{K+2}=1$, is given by
\begin{equation}\label{ppp2}
\pi^Z_0p_0p_1\dots p_{K-i}(1-p_{K-i+1})p_{K-i}p_{K-i+1}\dots p_{K-1}.
\end{equation}
Under the assumption that the distribution of the particle configuration is reversible (and stationarity), the expressions at (\ref{ppp1}) and (\ref{ppp2}) must be equal. We thus find that
\begin{equation}\label{ppp35}
p_{i-1}(1-p_i)=p_{K-i}(1-p_{K-i+1}),\qquad \forall i\in\{1,\dots,K\}.
\end{equation}

\begin{figure}[!htb]
\vspace{0pt}
\centering
\scalebox{0.3}{\includegraphics{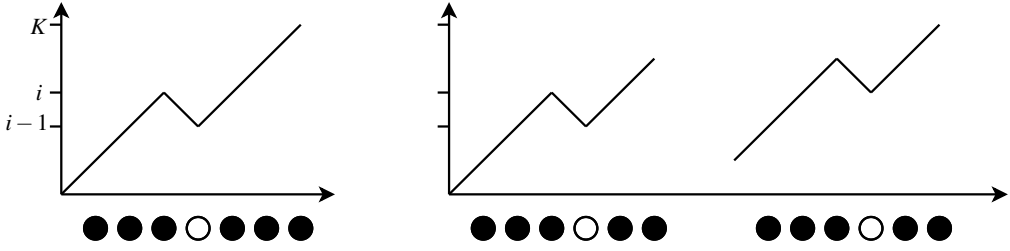}}
\scalebox{0.8}{
\rput(-16.4,1.95){$i-1$}
\rput(-16.2,2.45){$i$}
\rput(-16.2,3.55){$K$}}
\vspace{0pt}
\caption{The left figure shows the particle configuration and path segment of $Z$ considered in the proof of Proposition \ref{boundedonly} with $i=3$, $K=5$. The right figure shows the two possible path segments of $Z$ corresponding to the truncated configuration.}\label{config}
\end{figure}

Next, we consider the truncated configuration \[\eta_1=1,\dots,\eta_i=1,\eta_{i+1}=0,\eta_{i+2}=1,\dots,\eta_{K+1}=1\]
for $i\in \{1,\dots,K-1\}$. Again, see Figure \ref{config}. Since for this configuration we can only deduce that $Z_0\in\{0,1\}$, and so obtain that the probability of it occurring is
\begin{equation}\label{ppp3}
\pi^Z_0p_0p_1\dots p_{i-1}(1-p_i)p_{i-1}p_i\dots p_{K-2}+
\pi^Z_1p_1p_2\dots p_{i}(1-p_{i+1})p_{i}p_{i+1}\dots p_{K-1}.
\end{equation}
Similarly, the reverse configuration now occurs with probability
\begin{eqnarray}
&&\pi^Z_0p_0p_1\dots p_{K-i-1}(1-p_{K-i})p_{K-i-1}p_{K-i}\dots p_{K-2}\nonumber\\
&&+\pi^Z_1p_1p_2\dots p_{K-i}(1-p_{K-i+1})p_{K-i}p_{K-i+1}\dots p_{K-1}.\label{ppp4}
\end{eqnarray}
Now, $Z$ is a reversible Markov chain, and so we have from the detailed balance equations that $\pi^Z_0p_0=\pi^Z_1(1-p_1)$. Applying this, and equating (\ref{ppp3}) and (\ref{ppp4}), we find that
\[(1-p_1)p_{i-1}(1-p_i)+ p_{i}(1-p_{i+1})p_{K-1} = (1-p_1) p_{K-i-1}(1-p_{K-i})+ p_{K-i}(1-p_{K-i+1})p_{K-1}.\]
Appealing to \eqref{ppp35} and rearranging, we deduce from this that
\[(1-p_0)(1-p_1)p_{i-1}(1-p_i)=(1-p_0)(1-p_1) p_{i}(1-p_{i+1}).\]
By irreducibility, we must have that $1-p_1\neq 0$. Moreover, the assumption of \eqref{p0} implies that $1-p_0\neq 0$ (cf.\ the expression for particle density given at \eqref{expdens}). And so we conclude
\[p_{i-1}(1-p_i)= p_{i}(1-p_{i+1}),\qquad \forall i\in \{1,\dots,K-1\}.\]
In particular, this implies that $p_{i-1}(1-p_i)=p_{K-1}$ for $i\in\{1,\dots,K-2\}$, which yields in turn that the parameters $p_0,p_1,\dots,p_{K-2}$ are uniquely determined by $p_{K-1}$ through the relation
\begin{equation}\label{induction}
p_{i-1}=\frac{p_{K-1}}{1-p_i},
\end{equation}
giving us that $Z$ indeed comes from a one-parameter family of transition matrices.

Next, we note that (\ref{induction}) gives an injective map from $p_{K-1}$ to $p_0$. Indeed, suppose we have another chain $Z'$, with transition matrix determined by the parameters $p_0',\dots,p_{K-1}'$. If $p_{K-1}'<p_{K-1}$ and $p_{i}'<p_{i}$ for some $i$, then (\ref{induction}) implies that $p_{i-1}'<p_{i-1}$. Thus $p_{K-1}'<p_{K-1}$ implies $p_0'<p_0$, giving us the desired injectivity. So, the transition matrix of $Z$ is also uniquely determined by $p_0$.

Finally, since the construction of $\tilde{W}^{(K)}$ gives
\[\tilde{P}^{(K)}(0,1)=1-\frac{1-p}{\lambda_K},\]
Lemma \ref{evaluebeh} gives us that all the possible values of $p_0$ are obtained by varying $p\in(0,1)$. Hence, any chain $Z$ of the form described above must be equal to $\tilde{W}^{(K)}$ for some choice of $p\in (0,1)$.
\end{proof}

To complete the section, we extend the previous result to the unbounded case.

\begin{cor}\label{unboundedonly} The processes ${W}$ described in Section \ref{iidsec} for $p\in (0,\frac12)$ are the only two-sided stationary, irreducible Markov chains on $\{0,1,\dots\}$ that represent a carrier process for which the associate particle configuration satisfies (\ref{p0}) and $\overleftarrow{\eta}\buildrel{d}\over{=}{\eta}$.
\end{cor}
\begin{proof} Our first claim is that if $\tilde{W}^{(K)}$ is constructed from parameter $p^{(K)}\in (0,1)$, and $p^{(K)}\rightarrow p\in (0,\frac12]$, then
\begin{equation}\label{plim}
\tilde{P}^{(K)}(x,y)\rightarrow P(x,y),\qquad \forall x,y\geq 0,
\end{equation}
where $\tilde{P}^{(K)}$ is the transition matrix of $\tilde{W}^{(K)}$, and $P$ is the transition matrix of $W$, the carrier for the i.i.d. initial configuration with particle density $p$. To prove this, we observe that \eqref{exitasym} implies
\[\log \lambda_K= \lim_{N\rightarrow \infty}\frac{\log\mathbf{P}_{p^{(K)}}\left(\tau(K)>2N\:\vline\:W_0=0\right)}{N},\]
where $\tau(K)=\inf\{n\geq 0\: :\:W_n=K\}$ is the first hitting time of $K$ by $W$, and we index the probability measure by $p^{(K)}$ to highlight the parameter being considered. Writing $\tau^+(K)=\inf\{n\geq 1\: :\:W_n=0\}$ for the first return time to 0, we thus obtain
\[\log \lambda_K\geq  \lim_{N\rightarrow \infty}\frac{\log\mathbf{P}_{p^{(K)}}\left(\tau^+(0)<\tau(K)\:\vline\:W_0=0\right)^N}{N}=\log\mathbf{P}_{p^{(K)}}\left(\tau^+(0)<\tau(K)\:\vline\:W_0=0\right),\]
Hence
\begin{eqnarray}
\liminf_{K\rightarrow\infty}\lambda_K
&\geq& \liminf_{K\rightarrow\infty}\mathbf{P}_{p^{(K)}}\left(\tau^+(0)<\tau(K)\:\vline\:W_0=0\right)\nonumber\\
&=&\liminf_{K\rightarrow\infty}\left\{1-\frac{1}{1+\left(\frac{1-p^{(K)}}{p^{(K)}}\right)+\dots+\left(\frac{1-p^{(K)}}{p^{(K)}}\right)^2+\dots+\left(\frac{1-p^{(K)}}{p^{(K)}}\right)^K}\right\}\nonumber\\
&=&1,\label{lamlim}
\end{eqnarray}
where the penultimate equality is an elementary gambler's ruin calculation. Next, by definition, we have that
\[\lambda_Kh_K(0)=\tilde{P}^{(K)}(0,0)h_K(0)+\tilde{P}^{(K)}(0,1)h_K(1),\]
which rearranges to give
\begin{equation}\label{h0lim}
\frac{h_K(1)}{h_K(0)}=\frac{\lambda_K-1+p^{(K)}}{p^{(K)}}\rightarrow 1,
\end{equation}
as $K\rightarrow \infty$. Similarly, for $x\geq 1$, we have that
\[\lambda_Kh_K(x)=\tilde{P}^{(K)}(x,x-1)h_K(x-1)+\tilde{P}^{(K)}(x,x+1)h_K(x+1),\]
which rearranges to give
\begin{equation}\label{hlim}
\frac{h_K(x+1)}{h_K(x)}=\frac{\lambda_K}{p^{(K)}}-\frac{(1-p^{(K)})h_K(x-1)}{p^{(K)}h_K(x)},
\end{equation}
and a simple induction argument implies that this also converges to 1. Combining (\ref{lamlim}), (\ref{h0lim}) and (\ref{hlim}) yields (\ref{plim}).

We next suppose that $p^{(K)}\rightarrow p\in(\frac12,1]$, and claim that
\begin{equation}\label{claim2}
\limsup_{K\rightarrow\infty}\lambda_K<1.
\end{equation}
Indeed, first observe that
\begin{eqnarray*}
\mathbf{P}_{p^{(K)}}\left(\tau(K)>2N\:\vline\:W_0=0\right)&\leq&\mathbf{P}_{p^{(K)}}\left(\tau_S(-K)>2N\:\vline \:S_0=0\right)\\
&=&\sum_{n>2N}
\mathbf{P}_{p^{(K)}}\left(\tau_S(-K)=n\:\vline \:S_0=0\right)\\
&=&\sum_{n>2N}\frac{K}{n}
\mathbf{P}_{p^{(K)}}\left(S_n=-K\:\vline \:S_0=0\right)\\
&=&\sum_{n>2N}\frac{K}{n}
\mathbf{P}_{p^{(K)}}\left(\mathrm{Bin}(n,p)=\frac{n+K}{2}\right)\\
&\leq& C_k \left(4(1-p^{(K)})p^{(K)}\right)^{N},
\end{eqnarray*}
where $\tau_S(-K)$ is the hitting time of $-K$ by the path encoding $S$, for the second equality we apply the hitting time theorem for random walks (originally proved in \cite{Otter} for $K=1$, and a modern elementary proof appears in \cite{vdh}), $\mathrm{Bin}(n,p)$ is a binomial random variable with parameters $n$ and $p$, and the final estimate is an elementary exercise to obtain. It follows that
\[\limsup_{K\rightarrow\infty}\lambda_K\leq
\limsup_{K\rightarrow\infty}\limsup_{N\rightarrow\infty}\:\exp\left\{\frac{\log \left(C_k \left(4(1-p^{(K)})p^{(K)}\right)^{N}\right)}{N}\right\}=4(1-p)p<1,\]
which establishes \eqref{claim2}.

Now, let $Z$ be a Markov chain satisfying the desired properties. For each $K\geq 1$, using the procedure described in Lemma \ref{condconst}, one can construct a conditioned version of the chain $Z^{(K)}$ on $\{0,1,\dots,K\}$ analogously to the construction of $\tilde{W}^{(K)}$ from $W$. (In particular, this has transition matrix given by a formula as at \eqref{tildepkdef}.) Observe that, by following the proof of \eqref{plim}, it holds that if $P^{Z,K}$ is the transition matrix of $Z^{(K)}$ and $P^Z$ is the transition matrix of $Z$, then
\begin{equation}\label{plimv}
{P}^{Z,K}(x,y)\rightarrow P^Z(x,y),\qquad \forall x,y\geq 0,
\end{equation}
Indeed, the only part of the proof that is not an immediate adaptation is \eqref{lamlim}, but this is straightforward since if $(\lambda^Z_K)_{K\geq 1}$ are the relevant eigenvalues then we can check that
\begin{equation}\label{vlamlim}
\liminf_{K\rightarrow\infty}\lambda_K^Z\geq \mathbf{P}\left(\tau_Z^+(0)<\tau_Z(K)\:\vline\:Z_0=0\right)=\mathbf{P}\left(\tau_Z^+(0)<\infty\:\vline\:Z_0=0\right)=1,
\end{equation}
where we write $\tau_Z^+(0)$ for the first return time to 0 by $Z$, $\tau_Z(K)$ is the first hitting time of $K$ by $Z$, and the final equality is a consequence of the recurrence of $Z$.

A further important observation is that one may check similarly to Corollary \ref{condcor} that $Z^{(K)}$ is a two-sided stationary carrier process for which the associate particle configuration satisfies (\ref{p0}) and is reversible. In particular, Proposition \ref{boundedonly} therefore tells us that $Z^{(K)}$ has transition matrix given by $\tilde{P}^{(K)}$ for some parameter $p^{(K)}\in (0,1)$. In the remainder of the proof, we let $(p^{(K_i)})_{i\geq 1}$ be a convergent subsequence with limit $p\in [0,1]$. If $p> \frac12$, then we know from (\ref{claim2}) that $\limsup_{i\rightarrow\infty}\lambda_{K_i}<1$. However, we also know that $\lambda_{K_i}=\lambda_{K_i}^Z\rightarrow 1$ by \eqref{vlamlim}. Hence we arrive at a contradiction, and so $p$ must take a value in $[0,\frac12]$. If $p=0$, then we see that
\[P^{Z,K_i}(0,0)=\tilde{P}^{(K_i)}(0,0)=\frac{1-p^{(K_i)}}{\lambda_{K_i}}=\frac{1-p^{(K_i)}}{\lambda^Z_{K_i}}\rightarrow 1-p=1.\]
Combined with \eqref{plimv}, this implies that $P^Z(0,0)=1$, which can not hold since $Z$ is irreducible. Thus we must have $p\in(0,\frac12]$, and \eqref{plim} implies $P^Z=P$. Note that if $p=\frac12$, then it would not be possible to construct a two-sided stationary chain with transition matrix $P$, and so $p\neq \frac12$. Hence, we conclude that the transition matrix of $Z$ is given by $P$ for some $p\in(0,\frac12)$, which completes the proof.
\end{proof}

\subsubsection{Proof of Theorem \ref{mre} and Remark \ref{mrem}}\label{proofsec} That the i.i.d.\ initial configuration ($p\in[0,\frac12)$), Markov initial configuration ($p_0\in(0,1)$, $p_1\in [0,1)$, $p_0+p_1<1$) and bounded soliton examples ($K\in\mathbb{Z}_+$, $p\in (0,1)$) described above have path encodings with distribution supported in $\mathcal{S}_{sub-critical}$ follows from Lemma \ref{slinrand}, \eqref{markovdensity} and Corollary \ref{condcor}. That they are invariant in distribution under $T$ was established in Corollaries \ref{revcor}, \ref{markovcor} and \ref{condcor}. This completes the proof of Theorem \ref{mre}.

For the claims of Remark \ref{mrem}, let us now suppose that $S\in\mathcal{S}^{rev}$, $\mathbf{P}$-a.s., and  $TS\buildrel{d}\over{=}S$ holds. One example of a two-sided stationary, irreducible Markov configuration satisfying these conditions is the empty configuration (as covered by the i.i.d.\ Bernoulli configuration with $p=0$), and to avoid trivialities we exclude this case from the remainder of the discussion. In particular, using the notation of Section \ref{markovsec} for the transition matrix of $\eta$ in the Markov configuration case, we may assume that $p_0>0$. To guarantee the density condition of Theorem \ref{mra}, we moreover require that $p_0+p_1\leq 1$ (recall \eqref{markovdensity}). Now, the case $p_0\in(0,1)$, $p_1\in [0,1)$, $p_0+p_1<1$ is dealt with by Corollary \ref{markovcor}. The only other case that fits the criteria $p_0\in(0,1]$, $p_1\in[0,1]$ and $p_0+p_1=1$ is when $p_0=1$, $p_1=0$. To ensure stationarity, we must take $\eta$ to be the random configuration that takes each of the values  $(\mathbf{1}_{\{n\mbox{ odd}\}})_{n\in\mathbb{Z}}$ and $(\mathbf{1}_{\{n\mbox{ even}\}})_{n\in\mathbb{Z}}$  with probability $\frac12$, which fits into the example of Remark \ref{K-Wrem} with $K=1$. Finally, we deal with the $W$ Markov case. Since $W=\Phi^{-1}S$ for some $S\in\mathcal{S}^{rev}$, it must be the case that $W$ is irreducible on $\{0,1,\dots,K\}$ for some $K=\{0,1,\dots\}\cup\{\infty\}$. Hence, under the further assumption that \eqref{p0} holds, the result follows from Proposition \ref{boundedonly} and Corollary \ref{unboundedonly}. On the other hand, Proposition \ref{critprop} gives us the result when the density is equal to $\frac12$.

\subsection{Particle current and ergodicity for example invariant configurations}\label{currentclt}

We recall from \eqref{icdef} that $C_k$ represents the integrated current, that is, the total number of particles crossing the origin after $k$ time steps of the BBS. It is natural to ask how this quantity behaves as $k\rightarrow\infty$. In this section we study this question for the various sub-critical examples that were introduced in the previous section. Specifically, Theorem \ref{mrf} is proved as Theorems \ref{currentcltthm}, \ref{lll2} and \ref{llll2}. Moreover, we appeal to Theorem \ref{mrc} to conclude that each of the examples is ergodic under $T$; Corollary \ref{ergcormrf} is split across Corollaries \ref{ergcor}, \ref{lll3} and \ref{ergcor3}.

\subsubsection{I.i.d.\ initial configuration}

For the case of an i.i.d.\ initial configuration, we can provide precise results about the current. Indeed, we are able to explicitly describe the distribution of the particles crossing the origin on each time step, $((T^kW)_0)_{k \in\mathbb{Z}}$, as an i.i.d.\ sequence. In conjunction with Lemma \ref{Wprops}, which gives the distribution of $W_0$, we thus immediately obtain from the classical probability theory the following result, which is a more explicit version of Theorem \ref{mrf} in the i.i.d.\ case. Of course, many other detailed properties of i.i.d.\ sequences are also well known. The notation $\mu_p$ and $\sigma^2_p$ should be recalled from \eqref{meanvar}.

{\thm\label{currentcltthm} Suppose $\eta$ is given by a sequence of i.i.d.\ Bernoulli($p$) random variables with $p\in(0,\frac12)$. It then holds that $((T^kW)_0)_{k \in\mathbb{Z}}$ form an i.i.d.\ sequence of random variables, each distributed according to $\pi$, as defined at (\ref{pidef}). In particular, the following conclusions hold.\\
(a) $\mathbf{P}$-a.s.,
\[k^{-1}C_k\rightarrow \mu_p.\]
(b) It holds that
\[\frac{C_k-k\mu_p}{\sqrt{\sigma^2_p k}}\rightarrow N(0,1)\]
in distribution, where $N(0,1)$ is a standard normal random variable;\\
(c) The sequence $(k^{-1}C_k)_{k\geq 1}$ satisfies a large deviations principle with rate function given by
\begin{equation}\label{Cratefunction}
I_C(x):=
\left\{
\begin{array}{ll}
  x\log\left(\frac{(1-p)x}{p(1+x)}\right)+\log\left(\frac{1-p}{(1-2p)(1+x)}\right), & \mbox{if }x\geq 0,\\
  \infty, & \mbox{otherwise.}
\end{array}\right.
\end{equation}
NB. For a definition of what it means for a sequence of random variables to satisfy a large deviations principle with respect to a given rate function, see \cite[Section 1.2]{DZ}.}
\medskip

We also have the following corollary of the above result, which we immediately obtain from Theorem \ref{mrc}.

{\cor\label{ergcor} If $\eta$ is given by a sequence of i.i.d.\ Bernoulli($p$) random variables with $p\in(0,\frac12)$, then the transformation $\eta\mapsto T\eta$ is ergodic.}
\medskip

The proof that $((T^kW)_0)_{k \in\mathbb{Z}}$ is an i.i.d.\ sequence depends crucially on the following result.

{\lem \label{indeplem} If $\eta$ is given by a sequence of i.i.d.\ Bernoulli($p$) random variables with $p\in(0,\frac12)$, then $(({T}S)_{-n})_{n\in\mathbb{Z}_+}$ and $W_0$ are independent.}
\begin{proof} From Theorem \ref{goodsetthm}, we know that
\[W_0=M_0=-\min_{n\geq 0}(TS)_n.\]
We also have from Theorem \ref{mre} that $TS\buildrel{d}\over{=}S$, which means that $(({T}S)_{-n})_{n\in\mathbb{Z}_+}$ is independent of $(({T}S)_{n})_{n\in\mathbb{Z}_+}$. On combining these two observations, the result follows.
\end{proof}

We are now ready to prove the main result of this section.

\begin{proof}[Proof of Theorem \ref{currentcltthm}] We start by checking the independence claim. Firstly, observe that
\[(T^kW)_0=\sup_{m\leq 0}(T^kS)_m.\]
Moreover, for $k\geq l$, $(({T}^kS)_n)_{n\leq 0}$ is a measurable function of $(({T}^lS)_n)_{n\leq 0}$. Hence, we have that $((T^kW)_0)_{k\geq l}$ is $(({T}^lS)_n)_{n\leq 0}$ measurable. From Lemma \ref{indeplem}, we have that $({T}S)_{n\leq 0}$ is independent of $W_0$. Thus we obtain that $W_0$ is independent of $((T^kW)_0)_{k\geq 1}$. To extend this to the full result, we will follow an inductive procedure. In particular, note from Theorem \ref{mre} that $S$ is invariant under ${T}$. Hence combining our previous observation with Proposition \ref{subcritprop} yields that $(T^lW)_0$ is independent of $((T^kW)_0)_{k>l}$ for any $l\geq 0$. This is enough to establish the independence claim for $((T^kW)_0)_{k \in\mathbb{Z}_+}$, and to extend to the two-sided case is straightforward given the invariance of $W$ under $T$. To establish the distributional part of the claim, we simply note that the observation ${T}^kS\sim S$ and the identity $\sup_{m\leq 0}({T}^kS)_m = (T^kW)_0$ imply that $(T^kW)_0\sim W_0$, as desired.

Given the conclusion of the previous paragraph, the result is now standard. For the large deviations principle, we note that Cramer's theorem \cite[Theorem 2.2.3]{DZ} gives the result with rate function
\[I_C(x):=\sup_{\theta\in\mathbb{R}}\left(\theta x-\log M_W(\theta)\right),\]
where
\[M_W(\theta):=\mathbf{E}\left(e^{\theta W_0}\right)=\left\{
\begin{array}{ll}
  \frac{1-2p}{1-p-pe^{\theta}}, & \mbox{if }\theta<\log((1-p)/p),\\
  \infty, & \mbox{otherwise.}
\end{array}\right. \]
It is an elementary computation to deduce from this the expression at \eqref{Cratefunction}.
\end{proof}

\subsubsection{I.i.d.\ initial configuration conditioned to have bounded solitons}

In this section, we show that under the measure given by conditioning the i.i.d.\ configuration with parameter $p\in(0,1)$ not to have solitons larger than $K\in\mathbb{Z}_+$, as made precise in Section \ref{boundedsec}, the current sequence $((T^kW)_0)_{k \in \mathbb{Z}}$ is an irreducible aperiodic reversible two-sided stationary Markov chain on $\{0,1,\dots,K\}$. As a consequence, we show that the integrated current $C_k$ satisfies the following law of large numbers, central limit theorem and large deviations principle, as per the relevant claims of Theorem \ref{mrf}. To state the result, we define
\[\mu_p^K:=\mathbf{E}\left(\tilde{W}_0^{(K)}\right),\qquad \left(\sigma_p^K\right)^2:=\mathrm{Var}\left(\tilde{W}^{(K)}_0\right)+2\sum_{k=1}^{\infty}\mathrm{Cov}\left(\tilde{W}^{(K)}_0,
\left(T^k\tilde{W}^{(K)}\right)_0\right),\]
where $\tilde{W}^{(K)}$ is the process defined in Section \ref{boundedsec}. Moreover, we obtain from another application of Theorem \ref{mrc} that the particle configuration is ergodic under $T$, see Corollary \ref{lll3}.

\begin{thm}\label{lll2} Suppose $\eta$ is the random particle configuration of Section \ref{boundedsec} with parameters $K\in\mathbb{N}$ and $p\in(0,1)$. (NB.\ We exclude $K=0$ to avoid trivialities.)  It is then the case that the current sequence $((T^kW)_0)_{k \in \mathbb{Z}}$ is an irreducible aperiodic reversible two-sided stationary Markov chain on $\{0,1,\dots,K\}$. Moreover, the following conclusions hold.\\
(a) $\mathbf{P}$-a.s.,
\[k^{-1}C_k\rightarrow \mu_p^K.\]
(b) It holds that $(\sigma_p^K)^2\in(0,\infty)$ and
\[\frac{C_k-k\mu_p^K}{\sqrt{(\sigma_p^K)^2 k}}\rightarrow N(0,1)\]
in distribution, where $N(0,1)$ is a standard normal random variable.\\
(c) Let $\tilde{\Pi}^{(K)}=(\tilde{\Pi}^{(K)}(x,y))_{x,y=0}^K$ be the transition matrix of $((T^kW)_0)_{k\in\mathbb{Z}}$, and $\tilde{\Pi}_\theta^{(K)}$ be the exponentially tilted version given by setting
\[\tilde{\Pi}^{(K)}_\theta(x,y)=\tilde{\Pi}^{(K)}(x,y)e^{\theta y},\qquad x,y=0,1,\dots,K.\]
The sequence $(k^{-1}C_k)_{k\geq 1}$ then satisfies a large deviations principle, with rate function given by:
\[\tilde{I}_C^{(K)}(x)=\sup_{\theta\in\mathbb{R}}\left(\theta x-\log \Upsilon\left(\tilde{\Pi}^{(K)}_\theta\right)\right),\]
where $\Upsilon\left(\tilde{\Pi}^{(K)}_\theta\right)$ is the largest eigenvalue of $\tilde{\Pi}^{(K)}_\theta$.
\end{thm}

\begin{cor}\label{lll3} If $\eta$ is the random particle configuration of Section \ref{boundedsec} with parameters $K\in\mathbb{Z}_+$ and $p\in(0,1)$, then the transformation $\eta \to T\eta$ is ergodic.
\end{cor}

We start by proving a general lemma that establishes, for invariant BBSs, the Markov property of the carrier transfers to the Markov property of the current sequence.

\begin{lem} \label{markovequivalence}
Suppose $\eta$ is a random particle configuration such that the distribution of the corresponding path encoding $S$ is supported on $\mathcal{S}^{rev}$, and $T\eta\buildrel{d}\over{=}\eta$ holds. If $(W_n)_{n \in \mathbb{Z}}$ is a two-sided stationary Markov chain, then so is $((T^kW)_0)_{k\in\mathbb{Z}}$.
\end{lem}
\begin{proof}
Since $T\eta\buildrel{d}\over{=}\eta$, $((T^kW)_0)_{k\in\mathbb{Z}}$ is stationary. Thus, assuming the carrier is a two-sided stationary Markov chain, to establish the result it will suffice to check the Markov property for $((T^kW)_0)_{k\in\mathbb{Z}}$ at time $0$, i.e.\ show that, conditional on $W_0$, $((T^kW)_0)_{k\ge 0}$ and $((T^kW)_0)_{k\le 0}$ are independent. To this end, we first observe that the Markov property of the carrier yields that, conditional on $W_0$, $(W_n)_{n\ge0}$ and $(W_n)_{n\le 0}$ are independent. Since $(\eta_n)_{n \ge 1}$ is $(W_n)_{n \ge 0}$-measurable, and $(\eta_n)_{n \le 0}$ is $(W_n)_{n \le 0}$-measurable, it also holds that, conditional on $W_0$, $(\eta_n)_{n \ge 1}$ and $(\eta_n)_{n \le 0}$ are independent. Now, it is clear that $((T^kW)_0)_{k \ge 0}$ is $(\eta_n)_{n \le 0}$-measurable. By considering the reversed dynamics, we similarly deduce that $((T^kV)_0)_{k \le 0}$ is $(\eta_n)_{n \ge 1}$-measurable. Since $S \in \mathcal{S}^{inv}$, $\mathbf{P}$-a.s., we have from Theorem \ref{goodsetthm} that $(T^{k}V)_0=(T^{k-1}W)_0$, $\mathbf{P}$-a.s.\ for any $k\in\mathbb{Z}$. It follows that $((T^kW)_0)_{k \le -1}$ is $(\eta_n)_{n \ge 1}$-measurable. Thus we conclude that, conditional on $W_0$, $((T^kW)_0)_{k\ge0}$ and $((T^kW)_0)_{k\le 0}$ are independent, as desired.
\end{proof}

\begin{rem} By a similar argument to the proof of the preceding lemma, we can also deduce that under the condition $\eta \in \mathcal{S}^{inv}$, $\mathbf{P}$-a.s., if $((T^kW)_0)_{k\ge0}$ and $((T^kW)_0)_{k\le 0}$ are independent conditional on $W_0$, then $(W_n)_{n\ge0}$ and $(W_n)_{n\le 0}$ are also independent conditional on $W_0$. Hence, in addition, if $\eta$ is stationary under the spatial shifts, then the Markov property of the current sequence transfers to the Markov property of the carrier.
\end{rem}

\begin{cor}\label{iidmarkovcharacterization}
Suppose $\eta$ is a random particle configuration such that the distribution of the corresponding path encoding $S$ is supported on $\mathcal{S}^{inv}$, and $\eta$ is stationary under spatial shifts. If $((T^kW)_0)_{k\in\mathbb{Z}}$ is a two-sided stationary Markov chain, then $T\eta\buildrel{d}\over{=}\eta$ if and only if $\eta$ is given by the examples (a) or (c) in Theorem \ref{mre}. In particular, if $((T^kW)_0)_{k\in\mathbb{Z}}$ is an i.i.d. sequence, then $T\eta\buildrel{d}\over{=}\eta$ if and only if its distribution is given by (\ref{pidef}) for $p \in (0,\frac12)$.
\end{cor}

We also have that, for invariant BBSs, spatial symmetry of the carrier process transfers to a temporal symmetry property.

\begin{lem}\label{markovreversible} Suppose $\eta$ is a random particle configuration such that the distribution of the corresponding path encoding $S$ is supported on $\mathcal{S}^{rev}$, and $T\eta\buildrel{d}\over{=}\eta$ holds. If $W\buildrel{d}\over{=}\bar{W}$, then $((T^kW)_0)_{k\in\mathbb{Z}}\buildrel{d}\over{=}((T^{-(k+1)}W)_0)_{k\in\mathbb{Z}}$.
\end{lem}
\begin{proof}
Since $W\buildrel{d}\over{=}\bar{W}$, we have from Theorem \ref{mrd} that $S\buildrel{d}\over{=}RS$. Thus
\begin{align*}
\left((T^kW)_0\right)_{k\in\mathbb{Z}} & =\left((\Phi^{-1}(T^kS))_0\right)_{k\in\mathbb{Z}}\\
&\buildrel{d}\over{=}\left((\Phi^{-1}(T^kRS))_0\right)_{k\in\mathbb{Z}}\\
& =\left((\Phi^{-1}(RT^{-k}S))_0\right)_{k\in\mathbb{Z}}\\
&=\left((\tilde{R} \Psi^{-1}(T^{-k}S))_0\right)_{k\in\mathbb{Z}}\\
&=\left((\Psi^{-1}(T^{-k}S))_0\right)_{k\in\mathbb{Z}} \\
& =\left((T^{-k}V)_0\right)_{k\in\mathbb{Z}}\\
&=((T^{-(k+1)}W)_0)_{k\in\mathbb{Z}},
\end{align*}
where we have applied Lemma \ref{rtrlem} for the third equality, Lemma \ref{dualityrel} for the fourth, and Theorem \ref{goodsetthm} for the final equality.
\end{proof}

The following lemma will be useful when it comes to checking the irreducibility of the current process.

\begin{lem}\label{markovirreducible} Suppose $\eta$ is a random particle configuration such that the distribution of the corresponding path encoding $S$ is supported on $\mathcal{S}^{rev}$, and $T\eta\buildrel{d}\over{=}\eta$ holds. If $W$ is an irreducible two-sided stationary Markov chain on $\{0,1,\dots,K\}$ for some non-negative integer $K$, and satisfies $\mathbf{P}(W_0=W_1=0)>0$, then $\mathbf{P}(W_0=0,\:(TW)_0=l)>0$ for any $0 \le l \le K$.
\end{lem}
\begin{proof} For $1 \le l \le K$, it holds that
\begin{eqnarray*}
\lefteqn{\mathbf{P}\left(W_0=0,\:(TW)_0=l\right)}\\
&\geq&
\mathbf{P}\left(W_0=0,\:\eta_{n}=\mathbf{1}_{n\in\{-2l+1,\dots,-l\}}\:\mbox{for }n\in\{-2l-K+1,\dots,0\},\:S_{-2l-K}=M_{-2l-K}=-K\right)\\
&\geq&
\mathbf{P}\left(W_0=0,\:W_n-W_{n-1}=-\mathbf{1}_{n\in\{-l+1,\dots,0\}}+\mathbf{1}_{n\in\{-2l+1,\dots,-l\}}
\:\mbox{for }n\in\{-2l-K+1,\dots,0\}\right),
\end{eqnarray*}
which is strictly positive by assumption. In the same way,
\[\mathbf{P}(W_0=0,\:(TW)_0=0) \ge \mathbf{P}(W_0=W_{-1}=\cdots=W_{-K}=0) >0.\]
\end{proof}

We next provide sufficient conditions for the current process to be a nice Markov chain.

\begin{prop}\label{lll1} Suppose $\eta$ is a random particle configuration such that the distribution of the corresponding path encoding $S$ is supported on $\mathcal{S}^{rev}$, and $T\eta\buildrel{d}\over{=}\eta$ holds. If $W$ is an irreducible reversible two-sided stationary Markov chain on $\{0,1,\dots,K\}$ for some non-negative $K$, and satisfies $\mathbf{P}(W_0=W_1=0)>0$, then $((T^kW)_0)_{k \in \mathbb{Z}}$ is an irreducible aperiodic reversible two-sided stationary Markov chain on $\{0,1,\dots,K\}$.
\end{prop}
\begin{proof} By Lemma \ref{markovequivalence}, $((T^kW)_0)_{k \in \mathbb{Z}}$ is a two-sided Markov chain, and we also have by assumption that it has state space $\{0,1,\dots,K\}$ and is stationary. Supposing $W\buildrel{d}\over{=}\bar{W}$, from Lemma \ref{markovreversible} and invariance under $T$ we have that \[((T^kW)_0)_{k\in\mathbb{Z}}\buildrel{d}\over{=}((T^{-(k+1)}W)_0)_{k\in\mathbb{Z}}\buildrel{d}\over{=}((T^{-k}W)_0)_{k\in\mathbb{Z}}.\] It follows that $((T^kW)_0)_{k \in \mathbb{Z}}$ is a reversible Markov process. Finally, from Lemma \ref{markovirreducible}, we have that $\mathbf{P}((TW)_0=m \:|\: W_0=0)>0$ for any $m=0,1,\dots,K$, and by reversibility $\mathbf{P}((TW)_0=0 \:|\: W_0=m)>0$ for any $m=0,1,\dots,K$. This establishes the aperiodicity and irreducibility of the chain, and thus completes the proof.
\end{proof}

With the preceding result in place, we can now check the main result of the section. In the proof, for a function $f:\{0,1,\dots,K\}\rightarrow\mathbb{R}$, we use the notation $\tilde{\pi}^{(K)}(f):=\sum_{x=0}^K\tilde{\pi}^{(K)}_xf(x)$.

\begin{proof}[Proof of Theorem \ref{lll2}] By Theorem \ref{mre}, the conclusion of Proposition \ref{lll1} holds for the example of Section \ref{boundedsec} (with parameters $K\in\mathbb{Z}_+$ and $p\in(0,1)$), which establishes the first claim. Given this, part (a) is a straightforward application of the ergodic theorem.

For part (b), first observe that, since $((T^kW)_0)_{k \in \mathbb{Z}}$ is an irreducible aperiodic Markov chain on a finite state space, we have that
\[\max_{x,y\in\{0,1,\dots,K\}}\left|p_k(x,y)-\tilde{\pi}^{(K)}_y\right|\leq C\alpha^k,\qquad \forall k\geq 0,\]
for some $C\in (0,\infty)$, $\alpha\in (0,1)$, where we define $p_k(x,y):=\mathbf{P}((T^kW)_0=y\:|\: W_0=x)$ (see \cite[Theorem 4.9]{LPW}, for example). It follows that if $f:\{0,1,\dots,K\}\rightarrow\mathbb{R}$ is such that $\tilde{\pi}^{(K)}(f)=0$, then
\[g(x):=\sum_{k=0}^\infty P^kf(x),\]
where we write $P$ to be the transition matrix of $((T^kW)_0)_{k \in \mathbb{Z}}$, is well-defined (and finite) for each $x\in \{0,1,\dots,K\}$ (since the sum is absolutely convergent). Moreover, we see that $g$ is the solution of the Poisson equation, namely
\[(I-P)g(x)=f(x),\qquad \forall x\in \{0,1,\dots,K\}.\]
Hence, taking $f(x)=x-\mu_p^K$, we can apply \cite[Theorem 1.1]{KLO} to deduce that
\[\frac{C_k-k\mu_p^K}{\sqrt{k}}=\frac{1}{\sqrt{k}}\sum_{j=0}^{k-1}f\left((T^jW)_0\right)\rightarrow N\left(0,\sigma^2(f)\right),\]
where $\sigma^2(f):=\tilde{\pi}^{(K)}(g^2)-\tilde{\pi}^{(K)}((Pg)^2)$. Now, by considering when we have equality in Jensen's inequality, we see that $\sigma^2(f)=0$ if and only if $g$ is constant. Indeed, it is elementary to check that $\sigma^2(f)=0$ if and only if $g$ is constant on sets of the form $\{y \::\: p_1(x,y)>0\}$, i.e.\ the neighbours of $x$ from the point of view of the Markov chain. Note that, in the present setting, Lemma \ref{markovirreducible} gives that $\{y \::\: p_1(0,y)>0\}=\{0,1,\dots,K\}$, and so zero variance is equivalent to $g$ being constant everywhere, as claimed. However, if $g$ was constant, then we would have $(I-P)g=0$, which is not the case. Hence $\sigma^2(f)\in(0,\infty)$. To complete the proof, we note that $\sigma^2(f)$ can be rewritten as follows:
\begin{eqnarray*}
\sigma^2(f)&=&\tilde{\pi}^{(K)}\left(g^2\right)-\tilde{\pi}^{(K)}\left((g-f)^2\right)\\
&=&\tilde{\pi}^{(K)}\left(f(2g-f)\right)\\
&=&\tilde{\pi}^{(K)}\left(f(I+P)g\right)\\
&=&\tilde{\pi}^{(K)}\left(f^2\right)+2\sum_{k=1}^\infty\tilde{\pi}^{(K)}\left(fP^kf\right)\\
&=&(\sigma_p^K)^2.
\end{eqnarray*}

Finally, since the state space is finite, part (c) is an immediate consequence of the large deviations principle stated as \cite[Theorem 3.1.2]{DZ}, for example.
\end{proof}

\subsubsection{Markov initial condition} In this section, we show that for the Markov initial configuration of Section \ref{markovsec} with parameters $p_0\in (0,1)$, $p_1\in (0,1)$ satisfying $p_0+p_1<1$, the two-state process $(T^k\eta_0, (T^kW)_0)_{k \in \mathbb{Z}}$ is an irreducible aperiodic two-sided stationary Markov chain on $\Sigma$, where
\[\Sigma:=\left\{(a,b) : \:a \in \{0,1\},\: b \in \mathbb{Z}_+,\: (a,b) \neq (1,0)\right\}.\]
After undertaking some additional work compared with previous sections to handle the fact that the state space of the Markov chain is infinite, we are able to show that the integrated current satisfies the following law of large numbers, central limit theorem and large deviations principle. In the statement of the theorem, we use the notation $\mu_{p_0,p_1}$ and $\sigma^2_{p_0,p_1}$ from \eqref{mup0p1} and (\ref{var2}), respectively. We moreover obtain from this result and Theorem \ref{mrc} that the Markov initial configuration is ergodic under $T$ (see Corollary \ref{ergcor3}). Note that the results exclude the case $p_0\in(0,1)$, $p_1=0$ (which was included in Theorem \ref{mre}), since in this case the carrier $W$ is equal to the configuration $\eta$, and so only has state space $\{0,1\}$. Observe, however, that we already dealt with this case in the previous section, since this corresponds to the bounded soliton example with $K=1$ and $p$ determined by taking $\tilde{p}=p_0$ in (\ref{tildepdef}).

\begin{thm}\label{llll2} Suppose $\eta$ is the Markov initial configuration of Section \ref{markovsec} with parameters $p_0\in (0,1)$, $p_1\in (0,1)$ satisfying $p_0+p_1<1$. It is then the case that $(T^k\eta_0, (T^kW)_0)_{k \in \mathbb{Z}}$ is an irreducible aperiodic two-sided stationary Markov chain on $\Sigma$. Moreover, the following statements hold.\\
(a) $\mathbf{P}$-a.s.,
\[k^{-1}C_k\rightarrow \mu_{p_0,p_1}.\]
(b) It holds that
\[\frac{C_k-k\mu_{p_0,p_1}}{\sqrt{\sigma_{p_0,p_1}^2k}}\rightarrow N(0,1)\]
in distribution, where $N(0,1)$ is a standard normal random variable.\\
(c) The sequence $(k^{-1}C_k)_{k\geq 1}$ satisfies a large deviations principle, with rate function as described at \eqref{ratef}.
\end{thm}

\begin{cor}\label{ergcor3} If $\eta$ is the Markov initial configuration of Section \ref{markovsec} with parameters $p_0\in (0,1)$, $p_1\in (0,1)$ satisfying $p_0+p_1<1$, then the transformation $\eta \to T\eta$ is ergodic.
\end{cor}

We start by checking the Markov property of the relevant process.

\begin{lem} If $\eta$ is the Markov initial configuration of Section \ref{markovsec} with parameters $p_0\in (0,1)$, $p_1\in (0,1)$ satisfying $p_0+p_1<1$, then $(T^k\eta_0, (T^kW)_0)_{k \in \mathbb{Z}}$ is a two-sided stationary Markov chain on $\Sigma$.
\end{lem}
\begin{proof} Since $(\eta_n)_{n \ge 0}$ and $(\eta_n)_{n \le 0}$ are independent conditional on $\eta_0$, and the random variable $W_0$ is $(\eta_n)_{n \le 0}$-measurable, $(\eta_n)_{n \ge 0}$ and $(\eta_n)_{n \le 0}$ are independent conditional on $(\eta_0,W_0)$. Now, observe that $(T^k\eta_0, (T^kW)_0)_{k \ge 0}$ is $(\eta_n)_{n \le 0}$-measurable, and $(T^k\eta_0, (T^kV)_0)_{k \le 0}$ is $(\eta_n)_{n \ge 0}$-measurable. Since $S \in \mathcal{S}^{inv}$, $\mathbf{P}$-a.s., we have from Theorem \ref{goodsetthm} that $T^{k}V_0=T^{k-1}W_0$ for any $k$, $\mathbf{P}$-a.s. Hence $(T^k\eta_0, (T^kW)_0)_{k \le -1}$ is $(\eta_n)_{n \ge 0}$-measurable. Thus we conclude that the sequences $(T^k\eta_0, (T^kW)_0)_{k \ge 0}$ and $(T^k\eta_0, (T^kW)_0)_{k \le 0}$ are independent conditional on $(\eta_0,W_0)$, which establishes the Markov property at $k=0$. Since $(T^k\eta_0, (T^kW)_0)_{k \in \mathbb{Z}}$ is stationary under the natural shift, it must therefore be a two-sided stationary Markov chain.
\end{proof}

In the next lemma, we calculate the transition probabilities of the Markov chain explicitly.

\begin{lem}\label{tmatrix} If $\eta$ is the Markov initial configuration of Section \ref{markovsec} with parameters $p_0\in (0,1)$, $p_1\in (0,1)$ satisfying $p_0+p_1<1$, then the transition matrix $P=(p_{(i,l),(j,m)})_{(i,l),(j,m) \in \Sigma}$ of the Markov chain $(T^k\eta_0, (T^kW)_0)_{k \in \mathbb{Z}}$ is given by the following:
\[p_{(i,l),(j,m)}=\left\{
\begin{array}{ll}
 1-q_0, & \mbox{if }(i,l)=(0,0),\:(j,m)=(0,0),\\
 (1-q_j)q_0(1-q_1)(1-q_0q_1)^{-1}q_1^{m-j},  &  \mbox{if }(i,l)=(0,0),\:j\in\{0,1\},\:m\geq 1,\\
  (q_0q_1)^{i}(1-q_1)q_1^{m-1}, & \mbox{if }i\in\{0,1\},\:j=1-i,\:l,m\geq 1,\\
  1-q_0q_1, & \mbox{if }i=1,\;l\geq 1,\:(j,m)=(0,0),\\
  0, & \mbox{otherwise}.
\end{array}
\right.
\]
\end{lem}
\begin{proof}
First suppose $j=1-i$. Note that $T\eta_0=1$ implies $\eta_0=0$, and $T\eta_0=0$, $W_0 \ge 1$ implies $\eta_0=1$. Then, since $(1,0) \notin \Sigma$,
\begin{eqnarray*}
\lefteqn{\mathbf{P}(T\eta_0=1-i,\:(TW)_0=m\:|\:\eta_0=i,\:W_0=l)  }\nonumber\\ &=&\frac{\mathbf{P}(\eta_0=i,\:W_0=l,\:T\eta_0=1-i,\:(TW)_0=m)}{\mathbf{P}(\eta_0=i,\:W_0=l)} \nonumber\\
& =&\frac{\mathbf{P}(W_0=l,\:T\eta_0=1-i,\:(TW)_0=m)}{\mathbf{P}(\eta_0=i,\:W_0=l)}\nonumber\\
&=& \frac{\mathbf{P}(TV_0=l,\:T\eta_0=1-i,\:(TW)_0=m)}{\mathbf{P}(\eta_0=i,\:W_0=l)}\nonumber\\
& =&\frac{\mathbf{P}(V_0=l,\:\eta_0=1-i,\:W_0=m)}{\mathbf{P}(\eta_0=i,\:W_0=l)}\nonumber\\
& =& \frac{\mathbf{P}(V_0=l \:|\: \eta_0=1-i) \mathbf{P}(W_0=m\:|\:\eta_0=1-i)\mathbf{P}(\eta_0=1-i)}{\mathbf{P}(W_0=l\:|\:\eta_0=i)\mathbf{P}(\eta_0=i)},
\end{eqnarray*}
where for the final equality, we use the Markov property of $(\eta_n)_{n \in \mathbb{Z}}$ at $n=0$. Now, using the notation of the proof of Lemma \ref{Wpropsmarkov}, we have that
\[
\frac{\mathbf{P}(\eta_0=1-i)}{\mathbf{P}(\eta_0=i)}=\left(\frac{\rho}{1-\rho}\right)^{1-2i}=q_0^{1-2i},\]
\[
\mathbf{P}(V_0=l \:|\: \eta_0=1-i)=q_{1-i}q_1^{l-1}(1-q_1)\mathbf{1}_{\{l\geq 1\}}+(1-q_{1-i})\mathbf{1}_{\{l=0\}},\]
\begin{eqnarray*}
\mathbf{P}(W_0=l \:|\: \eta_0=i)&=&\mathbf{P}(V_0=l \:|\: \eta_1=i)\nonumber\\
&=&\mathbf{P}(V_0=l+1-2i\:|\: \eta_0=i)+\mathbf{P}(V_0=0\:|\: \eta_0=0)\mathbf{1}_{\{i=l=0\}}\nonumber\\
&=&q_{i}q_1^{l-2i}(1-q_1)\mathbf{1}_{\{l\geq 2i\}}+(1-q_{i})\mathbf{1}_{\{i=l\}},
\end{eqnarray*}
and we can similarly compute $\mathbf{P}(W_0=m \:|\: \eta_0=1-i)$. Putting these together yields the relevant transition probabilities.

Next, we consider the $i=j=l=0$ case. Proceeding similarly to above, we deduce
\begin{eqnarray*}
\lefteqn{ \mathbf{P}(T\eta_0=0,\:(TW)_0=m\:|\eta_0=0,\:W_0=0)  }\\
& =&\frac{\mathbf{P}(W_0=0,\:T\eta_0=0,\:(TW)_0=m)}{\mathbf{P}(\eta_0=0,\:W_0=0)}\\
& = &\frac{\mathbf{P}(TV_0=0,\:T\eta_0=0,\:(TW)_0=m)}{\mathbf{P}(\eta_0=0,\:W_0=0)}\\
& =&\frac{\mathbf{P}(V_0=0,\:\eta_0=0,\:W_0=m)}{\mathbf{P}(\eta_0=0,\:W_0=0)}\\
& =& \frac{\mathbf{P}(V_0=0 \:| \:\eta_0=0) \mathbf{P}(W_0=m\:|\:\eta_0=0)\mathbf{P}(\eta_0=0)}{\mathbf{P}(W_0=0\:|\:\eta_0=0)\mathbf{P}(\eta_0=0)}\\
&=& \frac{\mathbf{P}(V_0=0 \:| \:\eta_0=0) \mathbf{P}(W_0=m\:|\:\eta_0=0)}{\mathbf{P}(W_0=0\:|\:\eta_0=0)}.
\end{eqnarray*}
From this, we obtain the remaining non-zero transition probabilities. Indeed, it is easy to check that the transition probability is $0$ in the other cases.
\end{proof}

Applying the explicit transition probabilities of the previous lemma, we now describe a decomposition of the Markov chain $(T^k\eta_0, (T^kW)_0)_{k \in \mathbb{Z}}$ into a simpler skeleton Markov chain that takes values in a finite state space, and an independent sequence of i.i.d.\ geometric random variables. More precisely, consider the three point set $\Sigma^*=\{\mathbf{00},\mathbf{0},\mathbf{1}\}$, where $\mathbf{00}=\{(0,0)\}$, $\mathbf{0}=\{(0,\ell):\:\ell \ge 1\}$ and $\mathbf{1}=\{(1,\ell):\:\ell \ge 1\}$. We then let $X=(X_k)_{k\in\mathbb{Z}}$ be a two-sided stationary Markov chain on $\Sigma^*$, with transition probability matrix
\begin{equation}\label{pstar}
P^*:=\left(
  \begin{array}{ccc}
    1-q_0 & \frac{(1-q_0)q_0q_1}{1-q_0q_1}   &\frac{(1-q_1)q_0}{1-q_0q_1} \\
    0  &  0 & 1 \\
   1-q_0q_1  &  q_0q_1 & 0 \\
  \end{array}
\right),
\end{equation}
and $\xi=(\xi_k)_{k \ge 1}$ be an  i.i.d.\ geometric sequence with support $\{1,2,\dots\}$ and parameter $q_1$, independent of $X$. It is then an elementary exercise to check that $(T^k\eta_0, (T^kW)_0)_{k \in \mathbb{Z}}$ can be coupled with these random variables in such a way that: for $k\in\mathbb{Z}$,
\[(T^k\eta)_0=\mathbf{1}_{\{X_k=\mathbf{1}\}},\qquad  (T^kW)_0= \mathbf{1}_{\{X_k \in \{\mathbf{0},\mathbf{1}\}\}}\xi_k.\]
We note that the invariant measure $\pi^X$ for the Markov chain $X$ is given by:
\begin{equation}\label{pixdef}
\pi^X(\mathbf{00})=\frac{1-q_0q_1}{1+q_0},\qquad \pi^X(\mathbf{0})=\frac{q_0q_1}{1+q_0},\qquad \pi^X(\mathbf{1})=\frac{q_0}{1+q_0}.
\end{equation}
Hence, the invariant measure $\pi$ for the Markov chain $(T^k\eta_0, (T^kW)_0)_{k \in \mathbb{Z}}$ satisfies:
\begin{equation}\label{invmeas11}
\pi((i,m))=\left\{
\begin{array}{ll}
  \frac{1-q_0q_1}{1+q_0}, & \mbox{if }(i,m)=(0,0),\\
  \frac{q_0(1-q_1)}{1+q_0}q_1^m,  & \mbox{if }i=0,\:m\geq 1,\\
  \frac{q_0(1-q_1)}{1+q_0}q_1^{m-1} & \mbox{if }i=1,\:m\geq 1.
\end{array}\right.
\end{equation}
With these preparations in place, we can now prove the main result of the section.

\begin{proof}[Proof of Theorem \ref{llll2}] Applying the preceding two lemmas, we readily obtain that the process $(T^k\eta_0, (T^kW)_0)_{k \in \mathbb{Z}}$ is an irreducible aperiodic two-sided stationary Markov chain on $\Sigma$. Part (a) is then a straightforward consequence of the ergodic theorem.

We now prove (b). To do this, we will apply \cite[Theorem 1.1]{KLO}. In particular, let $f:{\Sigma}\rightarrow \mathbb{R}$ be given by $f(i,l)=l-\mu_{p_0,p_1}$, so that $\pi(f)=0$ and, by (\ref{invmeas11}), $\pi(f^2)<\infty$ (where for a function $h$ on $\Sigma$, we write $\pi(h)$ for the expectation of $h$ with respect to $\pi$). Moreover, let $g:{\Sigma}\rightarrow \mathbb{R}$ be defined by
\[g(i,l)=l+\sum_{\mathbf{x}\in\{\mathbf{00},\mathbf{0},\mathbf{1}\}}\alpha_\mathbf{x}\mathbf{1}_{(i,l)\in\mathbf{x}},\]
where
\[\alpha_{\mathbf{00}}=\frac{-1+q_0}{(1+q_0)(1-q_1)},\qquad \alpha_{\mathbf{0}}=0,\qquad \alpha_{\mathbf{1}}=\frac{-1+q_0q_1}{(1+q_0)(1-q_1)}.\]
It is then an elementary, albeit slightly lengthy, exercise to check that $g$ is a solution of the Poisson equation for $f$, i.e.
\begin{equation}\label{Poisson}
(I-P)g=f,
\end{equation}
where $P=(p_{(i,l),(j,m)})_{(i,l),(j,m)\in\Sigma}$ is the transition matrix of $(T^k\eta_0, (T^kW)_0)_{k \in \mathbb{Z}}$, as given by Lemma \ref{tmatrix}. Since $\pi(g^2)<\infty$, we can thus immediately apply \cite[Theorem 1.1]{KLO} to deduce the desired central limit theorem, with limiting variance given by
\[\sigma^2=\pi\left(g^2\right)-\pi\left((Pg)^2\right)\in[0,\infty).\]
Thus, to complete the proof of (b), it remains to show that $\sigma^2$ can be written as at \eqref{var1} and \eqref{var2}. To do this, we first observe that iterating the Poisson equation (\ref{Poisson}) yields: for $k\geq 1$,
\begin{eqnarray*}
\sigma^2&=&\pi(f^2)+2\sum_{l=1}^{k-1}\pi(fP^lf)+2\pi(fP^kg)\\
&=&\mathrm{Var}\left(W_0\right)+2\sum_{l=1}^{k-1}\mathrm{Cov}\left(W_0,
\left(T^lW\right)_0\right)+2\pi(fP^kg).
\end{eqnarray*}
As a result, to deduce the expression at \eqref{var1}, it will be enough to show that $\pi(fP^kg)\rightarrow 0$ as $k\rightarrow\infty$. To this end, we start by noting that the matrix $P^*$ from \eqref{pstar} is diagonalisable, with eigenvalues $0,-q_0,1$. It follows that, for any pair of subsets $A,B\subseteq \{\mathbf{00},\mathbf{0},\mathbf{1}\}$, there exists a constant $C_{A,B}$ such that, for $k\geq 1$,
\begin{equation}\label{convtostat}
\mathbf{P}\left(X_0\in A,\:X_k\in B\right)=\pi^X(A)\pi^X(B)+C_{A,B}(-q_0)^k,
\end{equation}
where we recall the notation $\pi^X$ from \eqref{pixdef}. Hence, using the decomposition of the Markov chain $(T^k\eta_0, (T^kW)_0)_{k \in \mathbb{Z}}$ described prior to this proof,
\begin{eqnarray*}
\pi(fP^kg)&=&\mathrm{Cov}\left(W_0,\left(T^kW\right)_0+
\sum_{\mathbf{x}\in\{\mathbf{00},\mathbf{0},\mathbf{1}\}}\alpha_\mathbf{x}\mathbf{1}_{\left(\left(T^k\eta\right)_0,\left(T^kW\right)_0\right)\in\mathbf{x}}\right)\\
&=&\mathrm{Cov}\left(\xi_0\mathbf{1}_{\{X_0\neq \mathbf{00}\}},\xi_k\mathbf{1}_{\{X_k\neq \mathbf{00}\}}+\alpha_{X_k}\right)\\
&=&\left(\frac{1}{(1-q_1)^2}C_{\{\mathbf{0},\mathbf{1}\},\{\mathbf{0},\mathbf{1}\}}+
\frac{1}{1-q_1}\sum_{\mathbf{x}\in\{\mathbf{00},\mathbf{0},\mathbf{1}\}}C_{\{\mathbf{0},\mathbf{1}\},\{\mathbf{x}\}}
\alpha_\mathbf{x}\right)(-q_0)^k,
\end{eqnarray*}
which clearly converges to zero. This confirms $\sigma^2$ can be written as at \eqref{var1}. To evaluate the expression explicitly and thereby arrive at \eqref{var2}, we again appeal to (\ref{convtostat}) to deduce that
\[\mathrm{Cov}\left(W_0,\left(T^kW\right)_0\right)=\mathrm{Cov}\left(W_0,\left(TW\right)_0\right)(-q_0)^{k-1}.\]
Consequently,
\[\sigma^2=\mathrm{Var}\left(W_0\right)+\frac{2\mathrm{Cov}\left(W_0,\left(TW\right)_0\right)}{1+q_0}.\]
Elementary calculations yield that
\[\mathrm{Var}\left(W_0\right)=\frac{q_0(1+q_1)^2}{(1+q_0)^2(1-q_1)^2},\]
and
\[\mathrm{Cov}\left(W_0,\left(TW\right)_0\right)=\frac{q_0\left(q_1(1+q_0)^2-q_0(1+q_1)^2\right)}{(1+q_0)^2(1-q_1)^2}.\]
Putting the preceding three formulae together, we obtain \eqref{var2}, as desired.

For part (c), we apply the argument of Theorem 3.1.2 and Exercise 3.1.4 of \cite{DZ}. In particular, this yields
a large deviations principle for $(k^{-1}C_k)_{k\geq 1}$ with rate function $I^*_C$ described as follows. Let $M^\mathbf{x}_W$ be a function given by
\[M_W^\mathbf{x}(\theta):=\left\{
\begin{array}{ll}
  1, & \mbox{if }\mathbf{x}=\mathbf{00}, \\
  \frac{(1-q_1)e^\theta}{1-q_1e^\theta}\mathbf{1}_{\{\theta<-\log q_1\}}+\infty\mathbf{1}_{\{\theta\geq-\log q_1\}}, & \mbox{if }\mathbf{x}\in\{\mathbf{0},\mathbf{1}\};
\end{array}
\right.
\]
this is the moment generating function for $(T^kW)_0$ conditional on $X_k=\mathbf{x}$. Let $P^*_\theta$ be the matrix defined by setting
\[P^*_\theta(\mathbf{x},\mathbf{y}):=P^*(\mathbf{x},\mathbf{y})M_W^\mathbf{y}(\theta),\qquad \mathbf{x},\mathbf{y}\in\Sigma^*,\]
and $\Upsilon(P^*_\theta)$ be its largest eigenvalue. The rate function $I^*_C$ is then given by:
\begin{equation}\label{ratef}
I^*_C(x)=\sup_{\theta\in\mathbb{R}}\left(\theta x - \log\Upsilon\left(P^*_\theta\right)\right).
\end{equation}
The only adaptation to \cite{DZ} is due to the fact that $M_W^\mathbf{0}(\theta)=M_W^\mathbf{1}(\theta)$ is finite if and only if $\theta<-\log q_1$, whereas in \cite{DZ} it is assumed that the corresponding moment generating functions are finite everywhere. However, it is straightforward to show that the same argument applies under the following assumption:
\[\lim_{\theta \to (-\log q_1)^-}\lim_{k\rightarrow\infty}\frac{1}{k}\log\mathbf{E}\left(e^{\theta C_k}\right)=\infty.\]
(That the inner limit exists and is equal to $\log\Upsilon(P^*_\theta)$ is readily checked as in the proof of \cite[Theorem 3.1.2]{DZ}.) For this, we observe that, for $\theta\geq 0$,
\[\mathbf{E}\left(e^{\theta C_k}\right)\geq\mathbf{P}\left(\sum_{l=0}^{k-1}\mathbf{1}_{\{X_l=\mathbf{0}\}} \geq k\pi^X(\mathbf{0})/2\right)
 M_W^{\mathbf{0}}(\theta)^{k\pi^X(\mathbf{0})/2}.\]
Hence, from the law of large numbers, we obtain
\[\lim_{\theta \to (-\log q_1)^-}\lim_{k\rightarrow\infty}\frac{1}{k}\log\mathbf{E}\left(e^{\theta C_k}\right)
\geq \lim_{\theta \to (-\log q_1)^-}\frac{\pi^X(\mathbf{0})}{2}\log M_W^{\mathbf{0}}(\theta)=\infty,\]
as desired.
\end{proof}

\subsection{Distance travelled by tagged particle for i.i.d.\ initial configuration}\label{distancesec}

In this section, we study the progress of a single tagged particle in the BBS when the initial configuration is given by an i.i.d.\ sequence of Bernoulli($p$) random variables for some $p\in(0,\frac12)$. (See Remarks \ref{fiforem} and \ref{liforem} for some comments on the Markov initial configuration and bounded soliton examples from Theorem \ref{mre}, and Remark \ref{boundedrem} for discussion of the bounded Markov carrier example from Remark \ref{K-Wrem}.) We recall $X^F=(X_k^F)_{k\geq 0}$ is the position of the tagged particle after $k$ evolutions of the BBS under the FIFO scheme, and $X^L=(X_k^L)_{k\geq 0}$ for the corresponding position under the LIFO scheme, as described in the introduction. The main result we prove, which is a more detailed statement of Theorem \ref{mrg}, is as follows. In particular, this establishes laws of large numbers for $X^F$ and $X^L$, demonstrates $X^F_k$ admits fluctuations of $\sqrt{k}$ around $kv_p$, and gives a central limit theorem for $X^L_k$.

{\thm \label{distancethm} Suppose $\eta$ is given by a sequence of i.i.d.\ Bernoulli($p$) random variables with $p\in(0,\frac12)$.\\
(a) $\mathbf{P}$-a.s.,
\[\frac{X^F_k}{k}\rightarrow v_p,\qquad \frac{X^L_k}{k}\rightarrow v_p,\]
where $v_p$ is defined as at (\ref{vpdef}).\\
(b)(i) The sequence
\[\left(\left|\frac{X_k^F-kv_p}{\sqrt{k}}\right|\right)_{k\geq 0}\]
is tight under $\mathbf{P}$. Moreover, for any $x>0$,
\begin{equation}\label{tightsharp}
\liminf_{k\rightarrow \infty}\mathbf{P}\left(\left|\frac{X_k^F-kv_p}{\sqrt{k}}\right|>x\right)>0.
\end{equation}
(ii) It holds that
\[\frac{X_k^L-kv_p}{\sqrt{\sigma_L^2 k}}\rightarrow N(0,1),\]
in distribution under $\mathbf{P}$, where $N(0,1)$ is a standard normal random variable, and
\[\sigma_L^2:=\frac{4p(1-p)}{(1-2p)^3}.\]
(c) The sequence $(k^{-1}X^L_k)_{k\geq 0}$ satisfies a large deviations principle with rate function given by
\begin{equation}\label{ildef}
I_L(x):=\sup_{\theta\in\mathbb{R}}\left(\theta x-\log {M}_{L}(\theta)\right),
\end{equation}
where
\begin{equation}\label{mldef}
M_L(\theta):=\left\{
\begin{array}{ll}\frac{1-\sqrt{1-4p(1-p)e^{2\theta}}}{2pe^\theta}, & \mbox{if }\theta\leq -\frac12 \log(4p(1-p)),\\
  \infty, & \mbox{otherwise.}
\end{array}\right.
\end{equation}}

The proof strategy will be quite different for the results concerning $X^F$ and $X^L$. For the former process, $X^F$, we will appeal to the results of Theorem \ref{currentcltthm} concerning the current across the origin. Whilst part (b)(i) of the above result might suggest a central limit theorem holds for $X^F$, such a conclusion does not follow directly from the central limit theorem for the current due to the correlation between this and the particle configuration. The latter process, $X^L$, turns out to be an easier process to analyse. Indeed, under a suitable Palm measure, obtained by conditioning $\eta$ to have a particle at $0$, we show that the increments of $X^L$ are i.i.d., and we can even give the explicit distribution of these increments (see Lemma \ref{hatiid}).

\begin{proof}[Proof of the parts of Theorem \ref{distancethm} concerning $X^F$] As noted in Section \ref{currentclt}, $C_k$ represents the number of particles moved from $\{\dots,-1,0,\}$ to $\{1,2,\dots\}$ on the first $k$ evolutions of the BBS. Hence, since the FIFO scheme preserves particle ordering, $X^F_k$ is the $(C_k+1)$st particle to the right of the origin in $T^k\eta$, that is
\begin{equation}\label{xf}
X^F_k=\min\left\{m:\:\sum_{i=1}^m(T^k\eta)_i=C_k+1\right\}.
\end{equation}
Now, from Theorem \ref{currentcltthm}(a), we know that $k^{-1}C_k\rightarrow \mu_p$. Hence, given any $\varepsilon>0$, $\mathbf{P}$-a.s. for large $k$,
\begin{equation}\label{xfbounds}
\min\left\{m:\:\sum_{i=1}^m(T^k\eta)_i\geq k(\mu_p-\varepsilon)\right\}\leq X^F_k\leq \min\left\{m:\:\sum_{i=1}^m(T^k\eta)_i\geq k(\mu_p+\varepsilon)\right\}.
\end{equation}
Moreover, for any $c>0$, we have from the Azuma-Hoeffding inequality that
\begin{equation}\label{azuma}
\mathbf{P}\left(\left|\sum_{i=1}^{ck}(T^k\eta)_i-ckp\right|>\varepsilon k\right)\leq 2e^{-\varepsilon^2k/2c}.
\end{equation}
In particular, by Borel-Cantelli and countability, we obtain that $k^{-1}\sum_{i=1}^{qk}(T^k\eta)_i\rightarrow qp$ for any rational $q>0$, $\mathbf{P}$-a.s. Combining this and \eqref{xfbounds}, it is elementary to obtain the $\mathbf{P}$-a.s.\ limit $k^{-1}X_k^F\rightarrow \mu_p/p=v_p$, establishing the relevant limit in part (a).

For the tightness claim of part (b)(i), we again appeal to Theorem \ref{currentcltthm}. Indeed, for $x,\lambda>0$, $k\geq 1$, $K=kv_p+x\sqrt{k}$, we have from \eqref{xf} that
\begin{eqnarray*}
\mathbf{P}\left(\frac{X_k^F-kv_p}{\sqrt{k}}> x\right)
&=&\mathbf{P}\left(\sum_{i=1}^K(T^k\eta)_i<C_k+1\right)\\
&\leq& \mathbf{P}\left(\sum_{i=1}^K(T^k\eta)_i<Kp -\lambda \sqrt{K}\right)+\mathbf{P}\left( C_k+1>Kp-\lambda \sqrt{K}\right).
\end{eqnarray*}
Since $Kp-\lambda \sqrt{K}=k\mu_p +(xp-\lambda\sqrt{v_p})\sqrt{k}+o(\sqrt{k})$, we have that from Theorem \ref{currentcltthm}(b) that
\[\lim_{x\rightarrow\infty}\limsup_{k\rightarrow \infty}\mathbf{P}\left( C_k+1>Kp-\lambda \sqrt{K}\right)=0.\]
Hence, by applying \eqref{azuma},
\[\lim_{x\rightarrow\infty}\limsup_{k\rightarrow \infty}\mathbf{P}\left(\frac{X_k^F-kv_p}{\sqrt{k}}> x\right)\leq 2e^{-\lambda^2/2}.\]
Since $\lambda$ can be chosen arbitrarily large, this establishes the tightness of $X_k^F-kv_p/\sqrt{k}$. To establish the corresponding result for $(kv_p-X_k^F)/\sqrt{k}$ essentially the same argument can be applied, and so we omit the proof.

Finally, we will show \eqref{tightsharp}. Proceeding similarly to above, $x,\lambda>0$, $k\geq 1$, $K=kv_p+x\sqrt{k}$, we have
\begin{eqnarray*}
\mathbf{P}\left(\left|\frac{X_k^F-kv_p}{\sqrt{k}}\right|> x\right)
&\geq &\mathbf{P}\left(\sum_{i=1}^K(T^k\eta)_i<pK+\lambda\sqrt{K}<C_k+1\right)\\
&\geq &\mathbf{P}\left(C_k+1>pK+\lambda\sqrt{K}\right)-\mathbf{P}\left(\sum_{i=1}^K(T^k\eta)_i>pK+\lambda\sqrt{K}\right).
\end{eqnarray*}
Applying Theorem \ref{currentcltthm} and the fact that $\sum_{i=1}^K(T^k\eta)_i$ is simply a binomial random variable with parameters $K$ and $p$, it follows that
\begin{equation}\label{lamlower}
\liminf_{k\rightarrow\infty}\mathbf{P}\left(\left|\frac{X_k^F-kv_p}{\sqrt{k}}\right|> x\right)\geq \mathbf{P}\left(N(0,1) > \frac{px+\lambda\sqrt{v_p}}{\sqrt{\sigma_p^2}}\right)-
\mathbf{P}\left(N(0,1) > \frac{\lambda}{\sqrt{p(1-p)}}\right).
\end{equation}
Noting that
\[\frac{v_p}{\sigma_p^2}=\frac{(1-2p)}{p(1-p)}<\frac{1}{p(1-p)},\]
we see that the lower bound of (\ref{lamlower}) is strictly positive for large $\lambda$.
\end{proof}

To complete the proof of Theorem \ref{distancethm} we apply the following lemma, for which we define \[\hat{\mathbf{P}}:=\mathbf{P}(\cdot\:|\:\eta_0=1).\]

{\lem\label{hatiid} Suppose $\eta$ is given by a sequence of i.i.d.\ Bernoulli($p$) random variables with $p\in(0,\frac12)$. Under $\hat{\mathbf{P}}$,  $(X_n^L-X_{n-1}^L)_{n\geq 1}$ form an i.i.d.\ sequence with
\begin{equation}\label{distribution}
\hat{\mathbf{P}}\left(X_n^L-X_{n-1}^L=m\right)=\frac{1}{m}\binom{m}{\frac{m+1}{2}}(1-p)^{\frac{m+1}{2}}p^{\frac{m-1}{2}},\qquad m\geq 1,\:m\mbox{ odd}.
\end{equation}
In particular,
\[\hat{\mathbf{E}}\left(X_n^L-X_{n-1}^L\right)=v_p,\qquad \sigma_L^2:=\mathrm{Var}_{\hat{\mathbf{P}}}\left(X_n^L-X_{n-1}^L\right)=\frac{4p(1-p)}{(1-2p)^3}.\]}
\begin{proof} Defining the shift operator $\theta_m$ on particle configurations by setting $(\theta_m\eta)_n:=\eta_{m+n}$, observe that the $X_k^L-X_0^L$ is the position of the particle started from the origin in $\theta_{X_0^L}\eta$ after $k$ evolutions of the BBS. Since we clearly have $\theta_{X_0^L}\eta\buildrel{d}\over{=}\eta$ under $\hat{\mathbf{P}}$ (indeed, under the relevant measure we have i.i.d.\ geometric inter-particle distances), it will be sufficient to prove the result when the process $X^L$ is replaced by $\hat{X}^L=(\hat{X}^L_k)_{k\geq 0}$, which tracks the position of the particle started from the origin.

Importantly, we observe that under the LIFO scheme, the particle started at the origin shifts after the first evolution of the BBS to $\hat{\tau}_S(1):=\min\{m\geq 0:\:S_m=1\}$. Indeed, between $0$ and $\hat{\tau}_S(1)-1$, there is an equal number of particles and empty spaces, and it is easy to show that the carrier shifts the particles to the empty spaces. Hence, we have that, for any $m\geq 1$ odd and measurable subset $A$,
\begin{eqnarray*}
\lefteqn{\hat{\mathbf{P}}\left(\hat{X}_1^L=m,\:\left((TS)_{\hat{X}^L_1+n}-(TS)_{\hat{X}^L_1}\right)_{n\geq 0}\in A\right)}\\
&=&p^{-1}{\mathbf{P}}\left(\eta_0=1,\:\hat{\tau}_S(1)=m,\:\left((TS)_{\hat{X}^L_1+n}-(TS)_{\hat{X}^L_1}\right)_{n\geq 0}\in A\right)\\
&=&p^{-1}{\mathbf{P}}\left((T\eta)_0=0,\:\hat{\tau}_{TS}(-1)=m,\:\left((TS)_{\hat{\tau}_{TS}(-1)+n}-(TS)_{\hat{\tau}_{TS}(-1)}\right)_{n\geq 0}\in A\right),
\end{eqnarray*}
where we note that $\eta_0=1$ is $\mathbf{P}$-a.s.\ equivalent to $(T\eta)_0=0$ and $\hat{\tau}_{TS}(-1)<\infty$ both holding, and moreover that on the intersection of the latter events $\hat{\tau}_S(1)=\hat{\tau}_{TS}(-1)$, $\mathbf{P}$-a.s. By the strong Markov property, the three events in the above probability are independent, and so
\begin{eqnarray*}
\lefteqn{\hat{\mathbf{P}}\left(\hat{X}_1^L=m,\:\left((TS)_{\hat{X}^L_1+n}-(TS)_{\hat{X}^L_1}\right)_{n\geq 0}\in A\right)}\\
&=&p^{-1}{\mathbf{P}}\left((T\eta)_0=0\right)\mathbf{P}\left(\hat{\tau}_{TS}(-1)=m\right)\\
&&\times\mathbf{P}\left(\left((TS)_{\hat{\tau}_{TS}(-1)+n}-(TS)_{\hat{\tau}_{TS}(-1)}\right)_{n\geq 0}\in A\:\vline\:\hat{\tau}_{TS}(-1)<\infty\right).
\end{eqnarray*}
Now, since $TS$ is a simple random walk distributed the same as $S$ (under $\mathbf{P}$), we see that
\begin{equation}\label{finitehit}
p^{-1}{\mathbf{P}}\left((T\eta)_0=0\right)=\frac{1-p}{p}=\frac{1}{\mathbf{P}\left(\hat{\tau}_{TS}(-1)<\infty\right)}.
\end{equation}
Hence, using also that $(S_n)_{n\geq 0}$ is independent of $\eta_0$, this implies
\begin{eqnarray*}
\lefteqn{\hat{\mathbf{P}}\left(\hat{X}_1^L=m,\:\left((TS)_{\hat{X}^L_1+n}-(TS)_{\hat{X}^L_1}\right)_{n\geq 0}\in A\right)}\\
&=&\mathbf{P}\left(\hat{\tau}_{TS}(-1)=m\:\vline\:\hat{\tau}_{TS}(-1)<\infty,\:(T\eta)_0=0\right)\\
&&\qquad\times\:\mathbf{P}\left(\left(S_{\hat{\tau}_S(-1)+n}-S_{\hat{\tau}_{S}(-1)}\right)_{n\geq 0}\in A\:\vline\:\hat{\tau}_{S}(-1)<\infty\right)\\
&=&\hat{\mathbf{P}}\left(\hat{X}_1^L=m\right)\hat{\mathbf{P}}\left(\left(S_n\right)_{n\geq 0}\in A\right).
\end{eqnarray*}
In words, this means that under $\hat{\mathbf{P}}$, the random variable $\hat{X}_1^L$ is independent of $((TS)_{\hat{X}^L_1+n}-(TS)_{\hat{X}^L_1})_{n\geq 0}$, and the latter is distributed as $(S_n)_{n\geq 0}$. As a consequence, we can iterate the argument to obtain that the sequence $(\hat{X}_n^L-\hat{X}_{n-1}^L)_{n\geq 0}$ is i.i.d.\ under $\hat{\mathbf{P}}$, as desired.

Since $\hat{X}_1^L=\hat{\tau}_S(1)$, the precise distribution at \eqref{distribution} is a consequence of the hitting time theorem for random walks (as can be found in \cite{Otter, vdh}, for instance). Moreover, the expressions for the mean and variance can be deduced by elementary calculations (either from the formula directly, or a first-step decomposition of the random walk).
\end{proof}

\begin{proof}[Proof of the parts of Theorem \ref{distancethm} concerning $X^L$] Observe that, for any bounded, measurable function $f$ on particle configurations, we have that
\begin{equation}\label{unidecomp}
{\mathbf{E}}\left(f(\eta)\right)=\frac{\hat{\mathbf{E}}\left(\sum_{m=0}^{X_0^L-1}f(\theta_m\eta)\right)}{\hat{\mathbf{E}}\left({X_0^L}\right)},
\end{equation}
where we again denote by $\theta_m$ the shift of the configuration $\eta$ given by $(\theta_m\eta)_n=\eta_{n+m}$. That is, it is possible to construct the law of $\eta$ under $\mathbf{P}$ by first selecting $\eta$ from $\hat{\mathbf{P}}$, size-biased by $X_0^L$, and then shifting according to $\theta_U$, where $U$ is uniform on $\{0,1,\dots,X_0^L-1\}$. (For a proof, see \cite[Theorem 1]{KZ}.)

Now, by Lemma \ref{hatiid}, we know that $X^L$ satisfies the targeted law of large numbers and central limit theorem under $\hat{\mathbf{P}}$. Moreover, for $m\in\{0,1,\dots,X_0^L-1\}$,
\[X_n^L(\theta_m\eta)=X_n^L(\eta)-m.\]
Hence \eqref{unidecomp} allows us to conclude the relevant results also hold under $\mathbf{P}$. Indeed, for the law of large numbers we have
\[\mathbf{P}\left(k^{-1}X_k^L\rightarrow v_p\right)
=\frac{\hat{\mathbf{E}}\left(\sum_{m=0}^{X_0^L-1}\mathbf{1}_{\{k^{-1}(X_k^L-m)\rightarrow v_p\}}\right)}{\hat{\mathbf{E}}\left({X_0^L}\right)}=\frac{\hat{\mathbf{E}}\left(\sum_{m=0}^{X_0^L-1}\mathbf{1}_{\{k^{-1}X_k^L\rightarrow v_p\}}\right)}{\hat{\mathbf{E}}\left({X_0^L}\right)}=1,\]
since $\mathbf{1}_{\{k^{-1}X_k^L\rightarrow v_p\}}=1$, $\hat{\mathbf{P}}$-a.s. Similarly, for the central limit theorem, we have
\begin{eqnarray*}
\mathbf{P}\left(\frac{X_k^L-kv_p}{\sqrt{\sigma_L^2k}}\in(a,b)\right)
=\frac{\sum_{l=0}^\infty\hat{\mathbf{E}}\left(\mathbf{1}_{\{X_0^L=l\}}\sum_{m=0}^{l-1}\mathbf{1}_{\left\{({X_k^L-X_0^L+l-m-kv_p})/{\sqrt{\sigma_L^2k}}\in(a,b)\right\}}\right)}{\hat{\mathbf{E}}\left({X_0^L}\right)}.\\
\end{eqnarray*}
Now, from the proof of Lemma \ref{hatiid}, we know that $X_0^L$ and $X^L_k-X^L_0$ are independent under $\hat{\mathbf{P}}$. This means that
\begin{eqnarray*}
\mathbf{P}\left(\frac{X_k^L-kv_p}{\sqrt{\sigma_L^2k}}\in(a,b)\right)
=\frac{\sum_{l=0}^\infty\sum_{m=0}^{l-1}\hat{\mathbf{P}}\left(X_0^L=l\right)\hat{\mathbf{P}}\left(\frac{X_k^L-X_0^L+l-m-kv_p}{\sqrt{\sigma_L^2k}}\in(a,b)\right)}{\hat{\mathbf{E}}\left({X_0^L}\right)}.\\
\end{eqnarray*}
Since for each fixed $l,m$, we have that
\[\hat{\mathbf{P}}\left(\frac{X_k^L-X_0^L+l-m-kv_p}{\sqrt{\sigma_L^2k}}\in(a,b)\right)\rightarrow {\mathbf{P}}\left(N(0,1)\in(a,b)\right),\]
the dominated convergence theorem thus yields that
\[\mathbf{P}\left(\frac{X_k^L-kv_p}{\sqrt{\sigma_L^2k}}\in(a,b)\right)\rightarrow{\mathbf{P}}\left(N(0,1)\in(a,b)\right).\]

Finally, for the large deviations principle of part (c) we first note that Lemma \ref{hatiid} and Cramer's theorem \cite[Theorem 2.2.3]{DZ} immediately give a large deviations principle for the sequence $(k^{-1}(X^L_k-X^L_0))_{k\geq 0}$ under $\hat{\mathbf{P}}$. Moreover, the rate function is given by the Legendre transform of $\hat{\mathbf{E}}\left(e^{\theta (X^L_n-X^L_{n-1})}\right)$. A first step decomposition for $\hat{\tau}_S(1)$ (as defined in the proof of Lemma \ref{hatiid}) readily allows us to deduce that this moment generating function is equal to $M_L(\theta)$, as defined at \eqref{mldef}, and so the relevant rate function is given by $I_L(x)$, as defined at (\ref{ildef}). To transfer this to a large deviations principle for $(k^{-1}X^L_k)_{k\geq 0}$ under $\mathbf{P}$, we can proceed similarly to the central limit theorem. In particular, we deduce from (\ref{unidecomp}) and apply the inequalities
\[-\log\mathbf{P}(X_k^L \in A)\leq
\frac{\sum_{l=0}^\infty\sum_{m=0}^{l-1}\hat{\mathbf{P}}\left(X_0^L=l\right)\left(-\log \hat{\mathbf{P}}\left({X_k^L-X_0^L+l-m}\in A\right)\right)}{\hat{\mathbf{E}}\left({X_0^L}\right)},\]
which is consequence of the convexity of $x\mapsto -\log (x)$, and
\[-\log\mathbf{P}(X_k^L \in A)\geq
-\log\left(\frac{ \hat{\mathbf{P}}\left(X_0^L=0\right)\hat{\mathbf{P}}\left({X_k^L-X_0^L}\in A\right)}{\hat{\mathbf{E}}\left({X_0^L}\right)}\right).\]
\end{proof}

\begin{rem}\label{fiforem} Note that, in establishing the law of large numbers for $X^F$ for the i.i.d.\ configuration, the key inputs were a law of large numbers for the integrated current across the origin and a concentration inequality for the density (i.e.\ \eqref{azuma}). For the Markov initial configuration and bounded soliton examples of Theorem \ref{mre}(b),(c), we have the law of large numbers for the current (with limit $\mathbf{E}W_0$) from Theorem \ref{mrf}. Moreover, for these two examples, the Azuma-Hoeffding inequality of \eqref{azuma} can be replaced with the Markov chain version of \cite[Theorem 2]{GO}, for example. Hence, we conclude that they both satisfy $X^F_k/k\rightarrow \mathbf{E}(W_0)/\mathbf{E}(\eta_0)$, $\mathbf{P}$-a.s.
\end{rem}

\begin{rem}\label{liforem} For $X^L$, via essentially the same argument as for the i.i.d.\ case, we can also establish a law of large numbers, central limit theorem and large deviations principle in the case when the initial configuration is the Markov initial configuration from Theorem \ref{mre}(b). Indeed, the only additional input needed to the above argument is at \eqref{finitehit}, which should be replaced by
\[\rho^{-1}{\mathbf{P}}\left((T\eta)_0=0\right)=\frac{1-\rho}{\rho}=\frac{1-p_1}{p_0}=
\frac{1}{\mathbf{P}\left(\hat{\tau}_{TS}(-1)<\infty\:\vline\:(T\eta)_0=0\right)},\]
where the final equality is established in the proof of Lemma \ref{Wpropsmarkov} (note that, in the notation of the latter proof, the expression on the right-hand side is $1/q_0$). The remaining changes are straightforward. Moreover, as in the i.i.d.\ case, the limiting speed is given by $v_{p_0,p_1}=\hat{\mathbf{E}}\hat{\tau}_{S}(1)$. Writing $t_j:={\mathbf{E}}(\hat{\tau}_{S}(1)|\eta_0=j)$ for $j=0,1$, a first-step decomposition yields
\[t_j=p_j(1+t_0+t_1)+1-p_j.\]
for $j=0,1$. These equations can be solved to give $v_{p_0,p_1}=t_1=(1-p_0+p_1)/(1-p_0-p_1)$, and thus the law of large numbers is of the form, $\mathbf{P}$-a.s.,
\[\frac{X^L_k}{k}\rightarrow v_{p_0,p_1}=\frac{1-p_0+p_1}{1-p_0-p_1}.\]
Note that, by \eqref{markovdensity} and \eqref{wmarkovexpect}, the limiting speed is of the form $\mathbf{E}W_0 / \mathbf{E} \eta_0$, matching the formula arrived at heuristically at \eqref{anticipateddistance}, and the limiting speed under the FIFO scheme, as discussed in the previous remark. Moreover, we can rewrite the above expression as follows:
\[v_{p_0,p_1}=\frac{1}{1-2\rho}\left(\frac{2\rho}{p_0}-1\right),\]
showing that the speed is equal to that of the tagged particle in an i.i.d.\ configuration with the same density if and only if $p_0=p_1=\rho$ (i.e.\ the configuration is i.i.d.). Note that, for a fixed density $\rho$, the monotonicity of the above formula in $p_0$ can be interpreted in the following way: as $p_0$ decreases (or equivalently $p_1$ increases), the configuration $\eta$ will typically contain longer strings of consecutive particles, which create larger solitons, and this leads in turn to an increased rate of escape. The variance in the central limit theorem can also be computed explicitly using a first-step decomposition.
\end{rem}

\begin{rem}\label{boundedrem} Let $W$ be a two-sided stationary Markov process that is irreducible on the space $\{0,1,\dots,K\}$ and satisfies (\ref{strongsym}). As discussed in Remark \ref{K-Wrem}, the associated path encoding is supported on $\mathcal{S}_K$ and is invariant under $T$. Moreover, it is clear that the current at the origin is given by the alternating sequence $(W_0,K-W_0,W_0,K-W_0\dots)$, and so
\[\frac{C_k}{k}\rightarrow \frac{K}{2},\]
as $k\rightarrow\infty$, $\mathbf{P}$-a.s. From this, the fact that the density of particles is $\frac12$, and (\ref{xf}), we deduce that
\[\frac{X_k^F}{k}\rightarrow K,\]
as $k\rightarrow\infty$, $\mathbf{P}$-a.s. Hence, under the FIFO scheme, the tagged particle moves at the speed of an isolated soliton of size $K$. This is not matched by the behaviour of the tagged particle under the LIFO scheme, however. Indeed, it is an elementary exercise to check from the definition of the process that, in this case,
\[X_k^L=\min\left\{n:\:\#\left\{\mbox{crossings of $\{W_{X_0^L}-1,W_{X_0^L}\}$ by $W$ in the interval $[X_0^L,n]$}\right\}=k\right\},\:\: \forall k\geq 1,\]
where crossings of the relevant interval can be up or down. Since the ergodicity of $W$ implies
\begin{eqnarray*}
\lefteqn{n^{-1}\#\left\{\mbox{crossings of $\{W_{X_0^L}-1,W_{X_0^L}\}$ by $W$ in the interval $[0,n]$}\right\}}\\
&\rightarrow &\pi_{W_{X_0^L}-1}p_{W_{X_0^L}-1} + \pi_{W_{X_0^L}}\left(1-p_{W_{X_0^L}}\right)=2\pi_{W_{X_0^L}}\left(1-p_{W_{X_0^L}}\right),
\end{eqnarray*}
where we write $\pi$ for the stationary probability measure of $W$ and suppose the transition matrix of $W$ is given by \eqref{vmatrix}, it follows that
\[\frac{X_k^L}{k}\rightarrow \left(2\pi_{W_{X_0^L}}\left(1-p_{W_{X_0^L}}\right)\right)^{-1},\]
as $k\rightarrow\infty$, $\mathbf{P}$-a.s. In particular, the limit is not constant in general.
\end{rem}

\section{Connections with Pitman's theorem and exclusion processes}\label{pitmantasepsec}

\subsection{One-sided random initial configurations and Pitman's theorem}\label{pitmansec} Those familiar with stochastic processes will immediately recognise the path transformations $S\mapsto M-S$ and $S\mapsto 2M-S$ used in this article from well-known works of L\'{e}vy and Pitman. In this section, we draw some explicit connections between the results of the previous section with some classical results in the area. Since the aim is to highlight what we consider interesting observations, rather than develop new theory, we restrict technical details to a minimum. Moreover, since the literature mainly focuses on the one-sided case, we also concentrate here one-sided particle configurations $\eta=(\eta_n)_{n\geq1}$.

To begin with, there is a strong parallel between Proposition \ref{Wprops} and a famous result of L\'{e}vy from \cite{Levy}. In particular, in the latter work, it was shown that if $B=(B_t)_{t\geq 0}$ is Brownian motion and $M^B=(M^B_t)_{t\geq 0}$ its running maximum (i.e.\ $M^B_t:=\sup_{s\leq t}B_s$), then $M^B-B$ is equal in distribution to reflected Brownian motion, or equivalently the process $|B|$. In the case when Brownian motion has a linear drift, similar results are also known (see \cite{Peskir}), with the limit being reflected Brownian motion with the opposite drift. Moreover, explicit formulae are known for the one-dimensional marginals of the latter process when started from 0, as well as its invariant distribution in the case when the original Brownian motion has strictly positive drift. In the case when $\eta=(\eta_n)_{n\geq1}$ is an i.i.d.\ Bernoulli($p$) sequence, the proof of Proposition \ref{Wprops} yields the distribution of $W=M-S$ as the reflected random walk with drift, or, more specifically, the Markov process started from $W_0=0$, with transition probabilities given by \eqref{Wprobs}. This is clearly the discrete analogue of L\'{e}vy's result, and we note it applies to any $p\in(0,1)$, not just the case when $S$ has strictly positive drift. Of course, whilst this process is defined for any $p\in(0,1)$, it does not admit a stationary probability distribution for $p\geq 1/2$, and so can not be extended to a two-sided stationary process for this range of $p$. Clearly Proposition \ref{Wprops} can be seen as the corresponding result for a two-sided random walk with strictly positive drift.

Another illustrious result in the area is the representation theorem of Pitman \cite{Pitman}, which shows that the process $2M^B-B$ has a BES(3) distribution. (See \cite{RP} for further related results concerning the relevant transformation.) To prove this, Pitman first derived a discrete version of the result, and then took scaling limits. His approach gives the distribution of $TS=2M-S$ in the one-sided, zero drift ($p=1/2$), i.i.d.\ configuration case. Specifically, this is the Markov chain with transition probabilities given by:
\[\mathbf{P}\left((TS)_{n}=x+1\:\vline\:(TS)_{n-1}=x\right)=\frac{x+2}{2(x+1)}=1-
\mathbf{P}\left((TS)_{n}=x-1\:\vline\:(TS)_{n-1}=x\right),\]
for all $x\in\mathbb{Z}_+$. In the one-sided i.i.d.\ configuration case with strictly positive drift (i.e.\ for any value of $p\in(0,1/2))$, the distribution of $TS$ is shown in \cite{HMOC} to be equal to the law of $S$ conditioned to be non-negative, which can also be expressed explicitly in terms of a Doob transform. In particular, for any value of $p\in(0,\frac12)$, we have from standard arguments, e.g.\ \cite[Section 17.6.1]{LPW}, that
\begin{equation}\label{doob}
\mathbf{P}\left((TS)_{n}=x+1\:\vline\:(TS)_{n-1}=x\right)=\left(1-p\right)\frac{1-\left(\frac{p}{1-p}\right)^{x+2}}{1-\left(\frac{p}{1-p}\right)^{x+1}}=1-
\mathbf{P}\left((TS)_{n}=x-1\:\vline\:(TS)_{n-1}=x\right),
\end{equation}
for all $x\in\mathbb{Z}_+$. We observe this conditioning has little effect away from the origin, with the transition probability of an up-jump being asymptotically equal to $1-p$ as $x\rightarrow\infty$. In fact, the work of \cite{HMOC} also applies to the Markov initial configuration case, but then $S$ is not a Markov process and the relevant Doob transform has to be defined for the two-dimensional Markov chain $((S_n,\eta_n))_{n\geq0}$. Related to these one-sided results, we remark that the invariance in distribution of $S$ under the transformation $S\mapsto {T}S$ was essentially established in the two-sided i.i.d.\ and Markov configuration cases in \cite[Corollary 3]{HMOC}. However, the invariance under $T$ of the two-sided, conditioned process $\tilde{S}^{(k)}$ from Section \ref{boundedsec} is apparently a new result.

Finally, whilst the results described in the previous paragraph give a complete characterisation of the state of the BBS after one time step when we have a one-sided i.i.d.\ Bernoulli starting configuration with parameter $p\leq 1/2$ (and in the two-sided case for $p\in(0,\frac12)$), it is also natural to ask what happens in the one-sided case when $p>1/2$, since the carrier is then still well-defined. In this setting, by undertaking a relatively straightforward path decomposition (conditioning on the position of the maximum of $S$), it is possible to check that, just as in the $p\leq 1/2$ case, $TS$ is distributed as $S$ conditioned to never hit $-1$. Of course, the latter conditioning is not well-defined, but we can make sense of it as a Doob transform. More specifically, $TS$ is the Markov process started from $0$ with transition probabilities given by (\ref{doob}). Whilst one might expect to have to make the exchange $p\leftrightarrow 1-p$, note that the latter formula is in fact invariant under this transposition. In particular, this observation yields that $TS\buildrel{d}\over{=} T(-S)$. That is, the action of the carrier reverses the action of the drift. Or, to state it another way, this conclusion tells us that in the high density ($p>1/2$) regime: the first step of the BBS transports most particles out of the system (to $\infty$, say); from then on, the system evolves exactly like the low density system with particle density $1-p$.

\subsection{BBS versus totally asymmetric simple exclusion processes}

To put our results for the BBS into further context, we briefly compare and contrast them with those known to hold for the totally asymmetric simple exclusion process (TASEP), which is one of the most widely studied interacting particle systems. More specifically, the exclusion process on $\mathbb{Z}$ is a continuous time Markov process on $\{0,1\}^{\mathbb{Z}}$ describing an evolution of interacting continuous time random walks on $\mathbb{Z}$ with an exclusion rule that prohibits there from being more than one particle per site. When each particle can move only to its right-hand neighbouring site, and the mean of the waiting time is constant and equal to $1$, the model is referred to as the TASEP. At the most basic level, we thus immediately see a connection with the BBS, for which the state space is also $\{0,1\}^{\mathbb{Z}}$, meaning we have an exclusion rule, and each particle can only move in the rightwards direction. Moreover, the order of particles is preserved by the TASEP, which is also the case for the BBS if we suppose the dynamics are given by the FIFO scheme. Given such similarities, it seems interesting to compare more detailed characteristics of the two systems. We will not give a comprehensive survey of the results for the TASEP, but describe some of the well-known results on invariant measures, as well as the asymptotic behaviour of currents and the tagged particle.

For the TASEP, the set of extremal invariant measures is completely characterized as the union of i.i.d.\ Bernoulli product measures with any density $p \in [0,1]$, and the blocking measures indexed by $N \in \mathbb{Z}$, which are the delta measures on the configurations $\eta=(\eta_n)_{n\in\mathbb{Z}}=(\mathbf{1}_{\{n \ge N\}})_{n\in\mathbb{Z}}$, $N\in\mathbb{Z}$ \cite[VIII.3.23]{Liggett}. Theorem \ref{mre} shows that the BBS admits a richer class of invariant measures. This can be understood to be a consequence of the deterministic dynamics of the BBS preserving solitons, whereas any large scale structures are destroyed by the random dynamics of the TASEP. In some sense the blocked dynamics for the TASEP parallel the dynamics of the particle configurations with path encodings in $\mathcal{S}_{critical}^*$, in that, in both cases, these are the most trivial dynamics possible for each system. On the other hand, all the invariant measures for the BBS satisfy that $\mathbf{P}(\eta_n=1)$ is constant (see Theorem \ref{mra}), which holds for a general class of symmetric exclusion processes \cite[VIII.1.44]{Liggett}, but is not the case for the TASEP, for which we clearly have $\mathbf{P}(\eta_n=1)=0$ for $n < N$ and $\mathbf{P}(\eta_n=1)=1$ for $n \ge N$ under the relevant blocking measure. This difference comes from the fact BBS is ``reversible'' in the sense of dynamical systems. Namely, for the BBS, if $\eta$ has path encoding supported in $\mathcal{S}^{rev}$ and $T\eta \buildrel{d}\over{=} \eta$, then $T^{-1} \eta \buildrel{d}\over{=} \eta$, and so the blocking measures can not be invariant.

Under the i.i.d.\ Bernoulli product measure with parameter $p$, the integrated current at the origin in the TASEP satisfies the following law of large numbers and central limit theorem:
\[\frac{J_{Nt}}{N} \to \mu_p^{TASEP}t,\qquad \mathbf{P}\mbox{-a.s.},\]
\[\frac{J_{Nt}-\mu_p^{TASEP}Nt}{\sqrt{N}} \buildrel{d}\over{\to} N\left(0, (\sigma_p^{TASEP})^2t\right),\]
where $J_{t}$ is the integrated current at the bond $\{0,1\}$ for the time interval $[0,t]$, and
\[\mu_p^{TASEP}=p(1-p),\qquad (\sigma_p^{TASEP})^2=p(1-p)|1-2p|,\]
see \cite{FerFont}, and also \cite[Section 4.1]{Gon} for a survey of results in this direction. These constants satisfy the relation that
\begin{equation}\label{isitremarkable}
\left(\chi(p) \frac{d}{dp}\mu_p^{TASEP}\right)^2=\left(\sigma_p^{TASEP}\right)^2,
\end{equation}
where
$\chi(p):=\mathrm{Var}(\eta_0)=p(1-p)$. This relation is known to hold for more general interacting particle systems satisfying the Boltzmann-Gibbs principle, such as totally asymmetric zero range process \cite[Theorem 4.2.1]{Gon}. From Theorem \ref{currentcltthm}, we also have a law of large numbers and central limit theorem for the current in the BBS started from a Bernoulli product measure. We observe that the relation at (\ref{isitremarkable}) holds in this case as well, since $\mu_p=\frac{p}{1-2p}$, $\sigma_p^2=\frac{p(1-p)}{(1-2p)^2}$ and $\chi(p)=\mathrm{Var}(\eta_0)=p(1-p)$. The behaviours of the two systems as $p\rightarrow\frac12$ are very different, however. In particular, in the $p=\frac12$ case, $\sigma_p^{TASEP}=0$ and the proper time scaling is not $Nt$, but $N^{3/2}t$, and the fluctuation is not Gaussian \cite{FerSpo}. For the BBS, $\mu_p\rightarrow\infty$ as $p\rightarrow\frac12$, and, as we will show in the next section, diffusive scaling of the entire system is needed to understand the dynamics.

In the TASEP, a law of large numbers and central limit theorem is also know to hold for the tagged particle under the i.i.d.\ Bernoulli product measure, and can be stated as
\[\frac{X_{Nt}}{N} \to v_p^{TASEP}t,\qquad \mathbf{P}\mbox{-a.s.},\]
\[\frac{X_{Nt}-v_p^{TASEP}t}{\sqrt{N}} \buildrel{d}\over{\rightarrow} N\left(0, (\sigma_p^{TASEP,tag})^2t\right),\]
where
\[v_p^{TASEP}=\frac{\mu_p^{TASEP}}{p}=1-p,\qquad (\sigma_p^{TASEP,tag})^2=1-p,\]
see \cite[Section 4.1]{Gon}, for example. For the FIFO scheme of the BBS, we similarly have the law of large numbers, with mean satisfying $v_p=\frac{\mu_p}{p}=\frac{1}{1-2p}$  (see Theorem \ref{distancethm}), but we were unable to establish the corresponding central limit theorem.

\section{BBS on $\mathbb{R}$}\label{contsec}

In this section, we consider a generalisation of the BBS, which is defined for continuous functions on $\mathbb{R}$. In particular, in Section \ref{pathsec} the dynamics of the BBS was expressed as the operator $T$ on piecewise linear functions with derivative $\pm 1$ (recall \eqref{wtsdef}). From this explicit expression for the operator $T$, it is natural to generalize the domain of the operator to continuous functions on $\mathbb{R}$. One motivation for doing this is that it provides a natural framework for studying the scaling limit of the discrete system. As an illustrative example, we check that if $S$ is the asymmetric simple random walk representing an i.i.d.\ particle configuration with density $p_n=\frac12-\frac{c}{2n}$, then under appropriate scaling as $n\rightarrow\infty$, we arrive at a Brownian motion with drift $c$; this can be considered the high density regime for the BBS. Moreover, it readily follows from the invariance of the simple random walks under $T$ that the limiting process is also invariant under $T$. Specifically, we prove Theorem \ref{mrh}.

\subsection{Operator $T$ for continuous functions} Unlike the discrete case, we can not describe the particle configuration $\eta$ directly, and so we consider the dynamics for the path encoding $S$ only. By analogy with the relevant discrete objects, let
\[\mathcal{S}_{c}^0=\left\{S:\mathbb{R}\rightarrow \mathbb{R}\::\: S_0=0,\:S\mbox{ continuous}\right\},\]
define the domain of $T$ by setting
\[\mathcal{S}_{c}^T=\left\{S \in \mathcal{S}_c^0\::\: \limsup_{x \to -\infty} S_x < \infty \right\},\]
and, for $S \in \mathcal{S}^T_c$, define
\[M_x=\sup_{y \le x} S_y, \qquad W_x=M_x-S_x, \qquad (TS)_x=2M_x-S_x-2M_0.\]
Note that the operator $T$ is the two-sided version of Pitman's transform as already discussed. The corresponding inverse operator $T^{-1}$ has domain
\[\mathcal{S}_{c}^{T^{-1}}=\left\{S \in \mathcal{S}_c^0\::\:  \liminf_{x \to \infty} S_x > -\infty \right\}.\]
and we further define, for $S \in \mathcal{S}_{c}^{T^{-1}}$,
\[I_x=\inf_{y \ge x} S_y, \qquad V_x=S_x-I_x, \qquad (T^{-1}S)_x=2I_x-S_x-2I_0.\]

\subsection{Reversible set and invariant set} As in the discrete case, it is natural to seek to characterise the sets
\[\mathcal{S}^{rev}_c:=\left\{S \in \mathcal{S}_c^0\::\:TS,T^{-1}S,T^{-1}TS,TT^{-1}S\mbox{ well-defined},\:T^{-1}TS=S,\:TT^{-1}S=S \right\},\]
and
\[\mathcal{S}^{inv}_c:=\left\{S \in \mathcal{S}_c^0\::\: T^kS \in \mathcal{S}^{rev}_c,\: \forall k \in \mathbb{Z} \right\},\]
i.e.\ the set where the one-step dynamics (forward and backward) are well-defined are reversible, and the set where the (forwards and backwards) dynamics are well-defined and consistent for all time. To this end, we recall that the functions $\Phi$ and $\Psi$ were useful in the discrete setting. Here, we will introduce the continuous analogues of these functions, however a notable difference is that in this setting $M-M_0$ is not a function of $W$ in general. (In the discrete case, we always have $M-M_0=\ell(W)$.)

Let the spaces $\mathcal{Y}_c$, $\mathcal{Y}_c^{\pm}$, $\mathcal{A}_c$ and $\mathcal{A}_c^0$ be given by:
\[\mathcal{Y}_c=\left\{ Y :\mathbb{R} \to \mathbb{R}_+ \::\:  Y\mbox{ continuous}\right\}, \qquad \mathcal{Y}_c^{\pm}:=\left\{ Y \in \mathcal{Y}_c \::\: \liminf_{x \to \pm \infty}Y_x=0 \right\},\]
\[ \mathcal{A}_c:=\left\{ A :\mathbb{R} \to \mathbb{R} \::\: A\mbox{ continuous, non-decreasing}\right\}, \qquad \mathcal{A}_c^0:=\left\{ A \in \mathcal{A}_c \::\: A_0=0 \right\}.\]
We then define $\Phi$ and $\Psi$ by setting
\begin{eqnarray*}
\begin{array}{rcl}
  \Phi:\mathcal{Y}_c \times \mathcal{A}^0_c &\to& \mathcal{S}_c^0\\
  (Y,A) & \mapsto & A-Y+Y_0,
\end{array}&&
\begin{array}{rcl}
  \Psi:\mathcal{Y}_c \times \mathcal{A}^0_c &\to& \mathcal{S}_c^0\\
  (Y,A) & \mapsto & A+Y-Y_0,
\end{array}
\end{eqnarray*}
and introduce the corresponding inverses as follows:
\begin{eqnarray*}
\begin{array}{rcl}
\Phi^{-1}: \mathcal{S}^T_c &\to& \mathcal{Y}_c \times \mathcal{A}^0_c\\
 S&\mapsto&(M-S,M-M_0),
 \end{array}&&
 \begin{array}{rcl}
\Psi^{-1}: \mathcal{S}^{T^{-1}}_c &\to& \mathcal{Y}_c \times \mathcal{A}^0_c\\
 S&\mapsto&(S-I,I-I_0),
 \end{array}
 \end{eqnarray*}
We readily see that $\Phi\Phi^{-1}S=S$ and $\Psi\Psi^{-1}S=S$. Also, the relations $TS=\Psi \Phi^{-1}S$ and $T^{-1}S=\Phi \Psi^{-1}S$ are obviously satisfied. To characterize $\Phi^{-1}( \mathcal{S}^T_c)$ and $\Psi^{-1}(\mathcal{S}^{T^{-1}}_c)$ (cf. Propositions \ref{phiproperty} and \ref{adapprop}), we introduce the following sets of pairs of functions.
\[(\mathcal{Y}_c \times \mathcal{A}^0_c)^{SK}=\left\{(Y,A) \in \mathcal{Y}_c \times \mathcal{A}^0_c\::\: \int_0^x Y_u dA_u=0,\:\forall x \in \mathbb{R} \right\}.\]
and
\[(\mathcal{Y}_c^{\pm} \times \mathcal{A}^0_c)^{SK} = \left\{ (Y,A) \in (\mathcal{Y}_c \times \mathcal{A}^0_c)^{SK} \::\: Y \in \mathcal{Y}_c^{\pm} \right\}.\]
It is clear that $\Phi^{-1}( \mathcal{S}^T_c) \subseteq (\mathcal{Y}_c^{-} \times \mathcal{A}^0_c)^{SK}$ and $\Psi^{-1}( \mathcal{S}^{T^{-1}}_c) \subseteq (\mathcal{Y}_c^{+} \times \mathcal{A}^0_c)^{SK}$. To show that these sets are respectively equal, we show the following theorem, which is a two-sided version of the Skorohod problem (see \cite[Lemma VI.2.1]{RY} for a statement of the classical result).

{\thm\label{twosidedskor} If $S \in \mathcal{S}^T_c$, then there exists a unique pair $(Y,A)$ satisfying the following conditions.\\
(i) $Y \in \mathcal{Y}_c^{-}$.\\
(ii) $A \in \mathcal{A}$.\\
(iii) The support of $dA$ is contained in $\{x \in \mathbb{R}:\: Y_x=0\}$, or equivalently $\int_0^x Y_u dA_u=0$ for all $x \in \mathbb{R}$.\\
(iv) $S=A-Y$.\\
Moreover, the pair is given by $(Y,A)=(M-S,M)$ and for any $(\tilde{Y},\tilde{A})\in (\mathcal{Y}_c \times \mathcal{A}^0_c)^{SK}$ satisfying $S=A-Y$, $Y_x \le \tilde{Y}_x$ and $A_x \le \tilde{A}_x$ for all $x \in \mathbb{R}$.}

\begin{proof} It is easy to see that $(M-S,M)$ satisfies the desired conditions. We only need to show the uniqueness. Suppose that $(\tilde{Y},\tilde{A})$ also satisfies the same condition but $(\tilde{Y}, \tilde {A})\neq (Y,A)$, where $(Y,A)=(M-S,M)$. Note that $A-\tilde{A}=S+Y-(S+\tilde{Y})=Y-\tilde{Y}$. First, we show that $\tilde{A}_x \ge A_x$ for all $x \in \mathbb{R}$. Suppose that $\tilde{A}_{x_0} < A_{x_0}$ for some $x_0$. Then, since $A=M$, there exists $y \le  x_0$ such that $\tilde{A}_{x_0} < S_y \le A_{x_0}$. However, this implies $\tilde{Y}_y=\tilde{A}_y -S_y <\tilde{A}_y -\tilde{A}_{x_0} \le \tilde{A}_{x_0} -\tilde{A}_{x_0} =0$, which gives a contradiction. Next we suppose that there exists $\tilde{A}_{x_0} > A_{x_0}$ for some $x_0$. Then, for any $x \le x_0$,
\begin{align*}
(Y-\tilde{Y})^2_{x}-(Y-\tilde{Y})^2_{x_0}  =-2 \int_{x}^{x_0}( Y-\tilde{Y})_u d (A_u-\tilde{A}_u) =2\int_{x}^{x_0} Y_u d\tilde{A}_u + 2\int_{x}^{x_0}\tilde{Y}_u dA_u \ge 0.
\end{align*}
So, $(Y-\tilde{Y})^2_{x} \ge (Y-\tilde{Y})^2_{x_0} > 0$ for any $x \le x_0$. Since $\tilde{Y}-Y$ is positive, we have $\tilde{Y}_x \ge c >0$ for all $x \le x_0$, where $c=\tilde{Y}_{x_0}-Y_{x_0}$. On the other hand, $\displaystyle \liminf_{x \to -\infty}\tilde{Y}_x=0$ by assumption, and so we have arrived at a contradiction.

For the final claim, we simply repeat the first part of the argument.
\end{proof}

We can also prove the following version in the same manner.

\begin{prop} If $S \in \mathcal{S}^{T^{-1}}_c$, then there exists a unique pair $(Y,A)$ satisfying the following conditions.\\
(i) $Y \in \mathcal{Y}_c^{-}$.\\
(ii) $A \in \mathcal{A}$.\\
(iii) The support of $dA$ is contained in $\{x \in \mathbb{R}:\: Y_x=0\}$, or equivalently $\int_0^x Y_u dA_u=0$ for all $x \in \mathbb{R}$.\\
(iv) $S=A+Y$.\\
Moreover, the pair is given by $(Y,A)=(S-I,I)$ and for any $(\tilde{Y},\tilde{A})\in (\mathcal{Y}_c \times \mathcal{A}^0_c)^{SK}$ satisfying $S=A+Y$, $Y_x \le \tilde{Y}_x$ and $A_x \ge \tilde{A}_x$ for all $x \in \mathbb{R}$.
\end{prop}

Applying the above results, we have the continuous counterpart of Propositions \ref{phiproperty} and \ref{adapprop}.

\begin{prop} The map
\[\Phi|_{(\mathcal{Y}_c^{-} \times \mathcal{A}^0_c)^{SK}} : (\mathcal{Y}_c^{-} \times \mathcal{A}^0_c)^{SK}  \to \mathcal{S}^T \]
is a bijection with inverse operator $\Phi^{-1}$. Also,
\[\Psi|_{(\mathcal{Y}_c^{+} \times \mathcal{A}^0_c)^{SK}} : (\mathcal{Y}_c^{+} \times \mathcal{A}^0_c)^{SK}  \to \mathcal{S}^{T^{-1}}\]
is a bijection with inverse operator $\Psi^{-1}$. Moreover,
\[\Phi^{-1} \Phi((\mathcal{Y}_c^{+} \times \mathcal{A}^0_c)^{SK}) \subseteq (\mathcal{Y}_c^{+} \times \mathcal{A}^0_c)^{SK},\qquad \Psi^{-1} \Psi((\mathcal{Y}_c^{-} \times \mathcal{A}^0_c)^{SK}) \subset (\mathcal{Y}_c^{-} \times \mathcal{A}^0_c)^{SK}.\]
\end{prop}
\begin{proof} We only give a proof for $\Phi$. The proof for $\Psi$ is similar. Suppose there exist $(Y,A), (\tilde{Y},\tilde{A}) \in (\mathcal{Y}_c^{-} \times \mathcal{A}^0_c)^{SK}$ satisfying $\Phi(Y,A)=A-Y+Y_0 =\tilde{A}-\tilde{Y}+\tilde{Y}_0= \Phi(\tilde{Y},\tilde{A})$, and denote this element of $\mathcal{S}_c^0$ by $S$. Then $M_0 \le Y_0$, and so $M_0 \in \mathbb{R}$. We can thus apply Theorem \ref{twosidedskor} to deduce that $(Y,A+Y_0)$ and $(\tilde{Y},\tilde{A}+\tilde{Y}_0)$ are the same, since they both solve the relevant Skorohod problem for $S$. Hence we obtain $Y=\tilde{Y}$ and $A=\tilde{A}$.
\end{proof}

With these observations, one can conclude the same characterization of the set $\mathcal{S}^{rev}_c$ in the continuous case as was given in the discrete case in Theorem \ref{mr1}.

\begin{thm} It holds that
\[\mathcal{S}^{rev}_c=\left\{S\in \mathcal{S}^0_c\::\:M_0<\infty,\:I_0>-\infty,\: \limsup_{x\rightarrow\infty}S_x=M_\infty,\;\liminf_{x\rightarrow-\infty}S_x=I_{-\infty}\right\},\]
where the limits $M_\infty =\lim_{x\rightarrow\infty}M_x=\sup_{x \in\mathbb{R}}S_x$ and $I_{-\infty}=\lim_{x\rightarrow-\infty}I_x=\inf_{x\in\mathbb{R}}S_x$ are well-defined by monotonicity.
\end{thm}

We can also obtain a continuous version of Lemma \ref{dualityrel}. To state this, we define a map $\tilde{R}: \mathcal{Y}_c \times \mathcal{A}^0_c \to \mathcal{Y}_c \times \mathcal{A}^0_c$ by setting $\tilde{R}(Y,A)=(\tilde{R}Y,\tilde{R}A)$, where $\tilde{R}Y_x=Y_{-x}$ and $\tilde{R}A_x=-A_{-x}$.

\begin{lem} It holds that
\[ R\Psi=\Phi\tilde{R}, \qquad R\Phi=\Psi\tilde{R}.\]
Moreover, $\tilde{R}((\mathcal{Y}_c^{-} \times \mathcal{A}^0_c)^{SK})=(\mathcal{Y}_c^{+} \times \mathcal{A}^0_c)^{SK}$, $\tilde{R}((\mathcal{Y}_c^{+} \times \mathcal{A}^0_c)^{SK})=(\mathcal{Y}_c^{-} \times \mathcal{A}^0_c)^{SK}$ and the maps in the following diagram are all bijections and commutative.
\[\xymatrix@C+30pt@R+10pt
{      (\mathcal{Y}_c^{-} \times \mathcal{A}^0_c)^{SK} \ar@<.5ex>[r]^-\Phi \ar[d]  &  \mathcal{S}^T_c \ar[d]^-{R} \ar@<.5ex>[l]^-{\Phi^{-1}}\\
                  (\mathcal{Y}_c^{+} \times \mathcal{A}^0_c)^{SK} \ar[u]^-{\tilde{R}} \ar@<.5ex>[r]^-\Psi   &  \mathcal{S}^{T^{-1}}_c \ar[u] \ar@<.5ex>[l]^-{\Psi^{-1}}
                  }\]
Also, the following diagram satisfies the same property.
\[\xymatrix@C+30pt@R+10pt
{      (\mathcal{Y}_c^{rev} \times \mathcal{A}^0_c)^{SK} \ar@<.5ex>[r]^-\Phi \ar[d]  &  \mathcal{S}^{rev}_c \ar[d]^-{R} \ar@<.5ex>[l]^-{\Phi^{-1}}\\
                   (\mathcal{Y}_c^{rev} \times \mathcal{A}^0_c)^{SK}   \ar[u]^-{\tilde{R}} \ar@<.5ex>[r]^-\Psi   &  \mathcal{S}^{rev}_c \ar[u] \ar@<.5ex>[l]^-{\Psi^{-1}}
                  }\]
where $(\mathcal{Y}_c^{rev} \times \mathcal{A}^0_c)^{SK} =(\mathcal{Y}_c^{+} \times \mathcal{A}^0_c)^{SK} \cap (\mathcal{Y}_c^{-} \times \mathcal{A}^0_c)^{SK}$.
\end{lem}

For the characterization of $\mathcal{S}^{inv}_c$, we only need to make minor changes of the argument used in the discrete case to establish the corresponding result, and so we omit the details and just give a statement. To this end, let us introduce some notation. For any strictly increasing function $F : \mathbb{R} \to \mathbb{R}$ satisfying $\lim_{x \rightarrow \infty} F(x)=\infty$, define
\[\mathcal{S}^+_{c,F}:=\left\{S\in\mathcal{S}^0_c:\:\:\lim_{x\rightarrow \infty}\frac{S_x}{F(x)}=1 \right\},\]
and for any strictly increasing function $F : \mathbb{R} \to \mathbb{R}$ satisfying $\lim_{x \rightarrow -\infty} F(x)=-\infty$,
\[\mathcal{S}^-_{c,F}:=\left\{S\in\mathcal{S}^0_c:\:\:\lim_{x\rightarrow -\infty}\frac{S_x}{F(x)}=1 \right\}.\]
Also, for any nonnegative real number $K$, let
\[\mathcal{S}^+_{c,K}:=\left\{S \in \mathcal{S}^0_c\::\:  \sup_{x \in \mathbb{R}} (M_x -I_x) =K, \: \limsup_{x\to \infty}S_x - \liminf_{x \to \infty}S_x=K \right\},\]
\[\mathcal{S}^+_{c,K}:=\left\{S \in \mathcal{S}^0_c\::\:  \sup_{x \in \mathbb{R}} (M_x -I_x) =K, \: \limsup_{x\to -\infty}S_x - \liminf_{x \to -\infty}S_x=K \right\}.\]

\begin{thm} For $S \in \mathcal{S}^0_c$, $S \in \mathcal{S}^{inv}_c$ if and only if $S \in \mathcal{S}^-_{c,*_1} \cap \mathcal{S}^+_{c,*_2}$, where $*_1$ and $*_2$ are some $F$ or $K$. Moreover, if the condition holds, then $T^kS \in \mathcal{S}^-_{c,*_1} \cap \mathcal{S}^+_{c,*_2}$ for any $k\in \mathbb{Z}$.
\end{thm}

\subsection{Invariance in distribution} Given the set-up in the previous section, it is now straightforward to check the continuous counterpart of Theorem \ref{mrd}. Since the proof is identical to the latter result, we omit it.

\begin{thm}\label{mr2c} Suppose $S$ is a random process supported on $\mathcal{S}^{rev}_c$. It is then the case that any two of the three following conditions imply the third:
\[RS\buildrel{d}\over{=}S,\qquad \tilde{R}(W,M-M_0)\buildrel{d}\over{=}(W,M-M_0),\qquad TS\buildrel{d}\over{=}S.\]
Moreover, in the case that two of the above conditions are satisfied, then the distribution of $S$ is actually supported on $\mathcal{S}^{inv}_c$.
\end{thm}

The second condition in the previous result, $\tilde{R}(W,M-M_0)\buildrel{d}\over{=}(W,M-M_0)$, is more complicated than the corresponding condition for the discrete setting. The following proposition provides a condition under which we can revert to the simpler requirement.

\begin{prop}\label{localtimed}
Suppose $S$ is a random process supported on $\mathcal{S}^{rev}_c$. If there exists a measurable function $L : \mathcal{Y}_c \to \mathcal{A}^0_c$ such that, $\mathbf{P}$-a.s.,
\[L(W)=M-M_0, \qquad L(\tilde{R}W)=\tilde{R}(M-M_0), \]
then the conditions $\tilde{R}(W,M-M_0)\buildrel{d}\over{=}(W,M-M_0)$ and $\bar{W}\buildrel{d}\over{=}W$ are equivalent, where, as in the discrete setting, we write $\bar{W}=\tilde{R}W$.
\end{prop}

\subsection{Brownian motion with drift}

As an example of a continuous invariant measure for $T$, we consider the process $S=(S_x)_{x \in \mathbb{R}}$ given by a two-sided Brownian motion with positive drift $c$. Namely, for $x\geq 0$, we define $S_x=B^1_x+cx$, $S_{-x}=-(B^2_x+cx)$, where $B^1,B^2$ are independent Brownian motions. We will write $\nu_c$ for the law of this process.

The main theorem of this subsection is the following. We will give two different proofs. The first uses the classical result on the scaling limit of simple random walks. The second is a direct application of Theorem \ref{mr2c}.

{\thm\label{invbm} If $S$ is the two-sided Brownian motion with positive drift $c$, then $TS \buildrel{d}\over{=}S$.}

\subsubsection{Proof of Theorem \ref{invbm} via simple random walk scaling limit} To begin with, we introduce the notation $\mu^{p}$ to represent the probability measure on $\mathcal{S}^0_c$ given by the linear interpolation of the two-sided random walk with one step distribution given by $P(S_n-S_{n-1}=1)= 1-p$ and $P(S_n-S_{n-1}=-1)=p$. As shown in Section \ref{examplessec}, we have the invariance of $\mu^{p}$ under $T$ for $p < \frac{1}{2}$. We will work on the high density limit of the random walk ($p\rightarrow\frac12$) to transfer the latter result to the Brownian motion with drift.

We start by presenting two lemmas. For a probability measure $\mu$ on $\mathcal{S}^0_c$ and $a, b >0$, we write $\mu_{a,b}$ to be the scaled measure given by
\[\mu_{a,b}\left(S \in A\right)= \mu\left( aS_{b \cdot} \in A\right).\]

\begin{lem} Let $a, b >0$. If $\mu$ is invariant under $T$, then $\mu_{a,b}$ is also invariant under $T$.
\end{lem}
\begin{proof}
Let $S^{a,b}_x=aS_{bx}$ for $a,b >0$ and $x \in \mathbb{R}$.
The claim follows from the simple observation that $TS^{a,b}=(TS)^{a,b}$.
\end{proof}

\begin{lem}
Suppose $\{\mu_n\}$ is a sequence of probability measures on $\mathcal{S}^0_c$, each of which is invariant under $T$, and $\mu_n$ converges weakly to $\mu$. Moreover, suppose that $\mu_n$ satisfies for any $y \in \mathbb{R}$,
\[\lim_{x \to -\infty} \limsup_{n \to \infty} \mu_n(M_x > S_y)=0\]
and $\mu$ satisfies for any $y \in \mathbb{R}$,
\[\lim_{x \to -\infty} \mu(M_x > S_y)=0.\]
It then holds that $\mu$ is also invariant under $T$.
\end{lem}
\begin{proof}
We need to show that for any $L>0$ and continuous bounded function $f: C([-L,L],\mathbb{R}) \to \mathbb{R}$,
\[\mu\left(f \left(S|_{[-L,L]}\right)\right) = \mu\left(f \left(TS|_{[-L,L]}\right)\right).\]
Let
\[
M^{L'}_x:=\left\{\begin{array}{ll}
                   M^{L'}_x=S_{-L'}, & \mbox{if }x < -L', \\
                   \sup_{-L' \le y \le x} S_y, & \mbox{if }-L' \le x \le L',\\
                   \sup_{-L' \le y \le L'} S_y, & \mbox{otherwise.}
                 \end{array}\right.\]
Also, let $(T^{L'}S)_x:=2M^{L'}_x-S_x-2M^{L'}_0$. Then, $T^{L'} :\mathcal{S}^0_c \to \mathcal{S}^0_c $ is continuous, and so
\begin{equation}\label{argxxx}
\lim_{n \to \infty} \mu_n \left(f \left((T^{L'}S)|_{[-L,L]}\right)\right) = \mu\left(f \left((T^{L'}S)|_{[-L,L]}\right)\right),
\end{equation}
for any $L,L'$. Moreover, if $L < L'$ and $M_{-L'} \le S_{-L}$, then $(T^{L'}S)|_{[-L,L]}=(TS)|_{[-L,L]}$. Therefore, for any $L' >L$,
\[\left|\mu_n\left(f \left((T^{L'}S)|_{[-L,L]}\right)\right) - \mu_n\left(f \left((TS)|_{[-L,L]}\right)\right)\right| \le 2 \|f\|_{\infty} \mu_n \left(M_{-L'} > S_{-L}\right).\]
Hence, by assumption, we have that
\[\lim_{L' \to \infty} \lim_{n \to \infty}\left|\mu_n\left(f \left((T^{L'}S)|_{[-L,L]}\right)\right) - \mu_n\left(f \left((TS)|_{[-L,L]}\right)\right)\right|=0,\]
which in conjunction with \eqref{argxxx} implies
\begin{equation}\label{klim}
\lim_{n \to \infty}\mu_n\left(f \left((TS)|_{[-L,L]}\right)\right)=
\lim_{L' \to \infty}  \mu(f (T^{L'}S|_{[-L,L]})).
\end{equation}
Finally observe that since we also have
\[
\left|\mu\left(f \left(T^{L'}S|_{[-L,L]}\right)\right) - \mu\left(f \left(TS|_{[-L,L]}\right)\right)\right| \le 2 \|f\|_{\infty} \mu \left(M_{-L'} > S_{-L}\right),\]
the assumption $\lim_{x \to -\infty}  \mu(M_x > S_y)=0$ for any $x$ implies that the right-hand side of \eqref{klim} is equal to $\mu(f((TS|_{[-L,L]}))$, as desired.
\end{proof}

We next check the assumptions of the previous result for the specific processes of interest.

\begin{lem} Let $c>0$, $p_n=\frac{1}{2}-\frac{c}{2n}$, and $\nu_n:=\mu^{p_n}_{n^{-1},n^{-2}}$. For any $y \in \mathbb{R}$,
\[\lim_{x \to -\infty} \limsup_{n \to \infty} \nu_n\left(M_x > S_y\right)=0\]
and
\[\lim_{x \to -\infty} \nu_c\left(M_x > S_y\right)=0.\]
\end{lem}
\begin{proof} Since $\nu_c (\limsup_{x -\infty}S_x=-\infty)=1$, the second claim of the lemma is obvious. To estimate the probability $\nu_n(M_x > S_y)$, first note that, for any $x<y$,
\begin{align*}
\nu_n \left(M_x > S_y\right) & \le \mu^{p_n} \left(M_{[x n^2]+1} > \min\left\{S_{[yn^2]},S_{[yn^2]+1} \right\}\right)\\
& = \mu^{p_n}\left(M_{[x n^2]+1-[yn^2]} > \min\left\{S_0,S_1\right\} \right)\\
& \le \mu^{p_n} \left(M_{[x n^2]+1-[yn^2]} \ge 0 \right),
\end{align*}
where $[z]$ is the maximum integer not greater than $z$. Thus we only need to show that
\[\lim_{x \to -\infty} \limsup_{n \to \infty} \mu^{p_n} \left(M_{[x n^2]} \ge 0 \right)=0.\]
For any $\ell \ge 1$, we have
\begin{align*}
\mu^{p_n} \left(M_{-\ell} \ge 0 \right) &\leq  \mu^{p_n} \left(S_{-\ell} \ge 0\right) + \sum_{k  \le -1}\mu^{p_n} \left(S_{-\ell}=k\right)\left(\frac{p_n}{1-p_n}\right)^{-k}.
\end{align*}
Now, since $S_{-\ell}\buildrel{d}\over{=}-S_\ell=-\sum_{k=1}^{\ell}(1-2\eta_k)$, we have
\begin{align*}
\mu^{p_n} \left(S_{-\ell} \ge 0\right) &= \mu^{p_n}\left(\sum_{k=1}^{\ell}(1-2\eta_k) \le 0\right)\\
& \le \mu^{p_n}\left(\frac{1}{\ell}\left|\sum_{k=1}^{\ell}(1-2\eta_k) - \frac{\ell c}{n}\right| \ge \frac{c}{n}\right)\\
& \le \frac{n^2}{\ell c^2} E\left(\left(\left(1-2\eta_k\right)-\frac{c}{n}\right)^2\right)\\
& \le \frac{n^2}{\ell c^2}.
\end{align*}
Moreover,
\begin{align*}
 \sum_{k  \le -1}\mu^{p_n} \left(S_{-\ell}=k\right)\left(\frac{p_n}{1-p_n}\right)^{-k}&= \sum_{-\ell \le k  \le -1} \binom{\ell}{\frac{\ell+k}{2}} p_n^{\frac{\ell+k}{2}}(1-p_n)^{\frac{\ell-k}{2}}\left(\frac{p_n}{1-p_n}\right)^{-k}\\
 &=\sum_{-\ell \le k  \le -1} \binom{\ell}{\frac{\ell+k}{2}} p_n^{\frac{\ell-k}{2}}(1-p_n)^{\frac{\ell+k}{2}} \\
& =  \sum_{-\ell \le k  \le -1}\mu^{p_n} (S_{-\ell}=-k)\\
&=\mu^{p_n} (S_{-\ell} \ge 1)\\
& \le \mu^{p_n} (S_{-\ell} \ge 0),
\end{align*}
where $\binom{\ell}{q} \equiv 0$ for $q \notin \mathbb{N}$. Therefore, we have
\[\lim_{x \to -\infty} \limsup_{n \to \infty} \mu^{p_n} \left(M_{[x n^2]} \ge 0\right) \le \lim_{x \to -\infty} \limsup_{n \to \infty}  \frac{2n^2}{[x n^2] c^2} =\lim_{x \to -\infty} \frac{2}{|x|c^2} =0.\]
\end{proof}

Now, keeping the notation $p_n=\frac{1}{2}-\frac{c}{2n}$, the classical invariance principle shows that $\mu^{p_n}_{n^{-1},n^{-2}}$ converges weakly to $\nu^c$, as $n$ goes to infinity. Thus combining the above lemmas yields Theorem \ref{invbm}.

\subsubsection{Proof of Theorem \ref{invbm} via Theorem \ref{mr2c}} We now give our second proof of Theorem \ref{invbm}, which will be via Theorem \ref{mr2c}. To this end, we only need to show that
\[RS\buildrel{d}\over{=}S,\qquad \tilde{R}(W,M-M_0)\buildrel{d}\over{=}(W,M-M_0)\]
under $\nu^c$. By definition, $RS\buildrel{d}\over{=}S$ is obvious. Also, by Proposition \ref{localtimed}, we only need to show that the existence of a measurable function $L:\mathcal{Y}_c\rightarrow\mathcal{A}_c^0$ with the relevant properties, and then check the condition $\bar{W} \buildrel{d}\over{=}W$. This is the aim of the next two lemmas, which complete the proof of Theorem \ref{invbm}.

\begin{lem} Let
\[L(Y)_x:=\limsup_{\epsilon \to 0} \frac{1}{\epsilon} \int_0^x \mathbf{1}_{\{Y_u \le \epsilon\}} du.\]
It is possible to suppose that, $\nu^c$-a.s.,
\[\left(W,L(W)\right) = \left(M-S,M-M_0\right),\qquad L(\tilde{R}W)=\tilde{R}L(W).\]
\end{lem}
\begin{proof} See the proof of \cite[Proof of Theorem 3.1]{Peskir} for a construction of the relevant random variables in such a way that the first equality holds, $\nu^c$-a.s. The second equality is obvious from the definition of $L$.
\end{proof}

\begin{lem}
Under $\nu^c$,
\[\bar{W} \buildrel{d}\over{=}W. \]
\end{lem}
\begin{proof}
Let $M^y_x:=\sup_{y \le z \le x} S_z$. Then,
\[W_x=\max\{M_y, M^y_x\}-S_x=\max\{M_y-S_y, M^y_x-S_y\}-(S_x-S_y)\]
for $x \ge y$. Since $M_y-S_y$ and $(S_x-S_y)_{x \ge y}$ are independent, from \cite{Peskir}, $(W_x)_{x \ge y}$ is the reflected Brownian motion with negative drift $c>0$ starting from $M_y-S_y$ at time $y$. Also, the distribution of $M_y-S_y$ is exponential with parameter $2c$, which is the stationary probability measure of the reflected Brownian motion with negative drift $c>0$. From \cite[Sections 8 and 9]{Kent}, for example, it follows that $\bar{W} \buildrel{d}\over{=}W$.
\end{proof}

\begin{rem} By exactly the same argument as for the discrete case in Theorem \ref{currentcltthm}, it is possible to check that, under $\nu_c$, the sequence $((T^kW)_0)_{k \in \mathbb{Z}}$ is i.i.d. Moreover, as we noted in the previous proof, $W_0$ is exponentially distributed with parameter $2c$. Thus the integrated current also satisfies a law of large numbers, central limit theorem and large deviations principle in this setting (cf.\ Theorem \ref{currentcltthm}).
\end{rem}

\section{Open questions}\label{oq}

In this section, we collect some of the questions that arise from the present work.

\begin{enumerate}
  \item As part of Theorem \ref{mrc}, we show that the invariance in distribution of a sub-critical random configuration $\eta$ under $T$ is equivalent to the stationarity of the current sequence $((T^kW)_0)_{k\in\mathbb{Z}}$ under the canonical shift. Whilst this result does give a characterisation of sub-critical invariant measures, it is slightly unsatisfactory, as the latter condition might not be straightforward to verify in examples. Ideally, we would like to give a complete characterisation of the invariant measures of $T$ in terms of basic properties of the initial configuration $\eta$ and carrier process $W$. In Theorem \ref{mrd} we give sufficient conditions in terms of the symmetry of $\eta$ and $W$, which are verifiable in the sub-critical examples of Theorem \ref{mre}. To what extent is it possible to go beyond this? One might further consider how the answer to this question is related to the soliton decomposition of \cite{Ferrari}.
  \item Similarly to the previous question, one might also hope to provide a complete characterisation of measures for which the distribution of $\eta$ is ergodic under $T$. Again, in the sub-critical case, Theorem \ref{mrc} provides something of an answer, establishing that ergodicity of the configuration is equivalent to the ergodicity of the current sequence. We also show that this criterion is applicable in the sub-critical examples of Theorem \ref{mre} (see Corollary \ref{ergcormrf} in particular). To what extent can the ergodicity of $((T^kW)_0)_{k \in\mathbb{Z}}$ be established more generally? In particular, is it always ergodic when $\eta$ is a stationary, ergodic sequence (under spatial shifts) satisfying $T\eta\buildrel{d}\over=\eta$ and whose path encoding $S$ is supported on $\mathcal{S}_{sub-critical}$? (We recall that, when $S$ is supported on $\mathcal{S}_{critical}$, Theorem \ref{mrb} gives that $T$ is only ergodic in the trivial case.)
  \item Following on from Remark \ref{gibbs}, one might ask more about the representation of invariant measures of the BBS as Gibbs measures. In particular, using the notation of the remark, is there a convenient way to express the functions $f_k$ for $k \ge 2$? What are the necessary and sufficient conditions on $(\beta_k)_{k \ge 0}$ for the associated Gibbs measure to exist and be invariant? Among invariant measures, how are the Gibbs measures distinguished from others?
  \item In the discussion preceding Proposition \ref{densityconstant}, it was observed that the distributions of one class of invariant configurations are given by $Q^{\otimes \mathbb{Z}} \circ \Lambda$, i.e.\ distributions for which the current forms an i.i.d.\ sequence with law $Q$. Beyond the i.i.d.\ case (or periodic generalisations of this, as described in Remark \ref{nonsrem}), where $Q$ is the distribution of a geometric random variable (or multiple thereof), is it possible to characterise these particle configurations more explicitly?
  \item Apart from establishing a central limit theorem and large deviations principle  for the tagged particle process in the i.i.d.\ case under the FIFO scheme, which was not achieved in Section \ref{distancesec}, and also checking further properties of the tagged particle for the other example configurations from Theorem \ref{mre} beyond those discussed in Remarks \ref{fiforem} and \ref{liforem}, it would be natural to study the distance travelled by a tagged particle in a more general setting. In particular, one might consider the environment viewed from the particle for either the FIFO or LIFO scheme. What are the invariant measures (on particle configurations such that $\eta_0=1$) for this process? Is the environment process ergodic with respect to these?
  \item In the continuous case (the BBS on $\mathbb{R}$ of Section \ref{contsec}), many of the same questions are relevant. Is it possible to completely characterise the invariant and ergodic measures of the BBS dynamics? Is there a soliton decomposition for these, cf.\ \cite{Ferrari}? How does invariance and ergodicity of configurations relate to stationarity and ergodicity of the current sequence? Moreover, when is the current sequence stationary and ergodic?
  \item A range of other ultradiscrete integrable systems have been studied, including variants of the BBS with multi-valued box capacities \cite{Taka}, carrier capacities \cite{TakaMatsu}, and so on. Is it possible to study these by arguments similar to those applied in this article?
  \item An unexpected structural similarity has been discovered between the BBS and the lamplighter group from the spectral view point \cite{KTZ}. Is it possible to study automata groups such as the lamplighter group via the approach of this article?
  \item In the high-density regime considered in Section \ref{contsec}, the limiting dynamics remained in discrete time. It would be of interest to explore whether, under a suitable scaling regime, it is possible to obtain a continuous time dynamical system from the BBS, and determine what its evolution rules and properties are.
  \item  As noted in the introduction, the box-ball system can be obtained from a tropicalisation/ ultradiscretisation of the discrete KdV equation, as presented at (\ref{discretekdv}) \cite{HT:ukdv, TTMS}. It can also be obtained from an ultradiscretisation of a discrete Toda equation \cite{MT:udToda,NTT:BBS-sorting}.
      To what extent do the results of this work yield insights into the rational dynamics governed by the latter equations?
\end{enumerate}

\section*{Acknowledgements}
We thank Pablo Ferrari for helpful discussions. DC would like to thank MS for her generous support and kind hospitality during two visits to the University of Tokyo in 2017, which is when the majority of their contribution to the article was completed. The research of TK was supported by JSPS KAKENHI (grant number 17K18725), that of ST by JSPS KAKENHI (grant number 16K13761), and that of MS by JSPS KAKENHI (grant number 16KT0021).

\providecommand{\bysame}{\leavevmode\hbox to3em{\hrulefill}\thinspace}
\providecommand{\MR}{\relax\ifhmode\unskip\space\fi MR }
\providecommand{\MRhref}[2]{%
  \href{http://www.ams.org/mathscinet-getitem?mr=#1}{#2}
}
\providecommand{\href}[2]{#2}


\begin{thebibliography}{10}

\bibitem{DZ}
A.~Dembo and O.~Zeitouni, \emph{Large deviations techniques and applications},
  Stochastic Modelling and Applied Probability, vol.~38, Springer-Verlag,
  Berlin, 2010, Corrected reprint of the second (1998) edition.

\bibitem{Ferrari}
P.~A. Ferrari, \emph{Ball box system in $\mathbb{Z}$}, slides from Information
  and Randomness, Santiago, 2016.

\bibitem{FerFont}
P.~A. Ferrari and L.~R.~G. Fontes, \emph{Current fluctuations for the
  asymmetric simple exclusion process}, Ann. Probab. \textbf{22} (1994), no.~2,
  820--832.

\bibitem{FerSpo}
P.~L. Ferrari and H.~Spohn, \emph{Scaling limit for the space-time covariance
  of the stationary totally asymmetric simple exclusion process}, Comm. Math.
  Phys. \textbf{265} (2006), no.~1, 1--44.

\bibitem{GO}
P.~W. Glynn and D.~Ormoneit, \emph{Hoeffding's inequality for uniformly ergodic
  {M}arkov chains}, Statist. Probab. Lett. \textbf{56} (2002), no.~2, 143--146.

\bibitem{GT}
P.~W. Glynn and H.~Thorisson, \emph{Two-sided taboo limits for {M}arkov
  processes and associated perfect simulation}, Stochastic Process. Appl.
  \textbf{91} (2001), no.~1, 1--20.

\bibitem{Gon}
P.~Gon\c{c}alves, \emph{Equilibrium fluctuations for totally asymmetric
  particle systems}, VDM Verlag Dr.\ M\"{u}ller e.K., 2010.

\bibitem{HMOC}
B.~M. Hambly, J.~B. Martin, and N.~O'Connell, \emph{Pitman's {$2M-X$} theorem
  for skip-free random walks with {M}arkovian increments}, Electron. Comm.
  Probab. \textbf{6} (2001), 73--77.

\bibitem{HW}
J.~M. Harrison and R.~J. Williams, \emph{On the quasireversibility of a
  multiclass {B}rownian service station}, Ann. Probab. \textbf{18} (1990),
  no.~3, 1249--1268.

\bibitem{hirota}
R.~Hirota, \emph{Nonlinear partial difference equations {I}}, Journal of Phys.
  Soc. Japan \textbf{43} (1977), 1424--1433.

\bibitem{vdh}
R.~van~der Hofstad and M.~Keane, \emph{An elementary proof of the hitting time
  theorem}, Amer. Math. Monthly \textbf{115} (2008), no.~8, 753--756.

\bibitem{IKT}
R.~Inoue, A.~Kuniba, and T.~Takagi, \emph{Integrable structure of box-ball
  systems: crystal, {B}ethe ansatz, ultradiscretization and tropical geometry},
  J. Phys. A \textbf{45} (2012), no.~7, 073001, 64.

\bibitem{Kbook}
T.~Kato, \emph{Dynamical scale transform in tropical geometry}, World
  Scientific Publishing Co. Pte. Ltd., Hackensack, NJ, 2017.

\bibitem{KTZ}
T.~Kato, S.~Tsujimoto and A.~Zuk, \emph{Spectral analysis of transition operators, automata groups and translation in {BBS}}, Commun. Math. Phys. \textbf{350} (2017), 205--229.


\bibitem{Kent}
J.~Kent, \emph{Time-reversible diffusions}, Adv. in Appl. Probab. \textbf{10}
  (1978), no.~4, 819--835.

\bibitem{KLO}
T.~Komorowski, C.~Landim, and S.~Olla, \emph{Fluctuations in {M}arkov
  processes}, Grundlehren der Mathematischen Wissenschaften [Fundamental
  Principles of Mathematical Sciences], vol. 345, Springer, Heidelberg, 2012,
  Time symmetry and martingale approximation.

\bibitem{KZ}
T.~Konstantopoulos and M.~Zazanis, \emph{A discrete-time proof of {N}eveu's
  exchange formula}, J. Appl. Probab. \textbf{32} (1995), no.~4, 917--921.

\bibitem{KdV}
D.~J. Korteweg and G.~de~Vries, \emph{On the change of form of long waves
  advancing in a rectangular canal, and on a new type of long stationary
  waves}, Philos. Mag. (5) \textbf{39} (1895), no.~240, 422--443.

\bibitem{LPW}
D.~A. Levin, Y.~Peres, and E.~L. Wilmer, \emph{Markov chains and mixing times},
  American Mathematical Society, Providence, RI, 2009, With a chapter by James
  G. Propp and David B. Wilson.

\bibitem{Lev}
L.~Levine, H.~Lyu, and J.~Pike, \emph{Phase transition in a random soliton
  cellular automaton}, preprint appears at arXiv:1706.05621, 2017.

\bibitem{Levy}
P.~L\'evy, \emph{Processus stochastiques et mouvement brownien}, Suivi d'une
  note de M. Lo\`eve. Deuxi\`eme \'edition revue et augment\'ee,
  Gauthier-Villars \& Cie, Paris, 1965.

\bibitem{Liggett}
T.~M. Liggett, \emph{Interacting particle systems}, Grundlehren der
  Mathematischen Wissenschaften [Fundamental Principles of Mathematical
  Sciences], vol. 276, Springer-Verlag, New York, 1985.

\bibitem{LM}
G.~L. Litvinov and V.~P. Maslov, \emph{The correspondence principle for
  idempotent calculus and some computer applications}, Idempotency ({B}ristol,
  1994), Publ. Newton Inst., vol.~11, Cambridge Univ. Press, Cambridge, 1998,
  pp.~420--443.

\bibitem{MaT}
J.~Mada and T.~Tokihiro, \emph{Correlation functions for a periodic box-ball
  system}, J. Phys. A \textbf{43} (2010), no.~13, 135205.


\bibitem{MT:udToda}
K.~Maeda and S.~Tsujimoto, \emph{Box-ball systems related to the nonautonomous ultradiscrete {T}oda equation on the finite lattice}, JSIAM Lett. \textbf{2} (2010), 95--98.

\bibitem{NTT:BBS-sorting}
A.~Nagai, D.~Takahashi and T.~Tokihiro, \emph{Soliton cellular automaton, {T}oda molecule equation and sorting algorithm}, Phys. Lett. A \textbf{255} (1999), 265--271.

\bibitem{OC}
N.~O'Connell, \emph{Directed polymers and the quantum Toda lattice},
  Ann. Probab. \textbf{40} (2012), no.~2, 437--458.

\bibitem{OCslides}
N.~O'Connell, \emph{From Pitman’s $2M-X$ theorem to random polymers and integrable systems}, slides from Stochastic Processes and their Applications, Boulder, Colorado, 2013. Available at: www.maths.ucd.ie/~noconnell/doob.pdf.

\bibitem{OCY}
N.~O'Connell and M.~Yor, \emph{Brownian analogues of {B}urke's theorem},
  Stochastic Process. Appl. \textbf{96} (2001), no.~2, 285--304.

\bibitem{Otter}
R.~Otter, \emph{The multiplicative process}, Ann. Math. Statist. \textbf{20}
  (1949), no.~2, 206--224.

\bibitem{Peskir}
G.~Peskir, \emph{On reflecting {B}rownian motion with drift}, Proceedings of
  the 37th {ISCIE} {I}nternational {S}ymposium on {S}tochastic {S}ystems
  {T}heory and its {A}pplications, Inst. Syst. Control Inform. Engrs. (ISCIE),
  Kyoto, 2006, pp.~1--5.

\bibitem{Pitman}
J.~W. Pitman, \emph{One-dimensional {B}rownian motion and the three-dimensional
  {B}essel process}, Advances in Appl. Probability \textbf{7} (1975), no.~3,
  511--526.

\bibitem{RY}
D.~Revuz and M.~Yor, \emph{Continuous martingales and {B}rownian motion}, third
  ed., Grundlehren der Mathematischen Wissenschaften [Fundamental Principles of
  Mathematical Sciences], vol. 293, Springer-Verlag, Berlin, 1999.

\bibitem{RP}
L.~C.~G. Rogers and J.~W. Pitman, \emph{Markov functions}, Ann. Probab.
  \textbf{9} (1981), no.~4, 573--582.

\bibitem{Taka}
D.~Takahashi, \emph{On a fully discrete soliton system}, Nonlinear evolution
  equations and dynamical systems ({B}aia {V}erde, 1991), World Sci. Publ.,
  River Edge, NJ, 1992, pp.~245--249.

\bibitem{TakaMatsu}
D.~Takahashi and J.~Matsukidaira, \emph{Box and ball system with a carrier and
  ultradiscrete modified {K}d{V} equation}, J. Phys. A \textbf{30} (1997),
  no.~21, L733--L739.

\bibitem{takahashi1990}
D.~Takahashi and J.~Satsuma, \emph{A soliton cellular automaton}, J. Phys. Soc.
  Japan \textbf{59} (1990), 3514--3519.

\bibitem{TS1991}
\bysame, \emph{On cellular automata as a simple soliton system}, Transactions
  of the Japan Society for Industrial and Applied Mathematics \textbf{1}
  (1991), no.~1, 41--60.

\bibitem{T}
T.~Tokihiro, \emph{Ultradiscrete systems (cellular automata)}, Discrete
  integrable systems, Lecture Notes in Phys., vol. 644, Springer, Berlin, 2004,
  pp.~383--424.

\bibitem{TT}
\bysame, \emph{The mathematics of box-ball systems}, Asakura Shoten, 2010.

\bibitem{TTMS}
T.~Tokihiro, D.~Takahashi, J.~Matsukidaira, and J.~Satsuma, \emph{From soliton
  equations to integrable cellular automata through a limiting procedure},
  Phys. Rev. Lett. \textbf{76} (1996), no.~18, 3247--3250.

\bibitem{HT:ukdv}
S.~Tsujimoto and R.~Hirota, \emph{Ultradiscrete {K}d{V} equation}, J. Phys.
  Soc. Japan \textbf{67} (1998), 1809--1810.

\bibitem{Varad}
V.~S. Varadarajan, \emph{Groups of automorphisms of {B}orel spaces}, Trans.
  Amer. Math. Soc. \textbf{109} (1963), 191--220.

\bibitem{Walters}
P.~Walters, \emph{An introduction to ergodic theory}, Graduate Texts in
  Mathematics, vol.~79, Springer-Verlag, New York-Berlin, 1982.

\bibitem{YYT}
D.~Yoshihara, F.~Yura, and T.~Tokihiro, \emph{Fundamental cycle of a periodic
  box-ball system}, J. Phys. A \textbf{36} (2003), no.~1, 99--121.

\bibitem{YT}
F.~Yura and T.~Tokihiro, \emph{On a periodic soliton cellular automaton}, J.
  Phys. A \textbf{35} (2002), no.~16, 3787--3801.

\end{thebibliography}
\end{document}